\documentclass{amsart}%

\usepackage[margin=1.1in]{geometry}

\usepackage{subfiles}
\usepackage{graphicx}
\usepackage{multicol,multirow}
\usepackage{amsmath,amssymb,amsfonts}
\usepackage{mathrsfs}
\usepackage{amsthm,thmtools}
\usepackage{rotating}
\usepackage{appendix}
\usepackage[numbers]{natbib}
\usepackage{ifpdf}
\usepackage[T1]{fontenc}
\usepackage{newtxtext}
\usepackage{newtxmath}
\usepackage{textcomp}
\usepackage[british]{babel}
\usepackage[utf8]{inputenc}
\usepackage[all,cmtip]{xy}
\usepackage{xy}
\usepackage{soul}

\usepackage{tikz-cd}
\usetikzlibrary{cd}
\usepackage{varioref}
\usepackage{hyperref}
\usepackage{mathtools}
\usepackage{xcolor}
\usepackage[old]{old-arrows}
\usepackage{qtree,tensor}
\usepackage{comment}
\usepackage{enumerate}
\usepackage[normalem]{ulem}
\def\ddefloop#1{\ifx\ddefloop#1\else\ddef{#1}\expandafter\ddefloop\fi}
\usepackage{xpatch}
 
\xpatchcmd{\paragraph}{\normalfont}{{\normalfont\bfseries}}{}{}
\tikzset{
  closed/.style = {decoration = {markings, mark = at position 0.5 with { \node[transform shape, xscale = .8, yscale=.4] {/}; } }, postaction = {decorate} },
  open/.style = {decoration = {markings, mark = at position 0.5 with { \node[transform shape, scale = .7] {$\circ$}; } }, postaction = {decorate} }
}
\usepackage{caption,subcaption}
\usepackage{float}
\usepackage{cleveref}

\def\ddef#1{\expandafter\def\csname bb#1\endcsname{\ensuremath{\mathbb{#1}}}}
\ddefloop ABCDEFGHIJKLMNOPQRSTUVWXYZ\ddefloop

\def\ddef#1{\expandafter\def\csname ff#1\endcsname{\ensuremath{\mathfrak{#1}}}}
\ddefloop ABCDEFGHIJKLMNOPQRSTUVWXYZ\ddefloop

\def\ddef#1{\expandafter\def\csname cc#1\endcsname{\ensuremath{\mathcal{#1}}}}
\ddefloop ABCDEFGHIJKLMNOPQRSTUVWXYZ\ddefloop

\newcommand{\pic}{\mathrm{Pic}\,}

\newcommand{\pt}{\mathrm{pt}}

\newcommand{\Spec}{\mathrm{Spec}}
\newcommand{\Symm}{\mathrm{Sym}\,}

\newcommand{\rk}{\mathrm{rk}}
\newcommand{\grk}{\mathrm{grk}}
\newcommand{\mrk}{\mathrm{mrk}}
\newcommand{\Bl}{\mathrm{Bl}}

\newcommand{\tf}{\mathrm{tf}}

\newcommand{\mori}{{\overline{\mathrm{NE}}}}

\newcommand{\eff}{{\mathrm{eff}}}

\newcommand{\Pic}{\mathfrak{Pic}}
\newcommand{\codim}{\mathrm{codim}}
\newcommand{\cl}{\mathrm{cl}}
\newcommand{\GG}{\pmb{G}}
\newcommand{\HH}{\pmb{H}}
\DeclareMathOperator*{\fprodM}{\sideset{}{_{\ffM_{g,n}}}\prod}
\DeclareMathOperator*{\fprodGamma}{\sideset{}{_{\Gamma_{0,n}}}\prod}
\newcommand{\pdim}{\mathrm{pdim}}
\newcommand{\valency}{\mathrm{val}}
\newcommand{\base}{\mathrm{base}}
\newcommand{\Base}{\mathrm{Base}}

\newcounter{mainresults}
\newtheorem{theorem}[subsubsection]{Theorem}
\newtheorem{corollary}[subsubsection]{Corollary}

\newtheorem{lemma}[subsubsection]{Lemma}
\newtheorem{proposition}[subsubsection]{Proposition}
\newtheorem{maintheorem}[mainresults]{Theorem}

\theoremstyle{definition}
\newtheorem{definition}[subsubsection]{Definition}
\newtheorem{construction}[subsubsection]{Construction}

\newtheorem{remark}[subsubsection]{Remark}
\newtheorem{example}[subsubsection]{Example}

\newtheorem{notation}[subsubsection]{Notation}
\newtheorem{assumption}[subsubsection]{Assumption}

\Crefname{maintheorem}{Theorem}{Theorems}
\Crefname{maindefinition}{Definition}{Definition}
\Crefname{notation}{Notation}{Notation}
\Crefname{construction}{Construction}{Construction}
\Crefname{assumption}{Assumption}{Assumption}

\begin{document}

\title[Irreducible components of $\overline{\ccM}_{0,n}(X,\beta)$]{Irreducible components of moduli spaces of maps to smooth projective toric varieties in genus 0}

\author{Alberto Cobos Rabano}
\address{Department of Mathematics, KU Leuven,
	Celestijnenlaan 200B, box 2400,
	BE-3001 Leuven, Belgium}
\email{alberto.cobos-rabano@kuleuven.be}

\author{Etienne Mann}
\address{LAREMA, Universit\'e d'Angers, CNRS,
	F-49000 Angers, France}
\email{etienne.mann@univ-angers.fr}


\begin{abstract}
	We study the irreducible components of the moduli space $\overline{\mathcal{M}}_{0,n}(X,\beta)$ for a smooth projective toric variety $X$. Our approach is based on the study of the irreducible components of an abelian cone over a smooth Noetherian Artin stack. We give concrete applications by computing the irreducible decomposition of $\overline{\mathcal{M}}_{0,n}(\Bl_{\pt} \mathbb{P}^2,\beta)$ for small values of $\beta$. 
\end{abstract}

\maketitle

\section{Introduction}

\subsection{Results}

Let $X$ be a smooth projective toric variety and $\beta$ a non-zero effective curve class on $X$. We study the irreducible components of the moduli space $\overline{\ccM}_{0,n}(X,\beta)$ of genus-0 $n$-marked stable maps to $X$ of class $\beta$. 

For that, we consider the locus $\ccM_{\GG}$ with fixed stable decorated marked dual tree $\GG\in\Gamma^{st}_{0,n}(X,\beta)$ (\Cref{def:decorated_marked_tree}), whose closure we denote by $\overline{\ccM}_{\GG}$. We show that if $\beta$ intersects all the toric boundary divisors non-negatively, then all irreducible components of $\overline{\ccM}_{0,n}(X,\beta)$ are of the form $\overline{\ccM}_{\GG}$, and we give several numerical criteria based on their dimension to single out the actual components (\Cref{thm:components_maps_collection}). This result also applies to the moduli space of stable toric quasimaps $\ccQ_{0,n}(X,\beta)$ with minor changes, as explained in \Cref{sec:components_qmaps}.

In practice, the set $\Gamma_{0,n}^{st}(X,\beta)$ often contains more elements than the number of irreducible components. In concrete applications, we find it useful to identify combinatorial conditions that a decorated tree $\GG \in \Gamma_{0,n}^{st}(X,\beta)$ must satisfy if $\overline{\ccM}_{\GG}$ is an irreducible component. To implement this idea for $X=\Bl_{\pt}\bbP^2$, we define the notion of tentative component (\Cref{def:tentative_component}) and prove that every irreducible component is a tentative component (for all relevant $\beta$) in \Cref{thm:criteria_tentative_component}.
With this approach, we describe explicitly all the irreducible components of the moduli space $\overline{\ccM}_{0,n}(\Bl_{\pt}\bbP^2,\beta)$ in low degree. 

To state our main result, we denote by $\ell, s,e\in A_1(\Bl_{\pt}\bbP^2)$ the curve classes of the total transform of a line in $\bbP^2$, of the strict transform of a line in $\bbP^2$ and of the exceptional divisor, respectively. We consider curve classes $\beta = as+be$ with $a=0$ and $b> 0$, or with $a\geq b\geq 0$ (see \Cref{fig:Mori_cone}). We further assume that $\beta \neq 0$ or that $n\geq 3$. These conditions ensure that the main irreducible component, which is the closure of the locus where the source curve is smooth, is non-empty.

\begin{maintheorem}(\Cref{thm:irreducibility_BlP2,thm:classification_2l,thm:classification_3l,cor:classification_3l_count,thm:classification_4l})
\label{mainthm:summary_BlP2}
    For $\beta = as + be$ as above,
    \begin{enumerate}
        \item $\overline{\ccM}_{0,n}(\Bl_{\pt}\bbP^2,\beta)$ is irreducible if and only if $a = 0$ or $b=0$ or $b=1$.
        \item For $b=2$, $\overline{\ccM}_{0,n}(\Bl_{\pt}\bbP^2,\beta)$ has $2^n + 1$ irreducible components.
        \item For $b=3$, $\overline{\ccM}_{0,n}(\Bl_{\pt}\bbP^2,\beta)$ has $1 + 2^{n+1} + \left\lceil \frac{3^n(\alpha + 2)}{2} \right\rceil$ irreducible components.
        \item For $b=4$, we identify all the irreducible components of  $\overline{\ccM}_{0,n}(\Bl_{\pt}\bbP^2,\beta)$.
    \end{enumerate}
\end{maintheorem}
All the curve classes covered in \Cref{mainthm:summary_BlP2} with $a\neq 0$ satisfy that each tentative component $\GG\in\Gamma_{0,n}(\Bl_{\pt}\bbP^2,\beta)$ is an actual irreducible component. The smallest degree for which this fact is unclear is $\beta = 5\ell$, (\Cref{ex:component_or_not}). On the other hand, not every irreducible component of $\ccQ_{0,n}(X,\beta)$ arises from a tentative component (\Cref{rmk:qmaps_counterexample_empty_mic} and \Cref{ex:qmaps_non_mic_component}).

\subsection{Strategy}

In Notations \ref{notation:product_Pics} and \ref{notation:product_Pics_stable}, we define $\ffP^{st} = \Pic^{st}_{0,n,X,\beta,\sigma}$, which is an open substack in a product of Picard stacks.
Following \cite{Chang-Li-maps-with-fields}, there is an open embedding $\overline{\ccM}_{0,n}(X,\beta) \hookrightarrow S$ where $\pi\colon S=\Spec_{\ffP^{st}} \Symm \ffF \to \ffP^{st}$ (Definition \eqref{eq:definition_ambient_cone}) is the abelian cone associated to a coherent sheaf $\ffF$ on $\ffP^{st}$. 
Using this open embedding, we can understand the irreducible components of $\overline{\ccM}_{0,n}(X,\beta)$ if we understand those of a general abelian cone over a smooth Noetherian Artin stack. We further reduce this problem from stacks to schemes in \Cref{prop:results_for_stacks}.

We consider an integral Noetherian scheme $B$ and a coherent sheaf $\ccF$ on $B$ of generic rank $f$. We let $\{D_{i,j}\}$, with $(i,j)\in\Delta$, be a finite stratification of $B$ by locally closed irreducible subschemes such that $\rk(\ccF\mid_{D_{i,j}}) = f+i$. Under these conditions, we show that the following structural result for the abelian cone $\pi\colon C_B(\ccF) \to B$.

\begin{maintheorem}(\Cref{corollary:structure_components_CBF})\label{mainthm:structure_components_CBF}
    The irreducible components of the abelian cone $C_B(\ccF)$ are
    \begin{equation}
        \left\{\overline{\pi^{-1}(D_{i,j})}\colon (i,j)\in \Delta,\ \overline{\pi^{-1}(D_{i,j})} \not\subseteq \overline{\pi^{-1}(D_{k,\ell})}\ \forall\, (k,\ell)\in\Delta\setminus\{(i,j)\}\right\}
    \end{equation}
\end{maintheorem}

In other words, the irreducible components are precisely the maximal elements (with respect to inclusion) among all candidate components $\overline{\pi^{-1}(D_{i,j})}$. In particular, $C_B(\ccF)$ has a main irreducible component $C_B^\circ (\ccF)$ which corresponds to the unique $D_{0,0}$ which is dense in $B$ (see \Cref{def:main_component}). All other irreducible components, called extra irreducible components, are supported on  closed subschemes strictly contained in $B$.

To understand inclusions among our candidate components, we look at their dimensions. These are easy to compute, as $\pi^{-1}(D_{i,j})$ is a vector bundle over the irreducible scheme $D_{i,j}$ for all $(i,j)\in\Delta$. Motivated by the applications, we work instead with the integers
\[
    d_{i,j}\coloneqq \dim(\pi^{-1}(D_{i,j})) - \dim(C_B^\circ(\ccF))= i- \codim_B(D_{i,j}),
\]
introduced in \Cref{notation:dij}. We give three useful numerical criteria (\Cref{prop:numerical_sufficient_criterion_component,prop:same_i_implies_inclusion,prop:containment_main_component}), combined in the following result. See \Cref{def:local_pdim} for the definition of local projective dimension at most 1.

\begin{maintheorem}\label{mainthm:main_numerical_results}
    Let $(i,j) \in \Delta$. Then
    \begin{enumerate}
        \item If $d_{i,j}\geq d_{k,\ell}$ for every $(k,\ell) \in \Delta$ such that $D_{i,j}\subset \overline{D}_{k,\ell}$, then $\overline{\pi^{-1}(D_{i,j})}$ is an irreducible component of $C_B(\ccF)$.
        \item If there exists $(i,\ell)\in \Delta$ such that $D_{i,j}\subseteq \overline{D}_{i,\ell}$ then $\overline{\pi^{-1}(D_{i,j})}\subseteq \overline{\pi^{-1}(D_{i,\ell})}$.
        \item If $\ccF$ has local projective dimension at most 1, then $d_{i,j}\geq d_{0,0} = 0$ implies that $\overline{\pi^{-1}(D_{i,j})}\not\subset C_B^\circ (\ccF)$, and $d_{i,j}< d_{0,0} = 0$ implies that $\overline{\pi^{-1}(D_{i,j})}$ is not an irreducible component.
    \end{enumerate}
\end{maintheorem}

The last part of \Cref{mainthm:main_numerical_results} relies on the fact that if $\pdim(\ccF)\leq 1$ then $\dim(C_B^\circ(\ccF))$ is a lower bound for the dimension of any component of $C_B(\ccF)$ (\Cref{thm:dimension_main_is_lower_bound}).

For the applications to $\overline{\ccM}_{0,n}(X,\beta)$, we have a natural stratification of $\ffP^{st}$ by substacks $\ffP_{\GG}$ with fixed (decorated) dual graph $\GG$ (see \Cref{def:decorated_marked_tree} and \Cref{constr:decorated_dual_graph}). This stratification satisfies the required conditions by \Cref{lem:stratification_prod_Pic_satisfies_assumptions}, but the assumption on the projective dimension is only guaranteed for curve classes $\beta$ which intersect all the toric boundary divisors non-negatively. 

We denote $S_{\GG}\coloneqq \pi^{-1}(\ffP^{st}_{\GG})$ and $\ccM_{\GG}\coloneqq \overline{\ccM}_{0,n}(X,\beta)\cap S_{\GG}$. 
As an immediate corollary of \Cref{mainthm:structure_components_CBF} we see that the irreducible components of $\overline{\ccM}_{0,n}(X,\beta)$ are precisely the substacks $\overline{\ccM}_{\GG}$ such that $\ccM_{\GG}\neq \emptyset$ and $\overline{S}_{\GG}$ is an irreducible component of the ambient abelian cone $S$. The non-emptiness of $\ccM_{\GG}$ is not automatic (see \Cref{rmk:condition_MG_nonempty}) and, in general, depends on the geometry of curves inside $X$. We also prove that if the main component is non-empty, then its dimension is the lowest among the dimensions of all the irreducible components of $\overline{\ccM}_{0,n}(X,\beta)$. These results and the analogue of the numerical criteria in \Cref{mainthm:main_numerical_results} for $\overline{\ccM}_{0,n}(X,\beta)$ are collected in \Cref{thm:components_maps_collection}. 

The combination of Theorems \ref{mainthm:structure_components_CBF} and \ref{mainthm:main_numerical_results} is enough to compute all the examples in \Cref{mainthm:summary_BlP2}. Concretely, the integers $d_{i,j}$ (which we denote $d_{\GG}$, see \Cref{notation:def_iG_dG}) are easy to compute in practice (see Definitions \eqref{eq:definition_iG_prod} and \eqref{eq:definition_dG_prod}), and the inclusion $\ffP_{\GG}\subseteq \overline{\ffP}_{\GG'}$ holds if and only if $\GG'$ can be obtained from $\GG$ by a sequence of edge-contractions (\Cref{rmk:inclusion_is_edge_contraction}). Furthermore, by the argument in \Cref{subsec:reduction_Le} it is enough to look at the line bundles corresponding to non-nef toric boundary divisors in $X$. For the particular case $X= \Bl_{\pt}\bbP^2$, there is only one non-nef toric boundary divisor.

\subsection{Related and future work}

For $g\geq 1$, the moduli space $\overline{\ccM}_{g,n}(\bbP^r,d)$ has several irreducible components. If $d>2g-2$, the closure of the locus of stable maps with smooth domain curve is an irreducible component, the main component. The irreducible components of $\overline{\ccM}_{1,n}(\bbP^r,d)$ are well-understood, see \cite{Vakil_thesis,Vakil-Zinger-desingu-main-compo-}. A detailed example, $\overline{\ccM}_{1,0}(\bbP^2,3)$, appeared in \cite{VZ} and later in \cite{Viscardi}.
In genus two, all the irreducible components of $\overline{\ccM}_{2,0}(\bbP^2,4)$ have been listed in \cite{Battistella_Carocci} and the main irreducible component has been described, as an application of the modular desingularization of $\overline{\ccM}_{2,n}(\bbP^r,d)$ in \cite{Battistella_Carocci_g2}.
For a smooth projective toric target $X$, the space $\overline{\ccM}_{g,n}(X,\beta)$ can already be reducible for $g=0$, see Theorems \ref{mainthm:summary_BlP2}.
It would be interesting to have a detailed description of the main component in the examples described in \Cref{mainthm:summary_BlP2}. We expect that the theory of limit linear series could be useful for $X=\Bl_{\pt}\bbP^2$ (and more generally, for $X$ with a single non-nef toric boundary divisor), thanks to the reduction to non-nef toric boundary divisors in \Cref{subsec:reduction_Le}.

The reducibility of $\overline{\ccM}_{g,n}(\bbP^r,d)$ in positive genus affects the structure of Gromov-Witten invariants of projective space.  Numerous works \cite{Zinger-genus-one-pseudo-holomorphic,Vakil-Zinger-desingu-main-compo-, Hu_Li_g1,ranganathan2019moduli, Hu_Li_Niu_g2, Battistella_Carocci_g2, niu1, niu2,CRMMP} define reduced Gromov-Witten invariants, which can be thought of as the contribution of the main component to the Gromov-Witten invariants. It is expected \cite{Zinger-standard-vs-reduced,Zinger-reduced-g1-CY,Hu_Li_derived_resolution} that the remaining irreducible components contribute in the form of lower genus invariants, leading to recursive formulas for reduced Gromov-Witten invariants. Such formulas are known in certain instances in genus 1 and 2 \cite{Zinger-reduced-g1-CY,li-zinger,Lee_Oh_part1,Lee_Oh_part2,Lee_Li_Oh}.
We would find it interesting to have a similar splitting of invariants among components for $\overline{\ccM}_{0,n}(X,\beta)$. For example, it may give a geometric proof of the Gromov-Witten/quasimap correspondence \cite{stable_quotients,clader2017higher,ciocan2017higher, ciocan2020quasimap,Zhou,KMMP,Cristina-virt-push,Cristina-stable-maps} when combined with the contraction morphism between maps and quasimaps in \cite{Cobos}. Although our general framework (\Cref{thm:components_maps_collection}) applies to quasimaps by \Cref{sec:components_qmaps}, this strategy requires further investigation on the irreducible components of $\ccQ_{0,n}(X,\beta)$ in light of \Cref{rmk:qmaps_counterexample_empty_mic} and \Cref{ex:qmaps_non_mic_component}.

One key ingredient used in \cite{Lee_Oh_part1,Lee_Oh_part2,Lee_Li_Oh} to study the reduced Gromov-Witten invariants of $\overline{\ccM}_{g,n}(\bbP^r,d)$ in low genus is the local description of the moduli space, obtained in \cite{Hu_Li_g1} for genus 1 and in \cite{Hu_Li_Niu_g2} for genus 2. We would find it desirable to have a similar description for genus-0 stable maps to a smooth projective toric variety $X$. In fact, this paper is motivated by a talk by Hu on the series of papers \cite{Hu_Li_g1,Hu_Li_Niu_g2,Hu_Li_derived_resolution,niu1,niu2} during the program \textit{Recent developments in higher genus curve counting} at the Simons Center for Geometry and Physics.

\subsection{Structure of this paper}

In \Cref{sec:irreducible_comps_cone} we study the irreducible components of an abelian cone over an integral Noetherian scheme. We describe the structure of the irreducible components in terms of a stratification $\{D_{i,j}\}$ of the base in \Cref{corollary:structure_components_CBF}. After that, we give certain numerical criteria to determine which strata $D_{i,j}$ give rise to a component in \Cref{subsec:numerical_criteria}. The main results are \Cref{prop:numerical_sufficient_criterion_component,prop:same_i_implies_inclusion,prop:containment_main_component}. We present some examples in \Cref{subsec:examples_abelian_cones} and conclude with the generalization to a coherent sheaf on a smooth Artin stack, \Cref{prop:results_for_stacks}.

In \Cref{sec:background} we introduce the moduli spaces of prestable curves and the Picard stack and recall how they can be stratified using trees with extra combinatorial data. In \Cref{subsec:decorated_graphs_prod_Pic} we introduce a product of Picard stacks, \Cref{notation:product_Pics}, which will be useful in the study of stable maps. We then generalize the stratification by trees with extra combinatorial data to this space in \Cref{cor:stratification_prod_Pic}.

The applications of \Cref{sec:irreducible_comps_cone} in this paper require to embed the moduli space of stable maps in an abelian cone using the toric boundary divisors. Before that, we explore the easier case of a single line bundle $S(\ffL)$ in \Cref{sec:abelian_cones_over_Pic}. This is the abelian cone associated to the sheaf $\ccF = R^1\pi_\ast (\ffL^\vee \otimes \omega_\pi)$ on the universal curves over the (stable part of) the Picard stack, see Definition \eqref{eq:definition_SL}. The most important result of this section is \Cref{prop:R1_has_pdim_1}, where we check that $\pdim(\ccF)\leq 1$.

In \Cref{sec:application_stable_maps} we explain how the results of \Cref{sec:irreducible_comps_cone} can be used to determine the irreducible components of stable maps and quasimaps. We recall how to embed the moduli space of stable maps $\overline{\ccM}_{0,n}(X,\beta)$ in an abelian cone $S$ (see Embedding \eqref{eq:open_embedding_stable_maps}). This embedding allows us to characterize the irreducible components of stable maps in terms of those of $S$, and to give numerical criteria to characterize the dual graph corresponding to irreducible components, see \Cref{thm:components_maps_collection}. In particular, we prove that the main component  has the lowest dimension among all irreducible components if it is non-empty and $\beta$ intersects all toric boundary divisors non-negatively. This results extend to quasimaps with minor changes (\Cref{sec:components_qmaps}).

Finally, in \Cref{sec:Bl_P2} we discuss certain concrete examples. We look at $X=\Bl_{\pt}\bbP^2$, where the problem reduces to understanding the abelian cone of a single line bundle by \Cref{subsec:reduction_Le}. We first identify the curve classes $\beta$ for which $\overline{\ccM}_{0,n}(\Bl_{\pt}\bbP^2,\beta)$ is irreducible in \Cref{thm:irreducibility_BlP2}. After that, we identify certain combinatorial conditions that a decorated tree must satisfy to give rise to a component, combined in the notion of tentative component (\Cref{def:tentative_component}). We prove that every irreducible component arises from a tentative component if $\beta$ intersects all the boundary divisors non-negatively, and upgrade \Cref{thm:components_maps_collection} to \Cref{thm:criteria_tentative_component}, so we only need to deal with tentative components. Finally, in \Cref{subsec:general_case} we prove the rest of \Cref{mainthm:summary_BlP2}.

\subsection{Acknowledgements}

We are very thankful to Yi Hu for private conversations at the early stages of this project, to Cristina Manolache and Renata Picciotto for informal discussions. We strongly thank an anonymous referee for their very helpful comments.

\section{Irreducible components of abelian cones} \label{sec:irreducible_comps_cone}

\subsection{Preliminaries: functoriality and Fitting ideals}\label{subsec:preliminaries}

Let $B$ be a  Noetherian scheme and $\ccF$ be a coherent sheaf on $B$. The abelian cone associated to $\ccF$ is the scheme
\[
    C_B(\ccF) = \Spec_{B} \Symm (\ccF).
\]
We denote by $\pi:C_B(\ccF) \to B$ the natural projection. Since $\Symm$ (\cite[\href{https://stacks.math.columbia.edu/tag/01CI}{Tag 01CI}]{stacks-project}) and $\Spec$ are compatible with pullback, so are abelian cones, in the following sense: for any morphism $f\colon B'\to B$ it holds that
\[
    B'\times_B C_B(\ccF) = C_{B'}(f^*\ccF).
\]

Our goal is to understand the irreducible components of an abelian cone $C_B(\ccF)$.

\begin{remark}
    The irreducible components of $C_B(\ccF)$ and of its reduction $C_{B^{\mathrm{red}}}(\ccF\mid_{B^{\mathrm{red}}})$ agree, thus we may assume that $B$ is reduced. Furthermore, the support of each irreducible component $Z$ of $C_B(\ccF)$ is contained on a single irreducible component $V$ of $B$, thus it is also an irreducible component of
    \[
        V\times_B C_B(\ccF) = C_{V}(\ccF\mid_{V})
    \]
    since abelian cones are compatible with pullback. Therefore, to study the irreducible components of $C_B(\ccF)$ we may assume without loss of generality that $B$ is integral.  
\end{remark}

\begin{notation}
    Let $B$ be an integral Noetherian scheme and $\ccF$ a coherent sheaf on $B$. We denote by $\grk_B(\ccF)$ the generic rank of $\ccF$ on $B$, that is, the rank of $\ccF$ at the generic point of $B$. We also consider the maximal rank of $\ccF$ on $B$ 
    \[
        \mrk_B(\ccF) = \max_{p\in B}\{\rk(\ccF\mid_p)\},
    \]
    that is, the maximum rank of $\ccF$ on any closed point of $B$. Since $B$ is Noetherian, $\mrk_B(\ccF)$ is finite.
\end{notation}

The irreducible components of $C_B(\ccF)$ are controlled by the rank of $\ccF$, therefore it is useful to recall the Fitting ideals of $\ccF$ and the corresponding stratification of $B$. See \cite[Section 20.2]{Eisenbud} or \cite[0C3C]{stacks-project}. 

\begin{definition}
    Let $M$ be a finitely presented $R$-module. Let $F\xrightarrow{\varphi} G\to M\to 0$ be a presentation with $F$ and $G$ free modules and $\rk(G)=g$. Given $-1\leq i < \infty$, the $i$-th Fitting ideal $F_i(M)$ of $M$ is the ideal generated by all $(g-i)\times (g-i)$-minors of the matrix associated to $\varphi$ after fixing basis of $F$ and $G$. We use the convention that $F_i(M) = R$ if $g-i\leq 0$ and $F_{-1}(M)=0$.
\end{definition}

Fitting ideals commute with base change and their definition is independent of the chosen presentation. In particular, the local definition can be glued to define Fitting ideals for any finitely presented sheaf $\ccG$ on a scheme $X$. By construction, there are inclusions
\[
    0=F_{-1}(\ccG)\subset F_0(\ccG)\subset\dots\subset F_n(\ccG)\subset\dots\subset \ccO_X,
\]
which stabilize if $X$ is Noetherian. We denote by $Z_i(\ccG)$ the closed subscheme cut out by the ideal sheaf $F_i(\ccG)$. Then, the sheaf $\ccG$ is locally free of rank $n$ on the locally closed subscheme $Z_{n-1}(\ccG)\setminus Z_n(\ccG)$.

\subsection{Structure of the irreducible components}\label{subsec:structure_components_cone}

Our goal is to describe the irreducible components of an abelian cone $C_B(\ccF)$ with $B$ an integral Noetherian scheme and $\ccF$ a coherent sheaf on $B$. We will explain why such an abelian cone always has a main component $C^\circ_B(\ccF)$ (\Cref{def:main_component}), which is the preimage of the dense open subset $B_0$ of $B$ where $\ccF$ has minimal rank. One of the main results in the literature about irreducible components of cones is \cite[Proposition 2.5]{Axelsson-Magnusson}, which implies for abelian cones that 
\[
    C_B^\circ(\ccF) = \Spec_{B}((\Symm(\ccF))^{\tf}),
\]
where $\tf$ denotes the torsion-free part of the $\ccO_B$-module $\Symm(\ccF)$. Note that the operators $\tf$ and $\Symm$ do not commute in general, so $C_B^\circ(\ccF)$ need not be an abelian cone (see \Cref{ex:torsion_symm_alg}).
By \cite[B.5.3]{Fulton}, each irreducible component of an abelian cone must be a cone over its support (which is an irreducible subscheme of $B$); but \Cref{ex:torsion_symm_alg} shows that components may not be abelian cones over their supports. The result of Axelsson and Magnusson suggests that it would be interesting to relate properties of the sheaf $\ccF$ with properties of its symmetric algebra $\Symm(\ccF)$ in order to understand the irreducible components of an abelian cone.

Another relevant result in the literature is \cite[Lemma 2.3]{Starr}. Working under the assumption that $\ccF$ has local projective dimension at most 1 (see \Cref{def:local_pdim}) and that $B$ is $\bbQ$-Gorenstein, Starr relates the irreducibility and regularity in codimension 1 of $C_B(\ccF)$ with certain inequalities about the codimension of the loci $B_i$. 

We propose a different approach to study the components of $C_B(\ccF)$: fixing a finite stratification $\{D_{i,j}\}$ by locally closed irreducible subschemes of $B$ on which $\rk(\ccF)$ is constant (\Cref{assumption:cone+stratification}). The preimage of each stratum $D_{i,j}$ is a vector bundle in $C_B(\ccF)$, whose closure is a candidate component. In fact, we shall start by proving that the components of $C_B(\ccF)$ are precisely the maximal (with respect to inclusion) such candidate components (\Cref{corollary:structure_components_CBF}). After that, we will discuss numerical criteria, based on the dimension of our candidate components, to determine which ones are actual components.

The fact that the dimension of each candidate component $\pi^{-1}(D_{i,j})$ plays an important role is clear from our exposition. This dimension is the sum of the dimension of the stratum $D_{i,j}$ and the rank of $\ccF$ on it, see \Cref{eq:dimension_preimage_Dij}.
However, the original reason why we believed that the irreducible components of $C_B(\ccF)$ should be related to the rank and codimension of each stratum is the result of Starr \cite[Lemma 2.3]{Starr} previously mentioned.

\begin{example}
\label{exp:I+I}\label{ex:torsion_symm_alg}
    (\cite[Example 6.1.7]{CRMMP})
      This is an example of a torsion free sheaf $\ccF$ on an integral Noetherian scheme $B$ such that $\Symm(\ccF)$ is not torsion free, thus $\Spec \Symm \ccF$ is not irreducible. Furthermore, the main component is not an abelian cone.
      
    Let $B=\bbA^2=\Spec(\bbC[x,y])$ be the affine plane, let $I= (x,y)$ be the ideal of the origin $0$ and let $M=I\oplus I$. Denote by $R=\bbC[x,y]$. Then $M$ has the following free presentation.
      \begin{equation}\label{eq:I+I}
      \begin{tikzcd}[ampersand replacement=\&,column sep = 2.5cm]
          0 \arrow[r] \& R^2  \arrow[r,"{\begin{psmallmatrix} -y\ & 0 \\ x\ &\ 0 \\ \ 0\ & -y \\ 0\ & \ x  \end{psmallmatrix}}"] \& R^4 \arrow[r,"{\begin{psmallmatrix} x\oplus 0 & y\oplus 0& 0\oplus x& 0 \oplus y \end{psmallmatrix}}"]\& M \arrow[r] \& 0
      \end{tikzcd}
      \end{equation}
    Then we have 
    \[
        \Symm(M) \simeq \bbC[x,y,X_1,Y_1,X_2,Y_2]/(yX_1-xY_1,yX_2-xY_2)
    \] 
    which has torsion as $xy(X_1Y_2-Y_1X_2) = 0$. 
    Note that $B_0=\bbA^2\setminus\{0\}$ and $B_1=\{(0,0)\}$ (See \Cref{rmk:existence_stratifications} below for the definition of $B_i$'s). The abelian cone $C_B(M)=\Spec \Symm(M)$ has two irreducible components. 
      \begin{itemize}
          \item $V(Y_1X_2-X_1Y_2,yX_2-xY_2,yX_1-xY_1) = \overline{\pi^{-1}(B_0)}$, which is the main component $C^\circ_B(M)$, of dimension $4$. 
          \item $V(x,y) = \overline{\pi^{-1}(B_1)}$, of dimension $4$
      \end{itemize}
      In this example the main component is not an abelian cone over $B$ because its first equation is not $\bbC[x,y]$-linear.
\end{example}

To study the irreducible components of an abelian cone $C_B(\ccF)$, we consider a stratification of $B$ with the condition that $\ccF$ has constant rank on each stratum. 

\begin{assumption}\label{assumption:cone+stratification}
    Let $B$ be an integral Noetherian scheme and $\ccF$ a coherent sheaf on $B$. We denote $b=\dim(B)$, $f =\grk_B(\ccF)$ and $m = \mrk_B(\ccF)$.
    Furthermore, we consider a finite stratification
    \[  
        B = \bigsqcup_{(i,j) \in \Delta} D_{i,j}
    \]
    with the following properties: 
    \begin{enumerate}
        \item $D_{i,j}$ is a locally closed irreducible subscheme of $B$, considered with its reduced induced structure,
        \item $\ccF\mid_{D_{i,j}}$ has constant rank and $\rk(\ccF\mid_{D_{i,j}}) = f + i$.
    \end{enumerate}
    We denote $\Delta_k \coloneqq \{(i,j)\colon i=k \}$ for any $k\in\{0,\ldots, m-f\}$, so that $\Delta = \sqcup_{i=0}^m \Delta_i$.
\end{assumption}

\begin{remark}\label{rmk:existence_stratifications}
    Let $B, \ccF$ as in \Cref{assumption:cone+stratification}. For each $i\in \{0,\ldots, m-f\}$, we let  \begin{equation}\label{eq:definition_Bi}
        B_i = B_i(\ccF) \coloneqq Z_{f+i-1}(\ccF) \setminus Z_{f+i}(\ccF).
    \end{equation}
    That is, $B_i$ is the locus where $\ccF$ has rank $f+i$. Then $\{B_i\colon 0\leq i \leq m\}$ is a stratification of $B$ by locally closed subschemes, but they need not be irreducible. 
    
    The stratifications considered in \Cref{assumption:cone+stratification} are precisely the refinements of the previous Fitting stratification which have finitely many strata, all of which are irreducible and locally closed.
\end{remark}

\begin{remark}\label{rmk:D00}
    Under \Cref{assumption:cone+stratification}, we have that $B_0 = \sqcup_{j \in \Delta_0} D_{0,j}$ with $\Delta_0$ finite. Since $B_0$ is irreducible, this implies that there exists a unique $j_0 \in \Delta_0$ such that $D_{0,j_0}$ is dense in $B_0$. We denote such $D_{0,j_0}$ by $D_{0,0}$.
\end{remark}

The use of stratifications of $B$ is justified by the following topological lemma.

\begin{lemma}\label{lem:topological_lemma}
    Let $X$ be a topological space and let $V_1,\ldots, V_n$ be a collection of mutually distinct irreducible closed subsets such that 
    \[
        X = \bigcup_{i = 1}^n V_i.
    \]
    Then the irreducible components of $X$ are
    \begin{equation}\label{eq:set_of_components}
        \{V_i\colon i\in\{1,\ldots, n\},   V_i\not\subseteq V_j \, \forall\,  j\neq i\}.
    \end{equation}
\end{lemma}

\begin{proof}
    Let $Z$ be an irreducible component of $X$. We first prove that $Z=V_i$ for some $i$. We have that
    \[
        Z = \bigcup_{i = 1}^n (Z\cap V_i) = (Z\cap V_1) \cup \left(\bigcup_{i = 2}^n (Z\cap V_i)\right)
    \]
    with $Z\cap V_i$ closed for all $i$. Thus, either $Z\subseteq Z\cap V_1$ or  $Z \subseteq \cup_{i = 2}^n (Z\cap V_i)$. Repeating this argument, we can find $i\in\{1,\ldots, n\}$ such that $Z\subseteq Z\cap V_i$, that is, $Z\subseteq V_i$. This inclusion must be an equality because $Z$ is an irreducible component and $V_i$ is irreducible. Thus $Z=V_i$. Next, we prove that $V_i$ lies in the Set \eqref{eq:set_of_components}.    By contradiction, if $V_i\subseteq V_j$ for some $j\neq i$, then $V_i\subsetneqq V_j$ by the assumption that $V_i$ and $V_j$ are mutually distinct. This contradicts the fact that $V_i=Z$ is a component because $V_j$ is irreducible. 

    Conversely, let $V_i$ lie in the Set \eqref{eq:set_of_components}. Since $V_i$ is irreducible, it must be contained in an irreducible component $Z$ of $X$. By the previous part, $Z=V_j$ for some $j$. Thus $V_i\subseteq V_j$, which can only happen if $i=j$ since $V_i$ is assumed to be in the Set \eqref{eq:set_of_components}. Thus $V_i = Z$ is an irreducible component. 
\end{proof}

Before stating a useful corollary of \Cref{lem:topological_lemma}, we introduce a notation for the closure.

\begin{notation}\label{not:closure}
    Let $X$ be a scheme and let $Y\subseteq X$ be a subscheme. We denote by $\cl_X (Y)$ the closure of $Y$ in $X$, with its reduced induced structure. If the ambient $X$ is clear from the context, we write $\overline{Y}$ for $\cl_X(Y)$.
\end{notation}

The following result proves that the irreducible components of $C_B(\ccF)$ are precisely the maximal elements (with respect to inclusion) among all $\overline{\pi^{-1}(D_{i,j})}$. For this reason, we often call $\overline{\pi^{-1}(D_{i,j})}$ a candidate component.

\begin{corollary}\label{corollary:structure_components_CBF}
    Under \Cref{assumption:cone+stratification}, the irreducible components of $C_B(\ccF)$ are
    \begin{equation}\label{eq:components_CBF}
        \left\{\overline{\pi^{-1}(D_{i,j})}\colon (i,j)\in \Delta,\ \overline{\pi^{-1}(D_{i,j})} \not\subseteq \overline{\pi^{-1}(D_{k,\ell})}\ \forall\, (k,\ell)\in\Delta\setminus\{(i,j)\}\right\}
    \end{equation}
\end{corollary}

\begin{proof}
    The stratification $\{D_{i,j}\}$ induces
    \[
        C_B(\ccF) = \bigcup_{(i,j)\in \Delta} \overline{\pi^{-1}(D_{i,j})}.
    \]
    By compatibility of abelian cones with pullback, $\pi^{-1}(D_{i,j}) = C_{D_{i,j}}(\ccF\mid_{D_{i,j}})$. By \Cref{assumption:cone+stratification}, $\ccF\mid_{D_{i,j}}$ has constant rank $f+i$ on the reduced scheme $D_{i,j}$, therefore $\pi^{-1}(D_{i,j})$ is a vector bundle over the irreducible scheme $D_{i,j}$, so $\overline{\pi^{-1}(D_{i,j})}$ is irreducible. Furthermore, if $\overline{\pi^{-1}(D_{i,j})} = \overline{\pi^{-1}(D_{k,\ell})}$, continuity of $\pi$ implies that $\overline{D}_{i,j} = \overline{D}_{k,\ell}$, therefore $D_{i,j}\cap D_{k,\ell} \neq \emptyset$, which is not possible by \Cref{assumption:cone+stratification}. Therefore the schemes $\overline{\pi^{-1}(D_{i,j})}$ are mutually distinct and the result follows from \Cref{lem:topological_lemma}. 
\end{proof}

\begin{lemma}\label{lemma:main_is_component}
    Under \Cref{assumption:cone+stratification}, $\overline{\pi^{-1}(B_0)} = \overline{\pi^{-1}(D_{0,0})}$ is an irreducible component of $C_B(\ccF)$. 
\end{lemma}

\begin{proof}
    It is clear that $\overline{\pi^{-1}(B_0)} = \overline{\pi^{-1}(D_{0,0})}$ by the definition of $D_{0,0}$ in \Cref{rmk:D00}. 
    If $\overline{\pi^{-1}(B_0)} = \overline{\pi^{-1}(D_{0,0})}$ is not a component, then there exists $(i,j) \in \Delta\setminus\{(0,0)\}$ such that
    \[
        \overline{\pi^{-1}(D_{0,0})} = C^\circ_B(\ccF) \subsetneqq \overline{\pi^{-1}(D_{i,j})}.
    \]
    By continuity of $\pi$, we have that $D_{0,0} \subseteq \overline{D}_{i,j}$, therefore $B = \overline{D}_{0,0} = \overline{D}_{i,j}$. This is not possible because $D_{0,0}\cap D_{i,j} = \emptyset$.
\end{proof}

\begin{definition}\label{def:main_component}
   Under \Cref{assumption:cone+stratification}, we call
    \[
        C^\circ_B(\ccF) \coloneqq 
        \overline{\pi^{-1}(B_0)} = 
        \overline{\pi^{-1}(D_{0,0})}
    \]
    the main (irreducible) component of $C_B(\ccF)$. Irreducible components of $C_B(\ccF)$ different from the main component are called extra components.
\end{definition}

\subsection{Numerical criteria}\label{subsec:numerical_criteria}

By \Cref{corollary:structure_components_CBF}, we can determine the irreducible components of $C_B(\ccF)$ by understanding the inclusions among candidate components. We give several numerical criteria that ensure or forbid such inclusions based on their dimension, which can be computed from the codimension and the rank of $\ccF$ on each stratum $D_{i,j}$. 

We start by computing the dimension of each candidate component in \Cref{corollary:structure_components_CBF}.

\begin{lemma}\label{lemma:dimension_tentative_components}
    Under \Cref{assumption:cone+stratification}, let $(i,j)\in\Delta$. Then 
    \begin{equation}\label{eq:dimension_preimage_Dij}
        \dim(\overline{\pi^{-1}(D_{i,j})}) = \dim(\pi^{-1}(D_{i,j})) = 
        \dim(D_{i,j}) + \rk_{D_{i,j}}(\ccF\mid_{D_{i,j}}) =    
        b - \codim_B(D_{i,j}) + f + i.
    \end{equation}
\end{lemma}

\begin{proof}
    The result is immediate because, as explained in the proof of \Cref{corollary:structure_components_CBF}, $\pi^{-1}(D_{i,j})$ is a vector bundle of rank $f+i$ over the irreducible scheme $D_{i,j}$.
\end{proof}

\begin{notation}\label{notation:dij}
    Under \Cref{assumption:cone+stratification}, let $(i,j)\in \Delta$. We define
    \begin{equation}\label{eq:definition_dij}
        d_{i,j}\coloneqq \dim(\pi^{-1}(D_{i,j})) - \dim(C_B^\circ(\ccF))= i- \codim_B(D_{i,j}).
    \end{equation}
    The last equality holds because $\dim(C_B^\circ(\ccF)) = \dim(\pi^{-1}(D_{0,0})) = b+f$ by \Cref{lemma:dimension_tentative_components}. In particular, $d_{0,0} = 0$. It follows immediately from \Cref{eq:dimension_preimage_Dij,eq:definition_dij} that
    \begin{equation}\label{eq:dimension_preimage_Dij_using_dij}
        \dim(\pi^{-1}(D_{i,j})) =  b + f + d_{i,j}.
    \end{equation}
\end{notation}

Our first observation is that inclusions between the candidate components $\overline{\pi^{-1}(D_{i,j})}$ imply both numerical relations among the numbers $d_{i,j}$ and inclusions between their supports.

\begin{proposition}\label{lem:containment_tentative_components_numerical_implication}
    Under \Cref{assumption:cone+stratification}, let $(i,j),(k,\ell) \in \Delta$ with $(i,j)\neq (k,\ell)$.
    If
    \[
       \overline{\pi^{-1}(D_{i,j})} \subseteq \overline{\pi^{-1}(D_{k,\ell})}.
    \] 
    then $D_{i,j}\subseteq \overline{D}_{k,\ell}$ and $d_{i,j} < d_{k,\ell}$.
\end{proposition}

\begin{proof}
    Projecting to $B$, we see that $D_{i,j} \subseteq \overline{D}_{k,\ell}$. 
    Since $(i,j)\neq (k,\ell)$, the inclusion $\overline{\pi^{-1}(D_{i,j})} \subseteq \overline{\pi^{-1}(D_{k,\ell})}$ must be strict. Therefore $\dim (\pi^{-1}(D_{i,j})) < \dim(\pi^{-1}(D_{k,\ell}))$, thus $d_{i,j} < d_{k,\ell}$ by \Cref{eq:dimension_preimage_Dij_using_dij}.
\end{proof}

As a corollary of \Cref{lem:containment_tentative_components_numerical_implication}, we get the following numerical criterion to detect irreducible components of $C_B(\ccF)$.

\begin{proposition}\label{prop:numerical_sufficient_criterion_component}
    Under \Cref{assumption:cone+stratification}, let $(i,j) \in \Delta$. 
    If $d_{i,j} \geq d_{k,\ell}$ for every $(k,\ell) \in \Delta$ such that $D_{i,j}\subset \overline{D}_{k,\ell}$, then $\overline{\pi^{-1}(D_{i,j})}$ is an irreducible component.
\end{proposition}

\begin{proof}
    This is an immediate corollary of \Cref{corollary:structure_components_CBF} and \Cref{lem:containment_tentative_components_numerical_implication}.
\end{proof}

The following criterion is useful to ensure that a candidate component $\overline{\pi^{-1}(D_{i,j})}$ is not a component.

\begin{proposition}\label{prop:same_i_implies_inclusion}
     Under \Cref{assumption:cone+stratification}, let $(i,j),(k,\ell) \in \Lambda$ such that $D_{i,j}\subseteq \overline{D}_{k,\ell}$. If $k=i$ then $\overline{\pi^{-1}(D_{i,j})} \subseteq \overline{\pi^{-1}(D_{k,\ell})}$.
\end{proposition}

\begin{proof}
     Both $\pi^{-1}(D_{i,j})$ and $\overline{\pi^{-1}(D_{i,\ell})}$ are contained in the cone
    \[
        C_{i,\ell}\coloneqq  C_{\overline{D}_{i,\ell}} (\ccF\mid_{\overline{D}_{i,\ell}}) = \pi^{-1}\left(\overline{D}_{i,\ell}\right).
    \]
    Furthermore, $\overline{\pi^{-1}(D_{i,\ell})} = C_{i,\ell}^\circ$. Since $\rk(\ccF\mid_{D_{i,j}}) = \rk(\ccF\mid_{D_{i,\ell}})$, we have that $D_{i,j}\subseteq B_0(\ccF\mid_{\overline{D}_{i,\ell}})$. This implies that 
    \[
        \pi^{-1}(D_{i,j}) \subseteq C_{i,\ell}^\circ = \overline{\pi^{-1}(D_{i,\ell})}.\qedhere
    \]
\end{proof}

Next, we relate the dimension of any irreducible component $Z$ of $C_B(\ccF)$ with the dimension of the main component under an extra assumption on the sheaf $\ccF$, which we introduce first.

\begin{definition}\label{def:local_pdim}
    Let $B,\ccF$ as in \Cref{assumption:cone+stratification}. The sheaf $\ccF$ has local projective dimension at most 1 if there exists an open cover $\{U_i\}$ of $B$ such that, for each $i$, there is an exact sequence
    \begin{equation}\label{eq:def_pdim1_2_term_presentation_ses}
            0 \to \ccE_1\to \ccE_0 \to \ccF\mid_{U_i}\to 0
    \end{equation}
    with $\ccE_0, \ccE_1$ projective (equivalently, locally free) on $U_i$. In that case, we write $\pdim(\ccF) \leq 1$.
\end{definition}

In practice, we can check the condition $\pdim(\ccF)\leq 1$ by looking at the ranks in a 2-term presentation of $\ccF$, as we explain next.

\begin{lemma}\label{rmk:k=0_locally_implies_local_pdim_1}
    Let $B,\ccF$ as in \Cref{assumption:cone+stratification}. 
    Then $\pdim(\ccF)\leq 1$ if and only if there exists an open cover $\{U_i\}$ of $B$ with the following property: for each $i$, there is an exact sequence
    \[
            \ccE_1\to \ccE_0 \to \ccF\mid_{U_i}\to 0
    \]
    with $\ccE_1$ and $\ccE_0$ locally free sheaves such that $\grk_B(\ccF) = \grk_{U_i}(\ccF) = \rk(\ccE_0) - \rk(\ccE_1)$. 
\end{lemma}

\begin{proof}
    In a sequence as in the statement, we have that $K=\ker(\ccE_1\to \ccE_0)$ satisfies $\grk_{U_i}(K) = 0$. This implies that $K$ is torsion, thus $K=0$ because it injects in a locally free sheaf. Conversely, if $\pdim(\ccF)\leq 1$ then the sequence \ref{eq:def_pdim1_2_term_presentation_ses} has the desired form. 
\end{proof}

\begin{theorem}\label{thm:dimension_main_is_lower_bound}
    Let $B,\ccF$ as in \Cref{assumption:cone+stratification} and $Z \subseteq C_B(\ccF)$ be an irreducible component. If $\pdim(\ccF)\leq 1$ then $\dim Z\geq \dim C^{\circ}_B(\ccF)$.
\end{theorem} 

\begin{proof}
    Fix $Q\in C_B(\ccF)$. It suffices to prove that the local dimension of $C_B(\ccF)$ at $Q$ is at least $\dim C^{\circ}_B(\ccF)$. 
    Since $\ccF$ is coherent, there is an affine open neighbourhood $U=\Spec(R)$ of $\pi(Q)$ in $B$ where the $R$-module $M=\Gamma(U,\ccF)$ has a free presentation
    \begin{equation}\label{eq:ses_module_M}
        \begin{tikzcd}
            R^g \arrow{r}{\phi} & R^h \arrow{r} & M \arrow{r} & 0.
        \end{tikzcd}
    \end{equation}
    with $h\geq 1$, $g\geq 0$. Furthermore, there is one such presentation with $f = h-g$ because $\pdim(\ccF)\leq 1$, see \Cref{rmk:k=0_locally_implies_local_pdim_1}. Using the compatibility of abelian cones with pullback, we have that
    \[
        \pi^{-1}(U) = C_U(\ccF\mid_U) \simeq \Spec(R[x_1,\ldots, x_h]/I)
    \]
    where 
    \[
        I = \left\langle \sum_{i=1}^{h} \phi_{i,j} x_i \right\rangle_{j=1}^{g},
    \]
    and where $\phi_{i,j}$ denote the entries of the matrix $\phi$.
    Therefore
    \[
        \dim_Q(C_B(\ccF)) = \dim_Q(\pi^{-1}(U)) \geq \dim(U)+h-g = \dim(B) + h-g = b+f = \dim(C^{\circ}_B(\ccF)),
    \]
    where the last equality follows from \Cref{eq:dimension_preimage_Dij}.
\end{proof}

\Cref{thm:dimension_main_is_lower_bound} will be used next to compare inclusions between candidate components and the main component. 
We will also use it to prove that the main component of stable maps is indeed an irreducible component (if it is non-empty), and also that it has the smallest dimension among all irreducible components under a positivity condition on the curve class that implies $\pdim(\ccF)\leq 1$, see Items (ii) and (iii) in \Cref{thm:components_maps_collection}.

Under the assumption $\pdim(\ccF)\leq 1$, we can characterize the candidate components $\overline{\pi^{-1}(D_{i,j})}$ which are contained in the main component in terms of the integers $\{d_{i,j}\}$ as a corollary of \Cref{thm:dimension_main_is_lower_bound}. Recall that $d_{0,0} = 0$.

\begin{proposition}\label{prop:containment_main_component}
    Under \Cref{assumption:cone+stratification} and $\pdim(\ccF) \leq 1$, let $(i,j) \in \Delta$ with $(i,j)\neq (0,0)$. 
    \begin{enumerate}
        \item\label{item:small_dij} If $d_{i,j} < d_{0,0}$ then $\overline{\pi^{-1}(D_{i,j})}$ is not an irreducible component of $C_B(\ccF)$. 
        \item\label{item:big_dij} If $d_{i,j} \geq d_{0,0}$ then $\overline{\pi^{-1}(D_{i,j})} \not\subset C^\circ_B(\ccF)$.
    \end{enumerate}
\end{proposition}

\begin{proof}
    If $d_{i,j}<d_{0,0}$ then $\dim(\overline{\pi^{-1}(D_{i,j})}) < \dim(C^\circ_B(\ccF))$ by  \Cref{eq:dimension_preimage_Dij_using_dij}, thus $\overline{\pi^{-1}(D_{i,j})}$ is not a component by \Cref{thm:dimension_main_is_lower_bound}. This proves \Cref{item:small_dij}. On the other hand, \Cref{item:big_dij} follows from \Cref{lem:containment_tentative_components_numerical_implication}, because the inclusion $D_{i,j} \subseteq \overline{D}_{0,0} = B$ is immediate.
\end{proof}

\begin{remark}
    For the reader interested in applications to stable maps, we will explain in \Cref{subsec:maps_open_in_cone} how to view the moduli space of stable maps $\overline{\ccM} = \overline{\ccM}_{0,n}(X,\beta)$, with $X$ a smooth projective toric variety, as an open substack of an abelian cone $S$, corresponding to a sheaf $\ffF$ over a stack $\ffP^{st}$ (see Definition \eqref{eq:definition_ambient_cone} and Embedding \eqref{eq:open_embedding_stable_maps}). In particular, the irreducible components of $\overline{\ccM}$ are precisely the irreducible components of $S$ which intersect $\overline{\ccM}$ (see \Cref{thm:components_maps_collection}). 
    The results of \Cref{sec:irreducible_comps_cone} will be used in \Cref{sec:Bl_P2} to describe all the irreducible components of $\overline{\ccM}$ in concrete examples. For that, we shall consider a natural stratification of $\ffP^{st}$ indexed by certain (decorated) dual graphs (see \Cref{cor:stratification_prod_Pic}). For such a graph $\GG$, the integers $i$ and $d_{i,j}$ will be denoted by $i_{\GG}$ and $d_{\GG}$, and they are easy to compute (see \Cref{eq:rank_is_h0,eq:generic_rank_is_h0_P1} and Definitions \eqref{eq:definition_iG_prod} and \eqref{eq:definition_dG_prod}). Finally, the assumption $\pdim(\ffF) \leq 1$ will be verified in \Cref{lem:stratification_prod_Pic_satisfies_assumptions} using \Cref{prop:R1_has_pdim_1}.
\end{remark}

\subsection{Examples}\label{subsec:examples_abelian_cones}

We gather some simple examples of abelian cones where we illustrate the concepts and results introduced in \Cref{subsec:structure_components_cone,subsec:numerical_criteria}.

The following is a counterexample to \Cref{thm:dimension_main_is_lower_bound} for a sheaf $\ccF$ admitting a 3-term resolution (i.e. whose global projective dimension of $\ccF$ is 2), showing that the assumption $\pdim(\ccF)\leq 1$ in \Cref{thm:dimension_main_is_lower_bound} is sharp. 

\begin{example}[Skyscraper sheaf]\label{ex:skycraper}
    Let $R=\bbC[x,y]$, $B=\Spec(R) = \bbA^2$ and $M=\bbC[x,y]/(x,y)$ the skyscraper sheaf at the origin. Then $M$ has rank $0$ on $D_{0,0} = B_0 = \bbA^2\setminus \{(0,0)\}$ and rank $1$ on $D_{1,0} = B_1 = \{(0,0)\}$. Thus $d_{0,0} = 0$ and $d_{1,0} = -1$.
    
    We have the following free resolution of $M$
      \begin{equation}\label{eq:sky}
      \begin{tikzcd}[ampersand replacement=\&]
          0 \arrow[r] \& R \arrow[r,"{\left(\begin{smallmatrix} y \\  -x  \end{smallmatrix}\right)}"] \& R^2  \arrow[r,"{\left(\begin{smallmatrix} x &\  y  \end{smallmatrix}\right)}"] \& R \arrow[r] \& M \arrow[r] \& 0
      \end{tikzcd}
      \end{equation}
      One can check that $M$ has projective dimension 2. The abelian cone associated to $M$ on $B=\Spec(R)$ is 
      \[
        C_{B}(M) = \Spec_R \Symm (M) =\bbC[x,y,A]/(xA, yA)
      \] 
       Hence, the abelian cone $C_B(M)$ has 2 irreducible components 
      \begin{itemize}
          \item $C^\circ_B(M) = \overline{\pi^{-1}(D_{0,0})} = V(A)$ of dimension $b+d_{0,0} = 2 + 0 = 2$ and
          \item $\overline{\pi^{-1}(D_{1,0})} = V(x,y)$ of dimension $b + d_{1,0} = 2 - 1 = 1$.
      \end{itemize}
     In particular, \Cref{thm:dimension_main_is_lower_bound} does not hold for this abelian cone.
\end{example}

\begin{example}\label{ex:ideal_origin_irreducible}
     Let $R=\bbC[x,y]$, $B=\Spec(R) = \bbA^2$ and $I=(x,y)$ the ideal sheaf of the origin $0\in \bbA^2$. Then $I$ has rank $1$ on $D_{0,0} = B_0 = \bbA^2\setminus \{0\}$ and rank $2$ on $D_{1,0} = B_1 = \{0\}$. Thus $d_{0,0} = 0$ and $d_{1,0} = -1$.
     
     We have the following resolution of $I$
      \begin{equation}\label{eq:I}
      \begin{tikzcd}[ampersand replacement=\&]
          0 \arrow[r] \& R  \arrow[r,"{\begin{psmallmatrix} y \\ -x  \end{psmallmatrix}}"] \& R^2 \arrow[r,"{\begin{psmallmatrix} x & y \end{psmallmatrix}}"]\& I \arrow[r] \& 0
      \end{tikzcd}
    \end{equation}
    It is easy to check that $I$ has projective dimension 1. We have that
      \[
      C_B(I)=\bbC[x,y,X,Y]/(yX-xY)
      \]
    is irreducible of dimension $3$.
    
    More generally, if $I$ is the ideal sheaf of the origin 0 in $\bbA^n$ then $D_{0,0} = \bbA^n\setminus \{0\}$, $D_{n,0} = \{0\}$, $d_{0,0} = 0$ and $d_{n,0} = -1$. Furthermore, $I$ has projective dimension $n-1$ and $C_B(I)$ is irreducible because it is open in $\operatorname{Proj}_B(\Symm(I)) = \Bl_0 \bbA^n$.
\end{example}

We observe that \Cref{ex:skycraper} and \Cref{ex:ideal_origin_irreducible} have the same stratifications and the same $\{d_{i,j}\}$. However, the second one is irreducible while the first one is not, and the symmetric algebras are very different. These examples show that the stratification $\{D_{i,j}\}$ and the numbers $d_{i,j}$ only capture partial information about the abelian cone $C_B(\ccF)$, and one should not expect them to determine all the irreducible components, at least not in general. However, under the assumption that $\pdim(\ccF)\leq 1$, we will be able to determine all the irreducible components of interesting examples in \Cref{sec:Bl_P2}, combining \Cref{corollary:structure_components_CBF} and  \Cref{prop:numerical_sufficient_criterion_component,prop:same_i_implies_inclusion,prop:containment_main_component}.

\subsection{Generalization to stacks}

We explain how to generalize to stacks the results proven so far for schemes. These results will be used for the application to stable maps, see \Cref{sec:application_stable_maps,sec:Bl_P2}.

The following assumption will replace \Cref{assumption:cone+stratification} when we work with stacks.

\begin{assumption}\label{assumption:stratification_stacks}
    Let $\ffB$ be a smooth Noetherian Artin stack and $\ffF$ a coherent sheaf on $\ffB$. Let $b=\dim(\ffB)$, $f=\grk_{\ffB}(\ffF)$ and $m=\mrk_{\ffB}(\ffF)$. We consider a finite stratification
    \[  
        \ffB = \bigsqcup_{(i,j) \in \Delta} \ffD_{i,j}
    \]
    with the following properties: 
    \begin{enumerate}
        \item $\ffD_{i,j}$ is a locally closed irreducible substack of $\ffB$, considered with its reduced induced structure,
        \item $\ffF\mid_{\ffD_{i,j}}$ has constant rank and $\rk(\ccF\mid_{\ffD_{i,j}}) = f + i$.
    \end{enumerate}
\end{assumption}

As in the case of schemes, we define integers $d_{i,j}$ as in Definition \eqref{eq:definition_dij} for each $(i,j) \in \Delta$, the stratum $\ffD_{0,0}$ as in \Cref{rmk:D00} and the main component of $C_\ffB(\ffF)$ as
\[
    C^\circ_\ffB(\ffF) \coloneqq 
        \overline{\pi^{-1}(B_0(\ffF))} = \overline{\pi^{-1}(\ffD_{0,0})}.
\]

\begin{proposition}\label{prop:results_for_stacks}
    Let $\ffB$, $\ffF$ and $\{\ffD_{i,j}\}$ satisfying \Cref{assumption:stratification_stacks}. Then the conclusion of 
    \Cref{corollary:structure_components_CBF},
    \Cref{lemma:main_is_component},
    \Cref{lemma:dimension_tentative_components},
    \Cref{lem:containment_tentative_components_numerical_implication},
    \Cref{prop:numerical_sufficient_criterion_component} 
    and
    \Cref{prop:same_i_implies_inclusion}
    hold on $C_{\ffB}(\ffF)$. If, furthermore, $\pdim(\ffF)\leq 1$ then the conclusion of 
    \Cref{thm:dimension_main_is_lower_bound}
    and
    \Cref{prop:containment_main_component}  hold on $C_{\ffB}(\ffF)$. 
\end{proposition}

\begin{proof}
    The result follows from the schematic versions because the components of an Artin stack can be computed in a smooth atlas and the conditions are local on the source. We spell out the details below. 
    
    Let $f\colon U\to \ffB$ be a smooth surjective morphism with $U$ a smooth Noetherian scheme. Since abelian cones are compatible with pullback, we have a Cartesian diagram
    \[
        \begin{tikzcd}
            C_U(f^\ast \ffF)\arrow{r}{g}\arrow{d}{\pi_U} & C_\ffB(\ffF) \arrow{d}{\pi}\\
            U\arrow{r}{f} & \ffB
        \end{tikzcd}
    \]
    with $g$ smooth and surjective. By \cite[Lemma 0DR5]{stacks-project}, there is a 1-to-1 correspondence between the irreducible components of $C_U(f^\ast \ffF)$ and of $C_{\ffB}(\ffF)$. Furthermore, we have that $B_i(f^\ast \ffF) = \ffB_i(\ffF)\times_{\ffB} U$ for all $i$ because Fitting ideals commute with base-change. Thus, the collection of $D_{i,j}\coloneqq U\times_{\ffB} \ffD_{i,j}$ for $(i,j)\in\Delta$ define a stratification of $U$ satisfying \Cref{assumption:cone+stratification}. This proves that we can compute the irreducible components of $C_{\ffB}(\ffF)$ by pulling back to a smooth chart $U$ of $\ffB$. 

    To avoid confusions, let us distinguish the integers $d_{i,j}$ from Definition \eqref{eq:definition_dij} using $D_{i,j}$, from the integers $\mathfrak{d}_{i,j}$, defined by the same formula using $\ffD_{i,j}$. To conclude, we need to prove that given $(i,j)\in\Delta$, the condition
    \[
        \mathfrak{d}_{i,j}\geq \mathfrak{d}_{k,\ell} \, \forall\, (k,\ell)\in\Delta \text{ with } \ffD_{i,j}\subseteq \overline{\ffD}_{k,\ell}
    \]
    on $\ffB$ pulls back to the condition
    \[
        d_{i,j}\geq d_{k,\ell} \, \forall\, (k,\ell)\in\Delta \text{ with } D_{i,j}\subseteq \overline{D}_{k,\ell}
    \]
    on $U$. This holds because, by Definition \eqref{eq:definition_dij},
    \[
        d_{i,j} - d_{k,\ell} = \dim(\pi_U^{-1}(\ffD_{i,j})) - \dim(\pi_U^{-1}(\ffD_{k,\ell}))
    \]
    is constant under pullback along smooth surjective morphisms, which have fixed relative dimension.
\end{proof}

\section{Stratifications by dual graphs}\label{sec:background}

We introduce several moduli spaces, such as prestable curves, the Picard stack and a product of Picard stacks. We also associate to them combinatorial data, in the form of trees, often called dual trees, with extra information.

\subsection{Prestable curves and the Picard stack}

\begin{notation}\label{notation:Pic}
    Given $g,n \in \bbZ_{\geq 0}$ and $d \in \bbZ$, we denote by 
    $
        \ffM_{g,n} 
    $
    the stack of genus-$g$ prestable $n$-marked curves and by 
    $
        \Pic_{g,n,d}
    $
    the Picard stack parametrizing tuples $(C,p_1,\ldots, p_n, L)$  with 
    \begin{itemize}
        \item $(C,p_1,\ldots, p_n)$ a
    genus-$g$ prestable $n$-marked curve and
        \item $L$ a line bundle on $C$ of total degree $d$.
    \end{itemize}
    We denote by 
    $\pi\colon \ffC \to \Pic_{g,n,d}$ the universal curve and by $\ffL$ the universal line bundle on $\ffC$. 
\label{notation:Pic_stable}
    We denote by 
    $
        \Pic_{g,n,d}^{sm}
    $
    the open substack of $\Pic_{g,n,d}$ where the curve is smooth and by 
    $
        \Pic_{g,n,d}^{st}
    $
    the open substack of $\Pic_{g,n,d}$ parametrizing tuples $(C,p_1,\ldots, p_n,L)$ where the line bundle
    \begin{equation}\label{eq:stability_Pic}
        L^{\otimes 3} \otimes \omega_\pi\left(\sum_{i=1}^n p_i\right)
    \end{equation}
    is ample. 
    We shall also use the notations $\ffM_{g,M}$ and $\Pic_{g,M,d}$ to denote the analogous stacks where the marking set $\{1,\ldots, n\}$ is replaced by a finite set $M$. 
\end{notation}

The stack $\Pic_{g,n,d}$ is locally Noetherian, irreducible and smooth. Furthermore, $\Pic_{g,n,d}^{st}$ is Noetherian. 

\begin{remark}\label{rmk:Pic_stable_nonempty}
    Note that, since $\Pic_{0,n,d}$ is integral, we have that $\Pic^{st}_{0,n,d}$ is non-empty if and only if $\Pic^{sm}_{0,n,d}\subseteq \Pic^{st}_{0,n,d}$. This inclusion holds if and only if $3d+n-2 > 0$, see Stability \eqref{eq:stability_weighted_trees}. 
\end{remark}

\begin{assumption}\label{assumption:Pic_stable_nonempty}
    We assume that $3d+n-2 > 0$. 
\end{assumption}

\subsection{Stratification by weighted dual graphs}\label{subsec:weighted_dual_trees}

In the particular case $g=0$, we attach combinatorial data, known as weighted dual graphs, to the elements in $\Pic_{0,n,d}$. The loci with fixed dual graph provide a natural stratification of $\Pic_{0,n,d}$.

Given a graph $G$, we denote the set of vertices (resp. edges) of $G$ by $V(G)$ (resp. $E(G)$). Given a vertex $v\in V(G)$, we write $\valency(v)$ for the valency of $v$. Later on we shall consider tuples $\GG$ of a graph $G$ with extra information. By abuse of notation, we often write $V(\GG)$ and $E(\GG)$ for $V(G)$ and $E(G)$.

\begin{definition}\label{def:weighted_tree} \label{def:contraction_weighted_tree}
    A weighted $n$-marked tree is a triple 
    \[
        \GG = (G, m\colon \{1,\ldots, n\}\to V(G), V(G)\to \bbZ\colon v \mapsto d_v)
    \]
    where $G$ is a tree. We say $\GG$ is $d$-weighted if $\sum_{v\in V(G)} d_v = d$. The set of all $d$-weighted $n$-marked trees is denoted by $\Gamma_{0,n,d}$. 
    
    An isomorphism between two weighted $n$-marked trees $\GG_1,\GG_2$ is an isomorphism of the graphs $G_1,G_2$ which commutes with the morphisms $m_i$ and the degree assignments.
    
    Given $\GG \in \Gamma_{0,n,d}$ and $e = (v_1,v_2) \in E(\GG)$, one can produce a new $\GG' \in \Gamma_{0,n,d}$ by deleting the edge $e$ and identifying the vertices $v_1,v_2$ into a new vertex $v'$, which receives all the markings from $v_1$ and $v_2$ and has degree $d_{v'} = d_{v_1} + d_{v_2}$. We say that $\GG'$ is an edge-contraction of $\GG$. More generally, if $\widetilde{\GG} \in \Gamma_{0,n,d}$ can be obtained from $\GG$ by a sequence of edge-contractions, we say that $\widetilde{\GG}$ is a contraction of $\GG$.
    
    We say that a weighted marked tree $\GG \in \Gamma_{0,n,d}$ is stable if 
    \begin{equation}\label{eq:stability_weighted_trees}
        3 d_v + \lvert m^{-1}(v)\rvert + \mathrm{val}(v) - 2 > 0 \ \text{ for every } v\in V(G). 
    \end{equation}
    We denote by $\Gamma_{0,n,d}^{st}$ the set of all stable $d$-weighted $n$-marked trees.
\end{definition}

Note that $\Gamma_{0,n,d}^{st}$ is finite, while $\Gamma_{0,n,d}$, in general, is not.

\begin{construction}[Weighted dual graph]\label{constr:weighted_dual_graph}
    The dual weighted graph of an element $(C,p_1,\ldots, p_n,L)\in \Pic_{0,n,d}$ is the triple $\GG=(G,m,\{d_v\}_{v\in V(G)})$ where $G$ is the dual graph of $C$, $m$ records the markings and $d_v \coloneq \deg(L\mid_{C_v})$.
    By construction, $\GG$ is $d$-weighted, and $(C,p_1,\ldots, p_n,L) \in \Pic^{st}_{0,n,d}$ if and only if $\GG \in \Gamma_{0,n,d}^{st}$.
\end{construction}

\begin{definition}
    For each $d$-weighted $n$-marked tree $\GG \in \Gamma_{0,n,d}$, we let $\Pic_{\GG}$ denote the locally closed substack of $\Pic_{0,n,d}$ parametrizing $(C,p_1,\ldots, p_n,L)$ such that the weighted dual graph of $(C,p_1,\ldots, p_n)$ is isomorphic to $\GG$.
\end{definition}

\begin{remark}\label{rmk:codimension_is_number_of_edges}
    For each $\GG \in \Gamma_{0,n,d}$, it holds that
    $
        \codim_{\Pic_{0,n,d}} (\Pic_{\GG}) = \lvert E(\GG) \rvert.
    $
    Furthermore, if $\GG \in \Gamma^{st}_{0,n,d}$ then also $\codim_{\Pic^{st}_{0,n,d}} (\Pic_{\GG}) = \lvert E(\GG) \rvert$.
\end{remark}

\begin{lemma}\label{lem:irreducible_strata_Pic}
    For each $\GG \in \Gamma_{0,n,d}$, the stack $\Pic_{\GG}$ is irreducible.
\end{lemma}

\begin{proof}
    For each vertex $v\in V(G)$, let 
    $
        M_v = m^{-1}(v) \bigsqcup \{e\in E(\GG)\colon v\in e\}.
    $
    Then there is a surjection from the irreducible stack
    $
        \left(\prod_{v\in V(G)} \Pic^{sm}_{0,M_v,d_v}\right)/\mathrm{Aut}(\GG)
    $
    to $\Pic_{\GG}$.
    In particular $\Pic_{\GG}$ is irreducible.
\end{proof}

\begin{corollary}\label{cor:stratification_Pic}
    The following are stratifications by locally closed irreducible substacks 
    \begin{align*}
        \Pic_{0,n,d} = \bigsqcup_{\GG \in \Gamma_{0,n,d}/_\simeq} \Pic_{\GG},& &  \Pic_{0,n,d}^{st} = \bigsqcup_{\GG \in \Gamma_{0,n,d}^{st}/_\simeq} \Pic_{\GG}.
    \end{align*}
\end{corollary}

\begin{proof}
    Immediate by \Cref{constr:weighted_dual_graph} and \Cref{lem:irreducible_strata_Pic}.
\end{proof}

\begin{lemma}\label{lem:h1_constant_Pic}
    The morphisms $\Pic_{0,n,d}\to \bbZ$ mapping $(C,p_1,\ldots, p_n,L)$ to $h^0(C,L)$ and to $h^1(C,L)$ are both constant on $\Pic_{\GG}$ for each $\GG \in \Gamma_{0,n,d}$.
\end{lemma}

\begin{proof}
    By Riemann-Roch, it suffices to show that $h^0(C,L)$ is constant. This holds because $h^0(C,L)$ can be computed combinatorially from the dual graph, combining the normalization sequence with \Cref{lem:dim_cohomology}. We provide a strategy to do so in \Cref{rmk:strategy_recursive_formula_h0}, based on \Cref{prop:base_larger_degree_implies_no_sections}.
\end{proof}

\begin{notation}\label{notation:h1_weighted_graph}
    Given $\GG \in \Gamma_{0,n,d}$ we write $h^0(\GG)$ (resp. $h^1(\GG)$) for the value of $h^0(C,L)$ (resp. $h^1(C,L)$) for any $(C,p_1,\ldots, p_n,L)\in \Pic_{\GG}$. 
\end{notation}

\begin{lemma}\label{lem:dim_cohomology}
    On $\bbP^1$ we have
    \begin{equation}\label{eq:dim_h0}
        h^0(\bbP^1,\ccO(k))=\begin{cases}
            k+1 & \mbox{if } k\geq 0\\
            0 & \mbox{if  }k\leq -1.
        \end{cases}\quad \text{and} \quad
        h^1(\bbP^1,\ccO(k))=\begin{cases}
            0 & \mbox{if } k\geq -1\\
            -k-1 & \mbox{if  }k\leq -2.
    \end{cases}
    \end{equation}
\end{lemma}

\begin{proof}
    This is a well-known computation. See for example \cite[Theorem 5.1]{Hartshorne_AG}.
\end{proof}

\subsection{Decorated dual graphs on products of Picard stacks}\label{subsec:decorated_graphs_prod_Pic}

We generalize \Cref{subsec:weighted_dual_trees} by associating decorated dual graphs to certain products of Picard stacks and proving that the loci with fixed decorated dual graphs give a stratification. This stratification will be useful in \Cref{sec:application_stable_maps,sec:Bl_P2} when we study stable maps to a toric variety. See \Cref{subsec:functor_toric_variety} for motivation.

\begin{assumption}\label{assumption:toric}
    We fix a smooth projective toric variety $X$ with fan $\Sigma$, a maximal cone $\sigma \in \Sigma$ and a non-zero effective curve class $\beta\in A_1(X)$. Here, $A_1(X)$ denote the Chow group of 1-dimensional cycles on $X$ modulo rational equivalence.
\end{assumption}

\begin{remark}\label{rmk:why_maximal_cone}
    We present a well-known fact in toric geometry, see \cite{C-FK}.
    Fix $X$ and $\sigma$ as in \Cref{assumption:toric}. By \cite[Theorem 4.1.3]{CLS} there is a short exact sequence
    \[
        0 \to M \to \bigoplus_{\rho \in \Sigma(1)} \bbZ\, D_\rho \to \pic(X) \to 0
    \]
    with $M$ the character lattice of the torus in $X$. Then $\{[D_\rho]\colon \rho \notin \sigma(1)\}$ is a $\bbZ$-basis of $\pic(X)$. Furthermore, let $\sigma(1)=\{\tau_1,\ldots, \tau_n\}$, let $\{u_{\tau_1},\ldots, u_{\tau_n}\}$ be the associated ray generators, which form a basis of the co-character lattice $N$, and let $\{m_1,\ldots, m_n\}$ be the dual basis. Then, by construction, for each $\rho \notin \sigma$ we have that
    \begin{equation}\label{eq:linear_expression_urho}
        u_\rho = \sum_{i=1}^n \langle m_i,u_\rho \rangle u_{\tau_i} 
    \end{equation}
    and 
    $
        \operatorname{div}(\chi^{m_i}) = D_{\tau_i} + \sum_{\rho \notin \sigma(1)} \langle m_i, u_\rho \rangle D_\rho,
    $
    thus we have the following relations in $\Pic(X)$:    \begin{equation}\label{eq:linear_equivalence_dual_basis}
        D_{\tau_i} \sim -\sum_{\rho \notin \sigma(1)} \langle m_i, u_\rho \rangle D_\rho.
    \end{equation}
\end{remark}

\begin{notation}\label{notation:product_Pics}
    Given $g,n\in \bbZ_{\geq 0}$,
    we consider the following fibre product of Picard stacks over $\ffM_{g,n}$
    \begin{equation}\label{eq:product_of_Pics}
        \Pic_{g,n,X,\beta,\sigma} \coloneqq 
        \fprodM_{\rho \notin \sigma(1)}\Pic_{g,n,\beta \cdot D_\rho}.
    \end{equation} We let $\ffC\to \Pic_{g,n,X,\beta,\sigma}$ denote the pullback of the universal curve on $\ffM_{g,n}$. On $\ffC$, we have universal line bundles $\ffL_\rho$ for $\rho \notin \sigma(1)$, each of them is the pullback of the universal line bundle on the $\rho$-th factor. We will denote $\Pic_{g,n,X,\beta,\sigma}$ simply by $\ffP$ when there is no risk of confusion. 
\end{notation}

In light of \Cref{rmk:why_maximal_cone}, we construct extra ``universal'' line bundles $\ffL_\tau$ on $\ffC$ for  $\tau \in \sigma(1)$ as follows.
Since $X$ is smooth, for each $\rho \notin \sigma(1)$ there is a unique $\bbZ$-linear combination in $N$
\[
    u_\rho = \sum_{\tau \in \sigma(1)} \alpha_\tau^\rho u_\tau.
\]
In the notation of  \Cref{eq:linear_expression_urho}, we have that $\alpha_{\tau_i}^{\rho} = \langle m_i,u_\rho\rangle$. Then, motivated by \Cref{eq:linear_equivalence_dual_basis}, we define for each $\tau \in \sigma(1)$ the following ``universal'' line bundle on $\ffC$:
\begin{equation}\label{eq:universal_line_bundle}
   \ffL_\tau \coloneqq \otimes_{\rho\notin \sigma(1)} \ffL_\rho^{-\alpha_\tau^\rho}.
\end{equation}

\begin{notation}\label{notation:product_Pics_stable}
    Given $X,\sigma, \beta$ as in \Cref{assumption:toric}, choose a very ample line bundle $\otimes_{\rho \in \Sigma(1)} \ccO_X(D_\rho)^{\otimes \beta_\rho}$ on $X$. 
    We denote by 
    \begin{equation}\label{eq:stable_prod_Pic}
         \Pic_{g,n,X,\beta,\sigma}^{st}
    \end{equation}
    the open substack of $ \Pic_{g,n,X,\beta,\sigma}$ parametrizing $(C,p_1,\ldots, p_n,(L_\rho)_{\rho \notin \sigma(1)})$ with 
    \begin{equation}\label{eq:stability_stable_maps}
        \left(\bigotimes_{\rho\in \Sigma(1)} L_\rho^{\beta^\rho}\right)^{\otimes 3} \otimes \omega_C\left(\sum_{i=1}^n p_i\right)
    \end{equation}
    ample.
\end{notation}

As in \Cref{subsec:weighted_dual_trees}, we consider trees with extra information and explain how to attach them to elements in $\Pic_{0,n,X,\beta,\sigma}$ , generalizing 
\Cref{constr:weighted_dual_graph}.

\begin{definition}\label{def:decorated_marked_tree}\label{def:contraction_decorated_trees}
    A decorated $n$-marked tree is a triple 
    \[
        \GG = (G, m\colon \{1,\ldots, n\}\to V(G),  V(G)\to A_1^{\eff}(X)\colon v\mapsto \beta_v)
    \]
    with $G$ a tree. We say that $\GG$ is $\beta$-decorated, or that it  has class $\beta$, if $\sum_{v\in V(G)} \beta_v = \beta$. The set of all $\beta$-decorated $n$-marked trees is denoted by $\Gamma_{0,n}(X,\beta)$.    
    An isomorphism between two decorated $n$-marked trees $\GG_1,\GG_2$ is an isomorphism of the graphs $G_1,G_2$ which commutes with the morphisms $m_i$ and the curve class assignments. Contractions and edge-contractions of decorated trees are defined analogously to \Cref{def:contraction_weighted_tree}
    
    We say that $\GG \in \Gamma_{0,n}(X,\beta)$ is stable if
    \begin{equation}\label{eq:stability_decorated_trees}
        \lvert m^{-1}(v)\rvert + \mathrm{val}(v) \geq 3 \ \text{ for every } v\in V(G) \text{ with } \beta_v=0.
    \end{equation}
    We denote by $\Gamma_{0,n}^{st}(X,\beta)$ the set of stable $\beta$-decorated $n$-marked trees.
\end{definition}

The following family of curve classes arises naturally when we look at irreducible components of stable maps, see \Cref{rmk:only_irreducible_classes_intersect_maps}.

\begin{definition}\label{def:mic_class}
    We say that a curve class $\beta \in A_1(X)$ is mic (acronym for multiple of an irreducible-representable class) if 
    $\beta = \lambda [C]$ with $\lambda \in \bbZ_{\geq 0}$ and $C$ an irreducible curve on $X$. We denote by $A_1^{mic}(X)$ the set of all curve classes which are mic on $X$. We say that $\GG \in \Gamma_{0,n}^{st}(X,\beta)$ is mic if $\beta_v$ is an mic for all $v \in V(\GG)$, and denote by $\Gamma_{0,n}^{mic}(X,\beta)$ the set of all $\beta$-decorated $n$-marked mic trees.
\end{definition}

Note that $\Gamma_{0,n}^{st}(X,\beta)$ is finite, while $\Gamma_{0,n}(X,\beta)$, in general, is not. 
We have a natural forgetful morphism to the following fibre product over $\Gamma_{0,n}$ 
\[
    \Gamma_{0,n}(X,\beta)\to \fprodGamma_{\rho \in \Sigma(1)} \Gamma_{0,n,\beta \cdot D_\rho}, 
\]
The $\rho$-th factor is given by
\begin{equation}\label{eq:combinatorial_projection_rho}
    \pi_\rho(G, m, V(G)\to A_1^{\eff}(X)) \coloneqq (G,m, \deg_\rho\colon V(G)\to \bbZ)
\end{equation}
with $\deg_\rho(v) \coloneqq \beta_v\cdot [D_\rho]$.

\begin{construction}[Decorated dual graph]\label{constr:decorated_dual_graph}
    The $\beta$-decorated weighted graph of an element $(C,p_1,\ldots, p_n,(L_\rho)_{\rho \notin \sigma(1)})\in \Pic_{0,n,X,\beta,\sigma}$ is the triple $\GG = (G,m,\{\beta_v\}_{v \in V(\GG)})$ where $G$ is the dual graph of $C$, $m$ records the markings and $\beta_v\in A_1^\eff(X)$ is the unique curve class such that $\beta_v\cdot D_{\rho} = \deg(L_\rho\mid_{C_v})$ for each $\rho \notin \sigma(1)$.
    Note that $\GG$ is stable if and only if $(C,p_1,\ldots, p_n,(L_\rho)_{\rho \notin \sigma(1)})\in \Pic_{0,n,X,\beta,\sigma}^{st}$.
\end{construction}

\begin{definition}
    For each $\GG \in \Gamma_{0,n}(X,\beta)$, $\Pic_{\GG}$ denotes the locally closed substack of $\Pic_{0,n,X,\beta,\sigma}$ parametrizing $(C,p_1,\ldots, p_n,(L_\rho)_{\rho \notin \sigma(1)})$ whose decorated dual graph is isomorphic to $\GG$.
\end{definition}

\begin{remark}\label{rmk:codimension_prod_pics_is_number_of_edges}\label{rmk:inclusion_is_edge_contraction}
    For each $\GG \in \Gamma_{0,n}(X,\beta)$, it holds that
    $
        \codim_{\Pic_{0,n,X,\beta,\sigma}} (\Pic_{\GG}) = \lvert E(\GG) \rvert.
    $
    If $\GG \in \Gamma^{st}_{0,n}(X,\beta)$ then also $\codim_{\Pic^{st}_{0,n,X,\beta,\sigma}} (\Pic_{\GG}) = \lvert E(\GG) \rvert$.
    Furthermore, given $\GG,\GG'\in \Gamma_{0,n}(X,\beta)$, it holds that $\Pic_{\GG}\subseteq \overline{\Pic_{\GG'}}$ if and only if $\GG'$ is a contraction of $\GG$ (in the sense of \Cref{def:contraction_decorated_trees}).
\end{remark}

As in the case of a single line bundle, fixing the decorated dual graph produces stratifications of $\Pic_{0,n,X,\beta,\sigma}$ and $\Pic^{st}_{0,n,X,\beta,\sigma}$.

\begin{corollary}\label{cor:stratification_prod_Pic}
    The following are stratifications by locally closed irreducible substacks 
    \begin{align*}
        \Pic_{0,n,X,\beta,\sigma} = \bigsqcup_{\GG \in \Gamma_{0,n}(X,\beta)/_\simeq} \Pic_{\GG}, & & \Pic_{0,n,X,\beta,\sigma}^{st} = \bigsqcup_{\GG \in \Gamma_{0,n}^{st}(X,\beta)/_\simeq} \Pic_{\GG}.
    \end{align*}
\end{corollary}

\begin{proof}
    The irreducibility follows from an argument analogous to \Cref{lem:irreducible_strata_Pic} and the stability can be checked combinatorially on $\GG$ as stated in \Cref{constr:decorated_dual_graph}. The rest is immediate.
\end{proof}

For the next result, recall that an element of $\Pic_{0,n,X,\beta,\sigma}$ is of the form $(C,p_1,\ldots, p_n,(L_\rho)_{\rho \notin \sigma(1)})$ and we can construct ``universal'' line bundles $L_\rho$ for $\rho \notin \sigma(1)$ following Definition \eqref{eq:universal_line_bundle}.

\begin{lemma}\label{lem:h1_constant_prod_Pic}
    For each $\rho_0 \in \Sigma(1)$ and each $\GG \in \Gamma_{0,n}(X,\beta)$, the morphisms $\Pic_{0,n,X,\beta,\sigma}\to \bbZ$ mapping $(C,p_1,\ldots, p_n,(L_\rho)_{\rho \notin \sigma(1)})$ to $h^0(C,L_{\rho_0})$ and to $h^1(C,L_{\rho_0})$ are constant on $\Pic_{\GG}$.
\end{lemma}

\begin{proof}
    Follows from \Cref{lem:h1_constant_Pic}.
\end{proof}

\begin{notation}\label{notation:h1_decorated_graph}
    Given $\GG \in \Gamma_{0,n}(X,\beta)$ and $\rho_0\in \Sigma(1)$, we write $h^0(\GG,L_{\rho_0})$ (resp. $h^1(\GG,L_{\rho_0})$) for the value of $h^0(C,L_{\rho_0})$ (resp. $h^1(C,L_{\rho_0})$) on any $(C,p_1,\ldots, p_n,(L_\rho)_{\rho \notin \sigma(1)})\in \Pic_{\GG}$. 
\end{notation}

We observe that the values $h^0(\GG,L_\rho)$ and $h^1(\GG,L_\rho)$ are both independent of the chosen maximal cone $\sigma \in \Sigma$.

\section{Abelian cones on the Picard stack}\label{sec:abelian_cones_over_Pic}

Let $\ffP = \Pic_{0,n,d}$ denote the Picard stack and $\ffP^{st} = \Pic_{0,n,d}^{st}$ be the open substack  introduced in \Cref{notation:Pic}.
We denote by $\pi\colon\ffC\to \ffP^{st}$ its universal curve and by $\ffL$ the universal line bundle on $\ffC$.
By \cite{Chang-Li-maps-with-fields}, the abelian cone 
\begin{equation}\label{eq:definition_SL}
    S(\ffL)\coloneqq \Spec_{\ffP^{st}} \Symm (R^1\pi_\ast (\ffL^\vee \otimes \omega_\pi)).
\end{equation}
is the moduli space parametrizing $(C,L,s)$ with $(C,L)$ in $\ffP^{st}$ and $s\in H^0(C,L)$.

\begin{remark}\label{rmk:cohomology_and_basechange}
    The reader may wonder why we use the sheaf $R^1\pi_\ast (\ffL^\vee \otimes \omega_\pi)$ instead of $R^0\pi_\ast \ffL$ in Definition \eqref{eq:definition_SL}. The reason is that the former satisfies cohomology and base change (see \cite[Section 5]{Mumford} or \cite[Theorem 12.11]{Hartshorne_AG}). Combined with Serre duality, this means that for a point $P = (C,p_1,\ldots, p_n,L)\in\ffP^{st}$ we have an isomorphism
    \begin{equation}\label{eq:R1_gives_h0}
        R^1\pi_\ast (\ffL^\vee \otimes \omega_\pi) \otimes k(P) \simeq H^1(C,L^\vee \otimes \omega_C) \simeq H^0(C,L)^\vee 
    \end{equation}
    However, the natural morphism $(R^0\pi_\ast \ffL)\otimes k(P) \to H^0(C,L)$ is not an isomorphism in general (see \Cref{rmk:no_cohomology_and_basechage} for a related example).
\end{remark}

In this section we explore the geometry of the abelian cone $S(\ffL)$ as a warm-up of the study of components of spaces of stable maps to toric varieties in \Cref{sec:application_stable_maps}. Informally speaking, the main difference is that in \Cref{sec:abelian_cones_over_Pic} we work with a single line bundle, while in \Cref{sec:application_stable_maps} we work with a line bundle for each toric boundary divisor. 
We find that explaining first the case of a single line bundle improves the exposition and readability of \Cref{sec:application_stable_maps}. Moreover, the application to the target $\Bl_{\pt}\bbP^2$ will depend only on one line bundle, as we explain in \Cref{subsec:reduction_Le}.

\subsection{Irreducible components of $S(\ffL)$}\label{subsec:components_cone_Pic}

Let $\ffF=R^1\pi_\ast (\ffL^\vee \otimes \omega_\pi)$, which is a coherent sheaf on $\Pic_{0,n,d}^{st}$. Our goal is to study the irreducible components of the abelian cone $S(\ffL) = C_{\ffP^{st}}(\ffF)$ using the results in \Cref{sec:irreducible_comps_cone}. 
For that, by \Cref{prop:results_for_stacks} it suffices to show that the stratification of $\ffP^{st}$ indexed by stable weighted dual trees in \Cref{cor:stratification_Pic} satisfies  \Cref{assumption:stratification_stacks} and that $\pdim(\ffF)\leq 1$, which we do next.

\begin{lemma}\label{lem:stratification_base_SL_satisfies_assumptions}
    Let $n\in \bbZ_{\geq 0}$ and $d\in \bbZ$ satisfying \Cref{assumption:Pic_stable_nonempty}. The stratification
\[
    \Pic_{0,n,d}^{st} = \bigsqcup_{\GG \in \Gamma_{0,n,d}^{st}/_\simeq} \Pic_{\GG}.
\]
of $\Pic_{0,n,d}^{st}$ satisfies \Cref{assumption:stratification_stacks}.
\end{lemma}

\begin{proof}
    By \Cref{cor:stratification_Pic} and the fact that $\Gamma^{st}_{0,n,d}$ is finite, we only need to check that $\ffF$ has constant rank on $\Pic_{\GG}$ for all $\GG \in \Gamma_{0,n,d}^{st}$. Given a stable weighted tree $\GG$ and $P=(C,p_1,\ldots, p_n,L)\in \Pic_{\GG}$, \Cref{eq:R1_gives_h0} implies that
    \begin{equation}\label{eq:rank_is_h0}
        \rk((R^1\pi_\ast (\ffL^\vee \otimes \omega_\pi))\mid_P) = h^1(C,L^\vee\otimes\omega_C) =  h^0(C,L).
    \end{equation}
    We conclude since $h^0(C,L)$ is constant on $\Pic_{\GG}$ by \Cref{lem:h1_constant_Pic}.
\end{proof}

As a particular case of \Cref{eq:rank_is_h0}, we have that if \Cref{assumption:Pic_stable_nonempty} holds then
\begin{equation}\label{eq:generic_rank_is_h0_P1}
    \grk_{\Pic_{0,n,d}^{st}}(R^1\pi_\ast (\ffL^\vee \otimes \omega_\pi)) = 
    \grk_{\Pic_{0,n,d}^{sm}}(R^1\pi_\ast (\ffL^\vee \otimes \omega_\pi)) = 
    h^0(\bbP^1,\ccO_{\bbP^1}(d)).
\end{equation}

\begin{notation}\label{notation:i_G and d_G for single line bundle}
    Fix $\GG\in\Gamma_{0,n,d}^{st}$. Recall that by \Cref{notation:h1_weighted_graph} we write $h^0(\GG)=h^0(C,L)$ for any  $(C,p_1\ldots,p_n,L) \in \Pic_{\GG}$. Let
    \begin{equation}\label{eq:definition_iG}
        i_{\GG} \coloneqq h^0(\GG) - h^0(\bbP^1,\ccO_{\bbP^1} (d)),
    \end{equation}
    and
    \begin{equation}\label{eq:definition_dG}
        d_{\GG} \coloneqq i_{\GG} - \codim_{\Pic^{st}_{0,n,d}}(\Pic_{\GG}) = i_{\GG} - \lvert E(\GG) \rvert.
    \end{equation}
    The last equality follows from \Cref{rmk:codimension_is_number_of_edges}.
\end{notation}

Under \Cref{assumption:Pic_stable_nonempty}, the only weighted graph with a single vertex and no edges $\GG_0\in\Gamma_{0,n,d}$ is stable, and it holds that $d_{\GG_0} = 0$ and $S^{\circ}(\ffL) \coloneqq \overline{S(\ffL)}_{\GG_0} = \overline{\pi^{-1}(\Pic^{sm}_{0,n,d})}$ is the main component of $S(\ffL)$ by \Cref{lemma:main_is_component}.
Furthermore, by \eqref{eq:R1_gives_h0} and the definition of $i_{\GG}$, we have that $\rk(\ffF\mid_{\Pic_{\GG}})$ is constant with value $\grk_{\Pic^{st}_{0,n,d}}(\ffF) + i_{\GG}$.
We also see that 
\[
        d_{\GG}\geq d_{\GG'} \iff h^0(\GG) - \lvert E(\GG) \rvert \geq h^0(\GG') - \lvert E(\GG') \rvert.
\]
So these inequalities, which appear in several statements in \Cref{sec:irreducible_comps_cone}, are easy to check in practice.

\begin{proposition}\label{prop:R1_has_pdim_1}
    Let $n,d \in \bbZ_{\geq 0}$ satisfying \Cref{assumption:Pic_stable_nonempty}. The sheaf $\ffF = R^1\pi_\ast (\ffL^\vee \otimes \omega_\pi)$ on $\Pic^{st}_{0,n,d}$ satisfies $\pdim(\ffF)\leq 1$.
\end{proposition}

\begin{proof}
    Let $B$ be a smooth chart of $\ffP^{st}$ with $B$ integral and Noetherian. Consider sections $A_1,\ldots, A_s$ of $\pi$ on $B$ with the following properties:
    \begin{enumerate}
        \item the sections $A_i$ are disjoint and away from the singular locus of each fibre,
        \item $R^1\pi_\ast \ffL(A) = 0$ on $B$,
        \item $\ccE xt^1(R^0\pi_\ast \ffL(A)\mid_A,\ccO_{\ffP^{st}}) = 0$ on $B$, 
    \end{enumerate}
    where $A=\sum_{i=1}^s A_i$. The existence of sections satisfying the first two properties is detailed in \cite[Section 4.1]{KMMP}, while the last property is true locally on $B$. 
    Consider the following exact sequence on $B$
    \begin{equation}\label{eq:exact,L,LA}
        \begin{tikzcd}
            0\arrow{r} &  \ffL\arrow{r} &  \ffL(A)\arrow{r} &  \ffL(A)\mid_A\arrow{r} &  0.
        \end{tikzcd}
    \end{equation}
     We apply the derived covariant functor $R^\bullet \pi_\ast$ and then we take the dual i.e. we apply the contravariant functor
     $\ccE \to \ccE^\vee :=R \mathcal{H} om(\ccE, \ccO_\Pic)$. We get the following exact triangle
    \begin{equation}\label{eq:exact,triangle}
        \begin{tikzcd}
           \ &\arrow{l}{+1}  (R^\bullet\ffL)^\vee &  \arrow{l}(R^\bullet \ffL(\ffA))^\vee  &  (R^\bullet\ffL(\ffA)\mid_\ffA)^\vee \arrow{l}  
        \end{tikzcd}
    \end{equation}
    Our choice of sections implies that $R^1\pi_\ast \ffL(A)= 0$ and it is clear that $R^1\pi_\ast \ffL(A)\mid_A = 0$, therefore  $R^0\pi_\ast \ffL(A)$ and $R^0\pi_\ast \ffL(A)\mid_A$ are vector bundles.
    Using Serre duality and the assumption about $\ccE xt^1$, we deduce that 
    \begin{equation}\label{eq:long_exact_sequence_R0L_dualized}
        0  \leftarrow 
        R^1\pi_\ast( \ffL^\vee \otimes \omega_\pi)  \leftarrow 
        R^1\pi_\ast (\ffL(A)^\vee \otimes \omega_\pi)  \leftarrow 
        R^1\pi_\ast ((\ffL(A)\mid_A)^\vee \otimes \omega_\pi)  \leftarrow 
        R^0\pi_\ast (\ffL^\vee \otimes \omega_\pi) \leftarrow 
        0.
    \end{equation}
    is exact. By \Cref{rmk:k=0_locally_implies_local_pdim_1} it suffices to prove that
    \[
        \grk_{\ffP^{st}}(R^1\pi_\ast( \ffL^\vee \otimes \omega_\pi)) = \grk_{\ffP^{st}}(R^1\pi_\ast (\ffL(A)^\vee \otimes \omega_\pi)) - \grk_{\ffP^{st}}(R^1\pi_\ast ((\ffL(A)\mid_A)^\vee \otimes \omega_\pi)).
    \]
    By \Cref{eq:rank_is_h0}, this is equivalent to
    \begin{equation}\label{eq:h0_decreases_like_degree}
        h^0(\bbP^1,\ccO_{\bbP^1}(d)) =  h^0(\bbP^1,\ccO_{\bbP^1}(d+s)) - s,
    \end{equation}
    which follows from \Cref{eq:dim_h0} using that $d,s\geq 0$.
\end{proof}

\begin{remark}\label{rmk:no_cohomology_and_basechage}
    If $n,d\in \bbZ_{\geq 0}$ satisfy \Cref{assumption:Pic_stable_nonempty}, the sheaf $R^0\pi_\ast (\ffL^\vee \otimes \omega_\pi)$ does not satisfy cohomology and base change, meaning it does not commute with restrictions to fibers as in \Cref{rmk:cohomology_and_basechange}. To see this, we observe that the proof of \Cref{prop:R1_has_pdim_1} implies that $(R^1\pi_* \ffL)^\vee = R^0\pi_\ast (\ffL^\vee \otimes \omega_\pi) = 0$ (see \Cref{rmk:k=0_locally_implies_local_pdim_1}), but there are elements $(C,p_1,\ldots, p_n,L)\in \ffP^{st}$ such that $h^1(C,L)>0$.
\end{remark}

\begin{notation}
    Let $\pi$ denote the natural projection $S(\ffL)\to \ffP^{st}$. For $\GG\in \Gamma^{st}_{0,n,d}$, we write
    \[
        S(\ffL)_{\GG} \coloneqq \pi^{-1}(\Pic_{\GG}).
    \]
    Following \Cref{not:closure}, we write $\overline{S(\ffL)}_{\GG}$ for $\cl_{S(\ffL)}(S(\ffL)_{\GG})$.
\end{notation}

\Cref{lem:stratification_base_SL_satisfies_assumptions} and \Cref{prop:R1_has_pdim_1} prove that the abelian cone $S(\ffL)$ satisfies the assumptions of \Cref{prop:results_for_stacks}. This implies that that all irreducible components of $S(\ffL)$ are of the form $\overline{S(\ffL)}_{\GG}$ for some stable weighted tree $\GG$, by \Cref{corollary:structure_components_CBF}. Furthermore, the numbers $d_{\GG}$ can be used, through \Cref{prop:numerical_sufficient_criterion_component,prop:same_i_implies_inclusion,prop:containment_main_component}, to determine when certain candidate components $\overline{S(\ffL)}_{\GG}$ are or are not actual irreducible components.

\section{Stable maps to toric varieties and abelian cones}\label{sec:application_stable_maps}

We first recall how maps to a toric variety can be described in terms of line bundle-section pairs using \cite[Theorem 1.1]{Cox}. This description can be used to view the moduli space of stable maps as an open inside an abelian cone (Embedding \ref{eq:open_embedding_stable_maps}), generalizing the analogous statement for projective spaces in \cite[Section 2]{Chang-Li-maps-with-fields}. The main result of this section is \Cref{thm:components_maps_collection}, where we explain how to relate irreducible components of stable maps and of the abelian cone, and give useful criteria to find irreducible components of stable maps, based on \Cref{sec:irreducible_comps_cone}.

\subsection{Maps to a toric variety}\label{subsec:functor_toric_variety}

Let $X$ be a smooth toric variety with fan $\Sigma$ in a lattice $N$. Let $M$ be the dual of $N$, that is, the character lattice of the torus in $X$.
A $\Sigma$-collection on a scheme $S$ is a triplet
\begin{equation}\label{eq:sigma_collection}
        \left((L_\rho)_{\rho\in \Sigma(1)}, (s_\rho)_{\rho\in \Sigma(1)}, (c_m)_{m\in M}\right)
\end{equation} 
of line bundles $L_\rho$ on $S$, sections $s_\rho \in H^0(S,L_\rho)$ and isomorphisms
    \[
        c_m\colon \otimes_{\rho\in\Sigma(1)} L_\rho^{\otimes \langle m,u_\rho\rangle} \simeq \ccO_S 
    \]
satisfying the following conditions:
    \begin{enumerate}
        \item compatibility: $c_m\otimes c_{m'} = c_{m+m'}$ for all $m,m'\in M$,
        \item\label{item:non-degeneracy} non-degeneracy: for each $x\in S$ there is a maximal cone $\sigma\in\Sigma$ such that $s_\rho(x)\neq 0$ for all $\rho\notin \sigma(1)$.
    \end{enumerate}
An equivalence between two $\Sigma$-collections on $S$ is a collection of isomorphisms among  the line bundles that preserve the sections and trivializations. Let $F_\Sigma$ be the functor associating to a scheme $S$ over $\bbC$ the set of isomorphism classes of $\Sigma$-collections on $S$. 

\begin{proposition}(\cite[Theorem 1.1]{Cox})\label{prop:funtor_points_toric}
    The functor $F_\Sigma$ is represented by the smooth toric variety $X$.
\end{proposition}

\begin{remark}\label{rmk:sigma_collection_equivalent_without_cm}
    If $X$ is smooth and projective and $\sigma \in \Sigma$ is a maximal cone, one can deduce from \Cref{rmk:why_maximal_cone} that giving a $\Sigma$-collection on $S$ is equivalent to giving a pair 
    \[
        \left((L_\rho)_{\rho\notin \sigma(1)}, (s_\rho)_{\rho\in \Sigma(1)}\right)
    \]
    satisfying the non-degeneracy condition. The line bundles $(L_{\rho})_{\rho\in\sigma(1)}$ are then recovered from the line bundles $(L_\rho)_{\rho \notin \sigma(1)}$ using Definition \eqref{eq:universal_line_bundle}.
\end{remark}

\begin{example}
    For example, let $X=\bbP^2$ with the fan in \Cref{fig:fan_P2}. By \Cref{prop:funtor_points_toric}, a morphism $S\to \bbP^2$ is determined by the data of line bundle section pairs $(L_i,s_i)_{i=0}^2$, satisfying the non-degeneracy condition, and isomorphisms
    \begin{align*}
        c_{(1,0)}\colon L_1\otimes L_0^\vee \simeq \ccO_S,    & & c_{(0,1)}\colon L_2\otimes L_0^\vee \simeq \ccO_S.
    \end{align*}
    Thus, using the isomorphisms $c_{(1,0)}$ and $c_{(0,1)}$, this data is equivalent to the data of a single line bundle $L$ with three non-degenerate sections $(s_i)_{i=0}^2$. This recovers the usual description of morphisms to $\bbP^2$.
\end{example}

\begin{figure}
    \centering
			\begin{tikzpicture}[scale = .8]
				\fill[fill=gray!40]  (0,0) -- (2,0) -- (2,2) -- (0,2);
				\fill[fill=gray!40]  (0,0) -- (0,2) -- (-2,2) -- (-2,-2);
				\fill[fill=gray!40]  (0,0) -- (2,0) -- (2,-2) -- (-2,-2);
				\draw[thick] (0,0)--(0,2);
				\draw[thick] (0,0)--(2,0);
				\draw[thick] (0,0)--(-2,-2);
    		\draw (2.5,0) node[align=center] {$\rho_{1}$};
    		\draw (0,2.5) node[align=center] {$\rho_{2}$};
    		\draw (-2.5,-2.5) node[align=center] {$\rho_{0}$};
            \draw (1,1) node[align=center] {$\sigma_{1,2}$};
    		\draw (-1,1) node[align=center] {$\sigma_{0,2}$};
    		\draw (1,-1) node[align=center] {$\sigma_{0,1}$};    
    		\end{tikzpicture}
    		\caption{The fan of $\bbP^{2}$.}
        \label{fig:fan_P2}
\end{figure}

\subsection{Stable maps as open in an abelian cone}\label{subsec:maps_open_in_cone}

We explain how to embed the moduli space of stable maps to a toric variety as open in an abelian cone, following \cite[Section 2]{Chang-Li-maps-with-fields}.

As in \Cref{assumption:toric}, we let $X$ be a smooth projective toric variety with fan $\Sigma$, let $\sigma\in \Sigma$ be a maximal cone and $\beta \in A_1(X)$ be an effective non-zero curve class. We also fix $g,n\in \bbZ_{\geq 0}$, and consider the stack $\Pic^{st}_{g,n,X,\beta,\sigma}$ introduced in \Cref{notation:product_Pics_stable}. Recall that on the universal curve $\ffC$ over $\Pic^{st}_{g,n,X,\beta,\sigma}$ we have line bundles $\ffL_{\rho}$ for $\rho \in \Sigma(1)$ (see \Cref{notation:product_Pics} and Definition \eqref{eq:universal_line_bundle}).

We consider the abelian cone
\begin{equation}\label{eq:definition_ambient_cone}
    S_{g,n,X,\beta,\sigma} \coloneqq \Spec_{\ffP^{st}_{g,n,X,\beta,\sigma}} \Symm\left(\oplus_{\rho \in \Sigma(1)} R^1\pi_\ast (\ffL_\rho^\vee \otimes \omega_\pi)\right),
\end{equation}
which we denote simply by $S$ if there is no risk of confusion. Then there is an open embedding
\begin{equation}\label{eq:open_embedding_stable_maps}
    \overline{\ccM}_{g,n}(X,\beta)\hookrightarrow S_{g,n,X,\beta,\sigma}
\end{equation}
cut out by the non-degeneracy condition.

\begin{remark}\label{rmk:quasimap_stability}
    There is an analogous open embedding for quasimaps. Firstly, following \Cref{notation:product_Pics_stable}, we take the open substack $\Pic_{g,n,X,\beta,\sigma}^{st,\ccQ}$ of $\Pic_{g,n,X,\beta,\sigma}$ parametrizing tuples $(C,p_1,\ldots, p_n,(L_\rho)_{\rho\in\Sigma(1)})$ with 
    \begin{equation}\label{eq:stability_stable_qmaps}
        \left(\bigotimes_{\rho\in \Sigma(1)} L_\rho^{\beta^\rho}\right)^{\otimes \epsilon} \otimes \omega_C\left(\sum_{i=1}^n p_i\right)
    \end{equation}
    ample for all $\epsilon\in \ccQ_{>0}$. Note that this is a stronger stability condition than the one in defined by line bundle \eqref{eq:stability_stable_maps}. 
    
    The analogue of the ambient abelian cone in Definition \eqref{eq:definition_ambient_cone} is
    \begin{equation}\label{eq:definition_ambient_cone_qmaps}
     S_{g,n,X,\beta,\sigma}^{\ccQ} \coloneqq \Spec_{\ffP^{st,\ccQ}_{g,n,X,\beta,\sigma}} \Symm\left(\oplus_{\rho \in \Sigma(1)} R^1\pi_\ast (\ffL_\rho^\vee \otimes \omega_\pi)\right).
    \end{equation}
    Finally, there is an open embedding 
    \begin{equation}\label{eq:open_embedding_stable_qmaps}
        \overline{\ccQ}_{g,n}(X,\beta)\hookrightarrow S^{\ccQ}_{g,n,X,\beta,\sigma}
    \end{equation}
    cut out by the generic non-degeneracy condition. 
\end{remark}

\subsection{Components of the ambient abelian cone $S_{0,n,X,\beta,\sigma}$}\label{subsec:comps_ambient_cone}

We fix $g=0$. Our next goal is to extract information about the irreducible components of the abelian cone $S=S_{0,n,X,\beta,\sigma}$ using the results in \Cref{sec:irreducible_comps_cone}.
For that, by \Cref{prop:results_for_stacks} it suffices to show that the stratification of $\Pic_{0,n,X,\beta,\sigma}^{st}$ indexed by stable decorated dual trees in \Cref{cor:stratification_prod_Pic} satisfies  \Cref{assumption:stratification_stacks} and that the following sheaf has local projective dimension at most 1. 

Let $\ffF=\oplus_{\rho \in \Sigma(1)} R^1\pi_\ast (\ffL_\rho^\vee \otimes \omega_\pi)$, which is a coherent sheaf on $B=\Pic_{0,n,X,\beta,\sigma}^{st}$. 
By \Cref{eq:rank_is_h0}, cohomology and basechange ensures that for $P=(C,p_1,\ldots, p_n,(L_\rho)_{\rho\notin\sigma(1)})\in \Pic^{st}_{0,n,X,\beta,\sigma}$,
\begin{equation}\label{eq:rank_is_sum_h0}
    \rk(\ffF\mid_P) = \sum_{\rho\in\Sigma(1)}h^0(C,L_{\rho}).
\end{equation}
As in \Cref{eq:generic_rank_is_h0_P1}, $\beta\neq 0$ implies that
\begin{equation}\label{eq:generic_rank_is_sum_h0_P1}
    \grk_{\Pic^{st}_{0,n,X,\beta,\sigma}}(\ffF) = 
    \grk_{\Pic^{sm}_{0,n,X,\beta,\sigma}}(\ffF) = 
    \sum_{\rho \in \Sigma(1)} h^0(\bbP^1,\ccO_{\bbP^1}(\beta \cdot D_\rho)).
\end{equation} 
The right-hand side of \Cref{eq:generic_rank_is_sum_h0_P1} can be computed easily using \Cref{lem:dim_cohomology}.

Next, we generalize \Cref{notation:i_G and d_G for single line bundle}.
\begin{notation}\label{notation:def_iG_dG}
    Fix $\GG\in\Gamma_{0,n}^{st}(X,\beta)$. Recall that, by \Cref{notation:h1_decorated_graph}, we write $h^0(\GG,L_{\rho_0})=h^0(C,L_{\rho_0})$ for any $(C,p_1\ldots,p_n,(L_\rho)_{\rho \notin \sigma(1)}) \in \Pic_{\GG}$ and any $\rho_0 \in \Sigma(1)$.
    Let 
    \begin{equation}\label{eq:definition_iG_prod}
        i_{\GG} \coloneqq \sum_{\rho \in \Sigma(1)} h^0(\GG,L_\rho) - \sum_{\rho \in \Sigma(1)} h^0(\bbP^1,\ccO_{\bbP^1} (\beta \cdot D_{\rho})),
    \end{equation}
    and
    \begin{equation}\label{eq:definition_dG_prod}
        d_{\GG} \coloneqq i_{\GG} - \codim_{\Pic^{st}_{0,n,X,\beta,\sigma}}(\Pic_{\GG}) = i_{\GG} - \lvert E(\GG) \rvert.
    \end{equation}
\end{notation}

\begin{proposition}\label{lem:stratification_prod_Pic_satisfies_assumptions}
    Let $X$, $\sigma\in\Sigma$ and $\beta \in A_1^{\eff}(X)$ satisfying \Cref{assumption:toric}.
    Let $\ffF=\oplus_{\rho \in \Sigma(1)} R^1\pi_\ast (\ffL_\rho^\vee \otimes \omega_\pi)$ and $B=\Pic_{0,n,X,\beta,\sigma}^{st}$. 
    The stratification
    \[
         \Pic_{0,n,X,\beta,\sigma}^{st} = \bigsqcup_{\GG \in \Gamma_{0,n}^{st}(X,\beta)/_\simeq} \Pic_{\GG}.
    \]
    satisfies \Cref{assumption:stratification_stacks}.
    Furthermore, if $\beta \cdot [D_\rho] \geq 0$ for all $\rho \in \Sigma(1)$ then $\pdim(\ffF)\leq 1$.
\end{proposition}

\begin{proof}
    The stratification satisfies \Cref{assumption:stratification_stacks} by the same argument used in \Cref{lem:stratification_base_SL_satisfies_assumptions}. The statement about the local projective dimension follows from \Cref{rmk:k=0_locally_implies_local_pdim_1} by the argument used in \Cref{prop:R1_has_pdim_1}. More precisely,Sequence      \eqref{eq:long_exact_sequence_R0L_dualized} gives a 3-term presentation of each summand of $\ffF$, whose direct sum satisfies \Cref{rmk:k=0_locally_implies_local_pdim_1}. This fact follows from the argument explained after Sequence \eqref{eq:long_exact_sequence_R0L_dualized}, based on \Cref{eq:h0_decreases_like_degree}, which requires that $\beta \cdot [D_\rho] \geq 0$ for all $\rho \in \Sigma(1)$.
\end{proof}

\begin{notation}
    Let $\pi$ denote the natural projection $S = S_{0,n,X,\beta,\sigma} \to \ffP^{st}_{g,n,X,\beta,\sigma}$, we write
    \[
        S_{\GG} \coloneqq \pi^{-1}(\Pic_{\GG}).
    \]
    Following \Cref{not:closure}, we write $\overline{S}_{\GG}$ for $\cl_{S_{0,n,X,\beta,\sigma}}(S_{\GG})$.
\end{notation}

By \Cref{lem:stratification_prod_Pic_satisfies_assumptions}, the abelian cone $S$ satisfies the assumptions of \Cref{prop:results_for_stacks}. This implies that that all irreducible components of $S$ are of the form $\overline{S}_{\GG}$ for some stable weighted tree $\GG\in \Gamma^{st}_{0,n}(X,\beta)$, by \Cref{corollary:structure_components_CBF}. Furthermore, the numbers $d_{\GG}$ can be used, through \Cref{prop:numerical_sufficient_criterion_component,prop:same_i_implies_inclusion,prop:containment_main_component}, to determine when certain candidate components $\overline{S}_{\GG}$ are or are not actual irreducible components. Finally, we denote by $\GG_0 \in \Gamma_{0,n}(X,\beta)$ the unique decorated graph with a single vertex and no edges. Note that $\GG_0$ is stable because we assume that $\beta$ is non-zero and that $d_{\GG_0} = 0$ by construction. Furthermore, $\overline{S}_{\GG_0}$ is the main irreducible components of $S$ by \Cref{lemma:main_is_component}.

\subsection{Components of genus-0 stable maps}\label{subsec:comps_maps}

Our goal is to describe the irreducible components of $\overline{\ccM}_{0,n}(X,\beta)$. To avoid trivial cases, we consider the following Assumption.

\begin{assumption}\label{assumption:stable_maps_non_empty}
    From now on, we implicitly assume that $n\geq 3$ or $\beta\neq 0$ to ensure that $\overline{\ccM}_{0,n}(X,\beta)$ is not empty.
\end{assumption}

We observe that if $\beta = 0$ (and $n\geq 3$) then $\overline{\ccM}_{0,n}(X,\beta) \simeq \overline{\ccM}_{0,n}\times X$ is clearly irreducible. Thus, we will often simply assume that $\beta \neq 0$.

Our strategy is to fix $X$, $\sigma$ and $\beta$ satisfying \Cref{assumption:toric} to relate the irreducible components of $\overline{\ccM}_{0,n}(X,\beta)$ and the irreducible components of the ambient cone $S_{0,n,X,\beta,\sigma}$, via the open Embedding \eqref{eq:open_embedding_stable_maps}. 
For that, we impose the technical assumption that $\beta \cdot[D_\rho]\geq 0$ for all $\rho \in \Sigma(1)$. This ensure that the sheaf defining the ambient cone $S$ has projective dimension at most 1 (\Cref{def:local_pdim}), thus we can use all the results in \Cref{sec:irreducible_comps_cone}.

\begin{notation}\label{notation:MG_and_MG_bar}
    Let $\pi$ denote the composition 
    \[
        \overline{\ccM}_{0,n}(X,\beta)\hookrightarrow S_{0,n,X,\beta,\sigma} \to \Pic^{st}_{0,n,X,\beta,\sigma}.
    \]
    For $\GG\in \Gamma^{st}_{0,n}(X,\beta)$, we write
    \[
        \ccM_{\GG} \coloneqq \pi^{-1}(\Pic_{\GG}) = S_{\GG} \cap \overline{\ccM}_{0,n}(X,\beta).
    \] 
    We write
    \[
        \overline{\ccM}_{\GG} = \cl_{\overline{\ccM}_{0,n}(X,\beta)} (\ccM_{\GG}).
    \]
\end{notation}

\begin{notation}
    Given $n\in \bbZ_{\geq 0}$ and $\beta \in A_1(X)$ satisfying \Cref{assumption:stable_maps_non_empty}, we denote by $\GG_0\in \Gamma^{st}_{0,n}(X,\beta)$ the decorated marked tree with a single vertex (and no edges). 
\end{notation}

The following result summarizes our knowledge of the irreducible components of $\overline{\ccM}_{0,n}(X,\beta)$.

\begin{theorem}\label{thm:components_maps_collection}
    Let $n\geq 0$ and let $X$, $\sigma\in\Sigma$ and $\beta \in A_1^{\eff}(X)$ satisfying \Cref{assumption:toric}. Then the following properties of irreducible components of $\overline{\ccM}_{0,n}(X,\beta)$ hold.
    \begin{enumerate}[(i)]
        \item\label{item:components_maps_description} The irreducible components of $\overline{\ccM}_{0,n}(X,\beta)$ are
    \begin{equation}
        \left\{\overline{\ccM}_{\GG}\colon 
        \GG\in \Gamma_{0,n}^{st}(X,\beta)/_\simeq, \ccM_{\GG}\neq \emptyset, \overline{S}_{\GG} \text{ is an irreducible component of } S_{0,n,X,\beta,\sigma}\right\}
    \end{equation}
        Furthermore, $\overline{S}_{\GG}$ is an irreducible component of $S_{0,n,X,\beta,\sigma}$ if and only if $S_{\GG} \not\subseteq \overline{S}_{\GG'}$ for any contraction $\GG' \in \Gamma^{st}_{0,n}(X,\beta)$ of $\GG$ with $\GG \neq \GG'$.
    
        \item\label{item:components_maps_main} If $\ccM_{\GG_0}$ is non-empty, then $\overline{\ccM}_{\GG_0}$ is an irreducible component of $\overline{\ccM}_{0,n}(X,\beta)$.
    \end{enumerate}
    Furthermore, if $\beta \cdot [D_\rho]\geq 0$ for all $\rho \in \Sigma(1)$, then 
    \begin{enumerate}[(i)]
    \setcounter{enumi}{2}
        \item\label{item:components_maps_main_smallest_dimension} 
        $
        \dim Z \geq \dim \overline{\ccM}_{\GG_0}
        $ 
        for every irreducible component $Z$ of $\overline{\ccM}_{0,n}(X,\beta)$.
    \end{enumerate}
    and for any $\GG \in \Gamma_{0,n}^{st}(X,\beta)$ such that $\ccM_{\GG}\neq \emptyset$, the following properties hold: 
    \begin{enumerate}[(i)]
        \setcounter{enumi}{3}
        \item\label{item:components_maps_sufficient_numerical_condition}  If $d_{\GG}\geq d_{\GG'}$ for every contraction $\GG' \in \Gamma_{0,n}^{st}(X,\beta)$, then $\overline{\ccM}_{\GG}$ is an irreducible component.
        \item \label{item:components_maps_negative_same_i} If there exists a contraction contraction $\GG' \in \Gamma_{0,n}^{st}(X,\beta)$ with $i_{\GG} = i_{\GG'}$ then $\overline{\ccM}_{\GG} \subseteq \overline{\ccM}_{\GG'}$. In particular, $\overline{\ccM}_{\GG}$ is not an irreducible component.
        \item\label{item:components_maps_smaller_0} If $d_{\GG} < d_{\GG_0} = 0$ then $\overline{\ccM}_{\GG}$ is not an irreducible component.
        \item\label{item:components_maps_bigger_0} If $d_{\GG}\geq d_{\GG_0} = 0$ then $\overline{\ccM}_{\GG} \not \subset \overline{\ccM}_{\GG_0}$
    \end{enumerate}
\end{theorem}

\begin{proof}
    Using the open Embedding \eqref{eq:open_embedding_stable_maps}, we observe that the irreducible components of $\overline{\ccM} = \overline{\ccM}_{0,n}(X,\beta)$ are precisely the irreducible components of the ambient abelian cone $S = S_{0,n,X,\beta,\sigma}$ (in Definition \eqref{eq:definition_ambient_cone}) which intersect $\overline{\ccM}$. By \Cref{lem:stratification_prod_Pic_satisfies_assumptions,prop:results_for_stacks}, we can apply the results in \Cref{sec:irreducible_comps_cone} to determine the irreducible components of $S$, thus also of $\overline{\ccM}$, as we do next.
    
    \Cref{item:components_maps_description} is a consequence of \Cref{corollary:structure_components_CBF}, using \Cref{rmk:inclusion_is_edge_contraction} and \Cref{lem:containment_tentative_components_numerical_implication}.
    For \Cref{item:components_maps_main}, $\overline{\ccM}_{\GG_0}$ is a component if it is non-empty by \Cref{lemma:main_is_component}. \Cref{item:components_maps_main_smallest_dimension} follows from \Cref{thm:dimension_main_is_lower_bound} and \Cref{lem:stratification_prod_Pic_satisfies_assumptions}.
    Finally, \Cref{item:components_maps_sufficient_numerical_condition,item:components_maps_negative_same_i,item:components_maps_smaller_0,item:components_maps_bigger_0} follow immediately from \Cref{prop:numerical_sufficient_criterion_component,prop:same_i_implies_inclusion,prop:containment_main_component}.
\end{proof}

\begin{remark}\label{rmk:condition_MG_nonempty}\label{rmk:only_irreducible_classes_intersect_maps}
    The condition $\ccM_{\GG}\neq \emptyset$ in \Cref{thm:components_maps_collection} is not automatic. 
    Given $\GG\in\Gamma_{0,n}^{st}(X,\beta)$, there are at least two simple necessary conditions for $\ccM_{\GG}$ to be non-empty. 
    \begin{enumerate}[(i)]
        \item\label{item:MG_empty_not_mic} $\GG$ must be mic (see \Cref{def:mic_class}). Indeed, if $(C,f)\in\overline{\ccM}_{0,n}(X,\beta)$ and $C'\subseteq C$ is an irreducible component, then either $f(C')$ is a point or $f\mid_{C'}$ is a degree-$d$ cover of its image, an irreducible curve in $X$. 
        \item\label{item:MG_empty_adjacent_disjoint_curves} For each edge $(v_1,v_2) \in E(\GG)$ with $\beta_{v_1},\beta_{v_2}\neq 0$,  there must exist irreducible curves $C_1$ and $C_2$ on $X$ and $\lambda_1,\lambda_2 \in \bbZ_{>0}$ with $\beta_{v_1} = \lambda_1[C_1]$ and $\beta_{v_2} = \lambda_2[C_2]$ and $C_1\cap C_2\neq \emptyset$. Indeed, if $(C,f)\in\overline{\ccM}_{0,n}(X,\beta)$ and $C', C''\subseteq C$ are two irreducible components with $C'\cap C'' = p$, then $f(p)\in f(C')\cap f(C'') = C_1\cap C_2$. 
    \end{enumerate}
    
    Condition \ref{item:MG_empty_adjacent_disjoint_curves} can be generalized as follows. If $\ccM_{\GG}\neq \emptyset$ for some $\GG \in \Gamma^{mic}_{0,n}(X,\beta)$   and there is a path $v_1,\ldots, v_n$ in $\GG$ such that $\beta_{v_2} = \ldots = \beta_{v_{n-1}} =0$ then there must exist curves $C_1$ and $C_n$ on $X$ and $\lambda_1,\lambda_n \in \bbZ_{>0}$ such that $\beta_{v_1} = \lambda_1[C_1]$, $\beta_{v_{n}} = \lambda_n[C_n]$ and $C_1\cap C_n\neq \emptyset$. 
\end{remark}

\begin{remark} \label{rmk:enough_smaller_subset_containing_components}
    Given that \Cref{item:components_maps_description,item:components_maps_sufficient_numerical_condition} in \Cref{thm:components_maps_collection} ultimately rely on \Cref{lem:topological_lemma}, one can check that these results still hold if we replace $\Gamma_{0,n}^{st}(X,\beta)$ by any smaller subset which still contains all the irreducible components of $\overline{\ccM}$, such as $\Gamma_{0,n}^{mic}(X,\beta)$ (see \Cref{rmk:condition_MG_nonempty}). We shall further pursue this idea in \Cref{sec:Bl_P2} to identify the irreducible components of $\overline{\ccM}_{0,n}(\Bl_{\pt}\bbP^2,\beta)$ for certain classes $\beta$ of ``low degree''.
\end{remark}

\subsection{Components of genus-0 stable quasimaps}\label{sec:components_qmaps}

As in \Cref{subsec:comps_maps}, components of $\ccQ_{0,n}(X,\beta)$ can be determined from the ambient abelian cone in Embedding \eqref{eq:open_embedding_stable_qmaps}. The main difference is that we need to replace the stability condition defined by the line bundle \eqref{eq:stability_stable_maps} by the one defined by line bundle \eqref{eq:stability_stable_qmaps}, as explained in \Cref{rmk:quasimap_stability}. 

Because of this change of stability condition, we must also change the notion of stability for decorated marked trees. Concretely,  $\GG \in \Gamma_{0,n}(X,\beta)$ is stable for quasimaps if and only if it satisfies both Stability Condition \eqref{eq:stability_decorated_trees} and also that 
\begin{equation}\label{eq:quasimap_stability_extra_condition}
    \lvert m^{-1}(v)\rvert + \mathrm{val}(v) \geq 2 \ \text{ for all } v \in V(G).
\end{equation}
We denote by $\Gamma^{st,\ccQ}_{0,n}(X,\beta)$ the set of all stable (in the sense of quasimaps) $\beta$-decorated $n$-marked trees.
In particular, besides \Cref{assumption:stable_maps_non_empty}, when we work with quasimaps we assume that $n\geq 2$. Otherwise, the space is empty.

Following \Cref{notation:MG_and_MG_bar}, we let $\pi$ denote the projection $\ccQ_{0,n}(X,\beta) \to \Pic_{0,n,X,\beta,\sigma}$ and, for $\GG\in \Gamma^{st,\ccQ}_{0,n}(X,\beta)$, we write
    \[
        \ccQ_{\GG} \coloneqq \pi^{-1}(\Pic_{\GG}) = S^{\ccQ}_{\GG} \cap \overline{\ccM}_{0,n}(X,\beta)
    \] 
and
    \[
        \overline{\ccQ}_{\GG} = \cl_{\ccQ_{0,n}(X,\beta)} (\ccQ_{\GG}).
    \]
Then, the analogue of \Cref{thm:components_maps_collection} holds for $\ccQ_{0,n}(X,\beta)$.

\begin{remark}\label{rmk:qmaps_counterexample_empty_mic}
    The fact that $\ccM_{\GG} = \emptyset$ unless $\GG$ is mic (\Cref{item:MG_empty_not_mic} of \Cref{rmk:condition_MG_nonempty}) is not true for quasimaps. In fact, the main component of $Q_{0,n}(\Bl_{\pt}\bbP^2,\beta)$ is not empty for any effective non-mic curve class $\beta$ and any $n\geq 3$. 
    
    Indeed, by \Cref{subsec:geometry_Bl} (see also \Cref{fig:Mori_cone}), an effective non-mic curve classes $\beta \in A_1(\Bl_{\pt}\bbP^2)$ is of the form $\beta = \lambda \ell + \mu e$ with $\lambda, \mu \in \bbZ_{\geq 1}$. To prove our claim, it suffices to find sections $s_i\in H^0(\bbP^1,\ccO(\beta \cdot [D_i]))$ for $i\in \{0,1,2,3\}$, where $(\beta \cdot [D_i])_i = (\lambda, \lambda+\mu,\lambda+\mu,-\mu)$ by \Cref{table:multiplication}, such that at most one of $\{s_0,s_3\}$ and at most one of $\{s_1,s_2\}$ is identically 0 (this is equivalent to generic non-degeneracy). For example, we can take the quasimap
    \[
         q\colon \bbP^1  \dashrightarrow \Bl_{\pt}\bbP^2 \colon [x\colon y]\mapsto [x^{\lambda}\colon x^{\lambda+\mu}\colon x^{\lambda+\mu}\colon 0].
    \]
    Adding $n$ marked points different from the basepoint $[0\colon 1] \in \bbP^1$, $q$ defines an element in the main component of $Q_{0,n}(\Bl_{\pt}\bbP^2,\beta)$.
\end{remark}

\section{Irreducible components of $\overline{\ccM}_{0,n}(\Bl_{\pt}\bbP^2,\beta)$}\label{sec:Bl_P2}

We illustrate  \Cref{thm:components_maps_collection} by determining the irreducible components of $\overline{\ccM}_{0,n}(\Bl_{\pt}\bbP^2,\beta)$ for $\beta$ of ``low degree''.

We begin with a quick review of the geometry of $\Bl_{\pt}\bbP^2$ and showing that nef toric boundary divisors do not play a role in the description of extra irreducible components in \Cref{subsec:geometry_Bl,subsec:reduction_Le}, so we only need to worry about the exceptional divisor. Then, we determine the curve classes $\beta$ for which stable maps is irreducible in \Cref{thm:irreducibility_BlP2}. For the remaining curve classes, our strategy is to find a subset of $\Gamma_{0,n}^{mic}(\Bl_{\pt}\bbP^2,\beta)/\simeq$ which still indexes all irreducible components of $\overline{\ccM}_{0,n}(\Bl_{\pt}\bbP^2,\beta)$, but of smaller size. To do so, through \Cref{subsec:structural} we identify certain properties that a decorated tree $\GG$ must verify for $\overline{\ccM}_{\GG}$ to be an irreducible component, which we combine in the Definition of tentative component \ref{def:tentative_component}. To conclude, we combine \Cref{thm:criteria_tentative_component} (an upgrade of \Cref{thm:components_maps_collection} for tentative components) and the tools in \Cref{subsec:h0} to determine all the irreducible components for $\beta = as+be$ with $a\geq b$ and $b=2,3,4$ \Cref{thm:classification_2l,thm:classification_3l,thm:classification_4l}.

Through this section, we apply certain results from \Cref{sec:irreducible_comps_cone}, stated from schemes, to stacks. This is justified by \Cref{prop:results_for_stacks} and \Cref{lem:stratification_prod_Pic_satisfies_assumptions}.

\subsection{Background: the geometry of $\Bl_{\pt}\bbP^2$}\label{subsec:geometry_Bl}

We consider the toric variety $X=\Bl_\pt \bbP^2$ with the fan in \Cref{fig:fan_BlP2}. The ray generators are
\begin{align*}
    u_0=(-1,-1),\    u_1=(1,0),\     u_2= (0,1),\     u_3= (1,1).
\end{align*}

\begin{figure}
    \centering
			\begin{tikzpicture}[scale = .8]
				\fill[fill=gray!40]  (0,0) -- (2,0) -- (2,2);
                \fill[fill=gray!40]  (0,0) -- (2,2) -- (0,2);
				\fill[fill=gray!40]  (0,0) -- (0,2) -- (-2,2) -- (-2,-2);
				\fill[fill=gray!40]  (0,0) -- (2,0) -- (2,-2) -- (-2,-2);
				\draw[thick] (0,0)--(0,2);
				\draw[thick] (0,0)--(2,0);
				\draw[thick] (0,0)--(-2,-2);
                \draw[thick] (0,0)--(2,2);
    		\draw (2.5,0) node[align=center] {$\rho_{1}$};
    		\draw (0,2.5) node[align=center] {$\rho_{2}$};
    		\draw (-2.5,-2.5) node[align=center] {$\rho_{0}$};
            \draw (2.5,2.5) node[align=center] {$\rho_{3}$};
            \draw (1.5,0.75) node[align=center] {$\sigma_{1,3}$};
            \draw (0.75,1.5) node[align=center] {$\sigma_{2,3}$};
    		\draw (-1,1) node[align=center] {$\sigma_{0,2}$};
    		\draw (1,-1) node[align=center] {$\sigma_{0,1}$};    
    		\end{tikzpicture}
    		\caption{The fan of $\Bl_{\pt}\bbP^{2}$.}
        \label{fig:fan_BlP2}
\end{figure}
We let $\ell\in A_1(X)$ denote the total transform of the class of a line in $\bbP^2$ and $e\in A_1(X)$ denote the class of the exceptional divisor. We let $s \coloneqq \ell-e$, which is the class of the strict transform of a line. The intersection numbers of the classes $s, e, \ell$ are collected in \Cref{table:multiplication}.
\begin{table}
    \centering
    \begin{tabular}{c|c|c|c}
      &s &e&$\ell$  \\
     \hline s    & 0 &1&1  \\
     \hline e   & 1&-1&0 \\
     \hline $\ell$ &1&0&1
    \end{tabular}
    \caption{Multiplication table in $A_1(X)$}
    \label{table:multiplication}
\end{table}
In particular, 
\[
    A_1(X) = \bbZ s \oplus \bbZ e.
\]
The Mori cone of $X$ is
\[
    \mori(X) = \bbR_{\geq 0}\, s + \bbR_{\geq 0}\, e.
\]
By \cite[Corollary 2.18]{Hartshorne_AG}, the set of mic classes (see \Cref{def:mic_class}) on $X$, pictured in \Cref{fig:Mori_cone}, is
\[
    A_1^{mic}(X) = \{as+be \in \mori(X) \colon a=0 \text{ and } b\geq 0; \text{ or } a\geq b\geq 0 \}.
\]

\begin{remark}\label{rmk:main_M_0n_nonempty_iff_mic}
    Under \Cref{assumption:stable_maps_non_empty}, the main component of $\overline{\ccM}_{0,n}(\Bl_\pt \bbP^2,\beta)$ is non-empty if and only if $\beta \in A_1^{mic}(X)$. 
    Indeed, $\beta$ must be mic by \Cref{item:MG_empty_not_mic}
    The converse follows from \cite[Corollary 2.18]{Hartshorne_AG}.
\end{remark}

\begin{figure}
    \centering
    \begin{tikzpicture}
        \fill[fill=gray!20] (0,0)--(4,0)--(4,4);
        \draw[step=1cm,gray,very thin] (-1,-1) grid (4,4);
        \draw (0,0) --(4,0);
        \draw (0,0)--(0,4);
        \draw   (0,0) -- (4,4);
        \draw  (1,0) node[anchor=north west] {$s$};
        \draw  (0,1) node[anchor=north east] {$e$};
        \draw  (1,1) node[anchor=north west] {$\ell$};
        \draw  (0,0) node[anchor=north east] {$0$};
        \draw (0,0) node{$\textcolor{white}{\bullet}$};
        \draw (0,0) node{$\circ$};
        \draw (0,1) node{$\bullet$};
        \draw (0,2) node{$\bullet$};
        \draw (0,3) node{$\bullet$};
        \draw (0,4) node{$\bullet$};
        \draw (1,0) node{$\blacklozenge$};
        \draw (2,0) node{$\blacklozenge$};
        \draw (3,0) node{$\blacklozenge$};
        \draw (4,0) node{$\blacklozenge$};
        \draw (1,1) node{$\blacklozenge$};
        \draw (2,1) node{$\blacklozenge$};
        \draw (3,1) node{$\blacklozenge$};
        \draw (4,1) node{$\blacklozenge$};
        \draw (2,2) node{$\times$};
        \draw (3,2) node{$\times$};
        \draw (4,2) node{$\times$};
        \draw (3,3) node{$\times$};
        \draw (4,3) node{$\times$};
        \draw (4,4) node{$\times$};
    \end{tikzpicture}
    \caption{
    The set $A_1^{mic}(X)$ (\Cref{def:mic_class}) is the union of the vertical line $\bbZ_{\geq 0}\, e$ and the gray cone $\bbZ_{\geq 0}\, s + \bbZ_{\geq 0}\, \ell$. The moduli space $\overline{\ccM}_{0,n}(\Bl_{\pt}\bbP^2,\beta)$ is irreducible for $\beta$ represented by $\bullet$ or $\blacklozenge$, and reducible for $\beta$ represented by $\times$. For  $\beta = 0$, the space is irreducible if $n\geq 3$, and empty otherwise. See \Cref{thm:irreducibility_BlP2}.}
    \label{fig:Mori_cone}
\end{figure}

By \Cref{subsec:functor_toric_variety} and \Cref{rmk:sigma_collection_equivalent_without_cm}, with the choice $\sigma=\sigma_{0,2}$, giving a morphism $f\colon S\to X$ is equivalent to giving two line bundles $L_s=f^\ast \ccO_X(D_1)$ and $L_e = f^\ast \ccO_X(D_3)$ and four sections 
\[
    s_0\in H^0(C, L_s\otimes L_e),\ s_1, s_2\in H^0(C, L_s),\ s_3\in H^0(C, L_e).
\]
satisfying the non-degeneracy condition:
$s_0$ and $s_3$ cannot vanish simultaneously, and $s_1$ and $s_2$ cannot vanish simultaneously.

\subsection{Reduction to the line bundle $\ffL_e$}\label{subsec:reduction_Le}

To study the moduli space $\overline{\ccM}_{0,n}(\Bl_{\pt}\bbP^2,\beta)$, we use the open Embedding \eqref{eq:open_embedding_stable_maps}. Concretely, 
we let $\sigma = \sigma_{0,2}$
and we let $\ffP^{st} = \Pic^{st}_{0,n,X,\beta,\sigma_{0,2}}$(see \Cref{notation:product_Pics_stable}). We have an open embedding
\begin{equation}\label{eq:embedding_maps_blowup}
    \overline{\ccM}_{0,n}(\Bl_{\pt}\bbP^2,\beta)
    \hookrightarrow S_{0,n,X,\beta,\sigma_{0,2}} = S((\ffL_s\otimes \ffL_e)\oplus \ffL_s^{\oplus 2}\oplus \ffL_e)
\end{equation}
by Embedding \eqref{eq:open_embedding_stable_maps}, where we use the notation introduced in Definition \eqref{eq:definition_SL}. 

The following argument reduces the description of irreducible components of $S_{0,n,X,\beta,\sigma_{0,2}}$ to those of $S(\ffL_e)$.
We saw in \Cref{eq:rank_is_h0} that the rank of $R^1\pi_\ast(\ffL_s^\vee \otimes \omega_\pi)$ at $(C,p_1,\ldots, p_n,L_e,L_s)\in \ffP^{st}$ is $h^0(C,L_s)$. Since $s$ is nef, it follows from \Cref{lem:non-negative-degrees_implies_no_h1} that $h^1(C,L_s)=0$, thus $\rk(R^1\pi_\ast (\ffL_s^\vee\otimes \omega_\pi))$ is constant on $\ffP^{st}$, so $S(\ffL_s)$ is a vector bundle over $\ffP^{st}$. The same argument applies if we replace $\ffL_s$ by $\ffL_s \otimes \ffL_e$, as $s+e=\ell$ is nef too.
Therefore, the extra irreducible components of the ambient abelian cone in Embedding \eqref{eq:embedding_maps_blowup} are controlled by $\ffL_e$. 

The previous argument applies to any smooth projective toric variety $X$ to show that components of $\overline{\ccM}_{0,n}(X,\beta)$ are controlled by the sheaves $R^1\pi_\ast(\ffL_\rho^\vee\otimes\omega_\pi)$ for $\rho$ not nef.

\begin{lemma}\label{lem:non-negative-degrees_implies_no_h1}
    Let $(C,L) \in \Pic_{0,n,d}$ such that $\deg(L\mid_{C_i})\geq 0$ for every irreducible component of $C_i$ of $C$. Then $H^1(C,L)=0$.
\end{lemma}

\begin{proof}
    Let $G$ be the dual graph of $C$. The assumption on the degrees ensures that $H^1(C_v,L\mid_{C_v}) = 0$ for all $v \in V(G)$ by \Cref{eq:dim_h0}. The normalization sequence gives us an exact sequence
    \[
        \begin{tikzcd}[column sep = small]
                0\arrow{r} & H^0(C,L) \arrow{r} & \bigoplus_{v\in V(G)} H^0(C_v,L\mid_{C_v}) \arrow{r}{ev} & \bigoplus_{e \in E(\GG)} \bbC \cdot p_e \arrow{r} & H^1(C,L) \arrow{r} & 0.
        \end{tikzcd}
    \]
    Given any $v \in V(\GG)$, any point $p_v \in C_v$ and any $\lambda_v\in\bbC$, the assumption $\deg(L\mid_{C_v}) \geq 0$ ensures that there exist $s_v \in H^0(C_v,L\mid_{C_v})$ with $s_v(p_v) = \lambda_v$. Since the dual graph of $C$ is a tree, one can use this observation to conclude that $ev$ is surjective, thus $H^1(C,L) = 0$.
\end{proof}

\subsection{Irreducibility of $\overline{\ccM}_{0,n}(\Bl_{\pt}\bbP^2,\beta)$}\label{subsec:irreducible_cases}

To begin our study of the irreducible components of $\overline{\ccM}_{0,n}(\Bl_{\pt}\bbP^2,\beta)$, we characterize the 
curve classes $\beta$ for which the space is irreducible. 

The easiest case is to take $\beta$ a multiple of the exceptional divisor $e$, which forces every irreducible component of the source curve to factor through the exceptional divisor $E$.

\begin{proposition}\label{prop:b_0_irreducible} 
    Let $\beta = be$ with $b> 0$ or $n\geq 3$. Then $\overline{\ccM}_{0,n}(\Bl_{\pt}\bbP^2,\beta)$ is irreducible.
\end{proposition}

\begin{proof}
    If $b=0$, $\overline{\ccM}_{0,n}(\Bl_{\pt}\bbP^2,0) \simeq \overline{\ccM}_{0,n} \times \Bl_{\pt}\bbP^2$ is clearly irreducible. If $b>0$, all stable maps of degree $be$ must factor through the exceptional divisor $E\simeq \bbP^1$. Therefore
    \[
        \overline{\ccM}_{0,n}(\Bl_{\pt}\bbP^2,be) \simeq \overline{\ccM}_{0,n}(\bbP^1,b),
    \]
    which is known to be irreducible.
\end{proof}

By \Cref{prop:b_0_irreducible}, we already understand all the irreducible components for $\beta$ a multiple of the exceptional class $e$. Such classes are represented by $\bullet$ in \Cref{fig:Mori_cone}. Recall that we are only interested in mic classes (\Cref{def:mic_class}) to have a non-empty moduli space (see \Cref{rmk:condition_MG_nonempty}). The remaining mic classes in $\Bl_{\pt}\bbP^2$ are precisely the mic classes satisfying $\beta \cdot e\geq 0$, and they are contained in the grey cone in \Cref{fig:Mori_cone}. This dichotomy, which will be useful through this section, motivates the following definition.

\begin{definition}
    We say that a curve class $\beta \in A_1(\Bl_{\pt}\bbP^2)$ is negative (resp. positive, resp. non-negative) if $\beta \cdot e < 0$ (resp. $\beta \cdot e > 0$, resp. $\beta \cdot e \geq 0$).
\end{definition}

We extend this notations to vertices of decorated trees.

\begin{notation}\label{notation:dv} 
    Let $\GG \in \Gamma_{0,n}(\Bl_{\pt}\bbP^2,\beta)$ and $v \in V(\GG)$, We call 
    \[
        d_v \coloneqq \beta_v\cdot e
    \]
    the \textit{degree of $v$}. Note that $d_v = \deg(L_e\mid_{C_v})$. We say that $v$ is negative (resp. positive, resp. non-negative) if $d_v$ (or equivalently $\beta_v$) is negative (resp. positive, resp. non-negative).
\end{notation}

As previously mentioned, the degree of a vertex induces the following dichotomy.

\begin{remark}\label{rmk:mic_negative_degree_implies_multiple_exceptional}
    For $\GG\in \Gamma^{mic}_{0,n}(\Bl_{\pt}\bbP^2,\beta)$ and $v \in V(\GG)$, it holds that
    \begin{itemize}
        \item if $v$ is negative then $\beta_v = -d_v\, e \in \bbZ_{>0} e$ is a multiple of the exceptional class;
        \item if $v$ is non-negative then $\beta_v \in \bbZ_{\geq 0}s + \bbZ_{\geq 0}\ell$, that is $\beta_v$, lies in the grey cone in \Cref{fig:Mori_cone}.
    \end{itemize}
    In other words, the sign of $v$ determines which irreducible component of $A_1^{mic}(\Bl_{\pt}\bbP^2)$ contains $\beta_v$. 
\end{remark}

\begin{remark}\label{rmk:sum_mic_not_always_mic}
    Let $\beta_1, \beta_2 \in A_1^{mic}(\Bl_{\pt}\bbP^2)$. 
    \begin{enumerate}[(i)]
        \item If $\beta_1$ and $\beta_2$ are negative then $\beta_1 + \beta_2 \in A_1^{mic}(\Bl_{\pt}\bbP^2)$. 
        \item  If $\beta_1$ and $\beta_2$ are non-negative then $\beta_1 + \beta_2 \in A_1^{mic}(\Bl_{\pt}\bbP^2)$
    \end{enumerate} 
    However, $A_1^{mic}(\Bl_{\pt}\bbP^2)$ is not closed under sums as $\ell + e$ is not mic. See \Cref{fig:Mori_cone}. 
\end{remark}

We now prove that for a non-negative mic curve class $\beta \neq 0$, an extra irreducible component $\overline{\ccM}_{\GG}$ must contain at least one vertex $v \in V(\GG)$ with $d_v \leq -2$. 
As a corollary, we deduce that $\overline{\ccM}_{0,n}(\Bl_{\pt}\bbP^2,\beta)$ is irreducible for any  non-negative and non-zero mic class $\beta = as+be$ such that 
$b\leq 1$. These are represented by $\blacklozenge$ in \Cref{fig:Mori_cone}.

\begin{lemma}\label{lem:not_component}
    Let $\beta \in A_1^{mic}(\Bl_{\pt}\bbP^2)$ with $\beta \cdot e\geq 0$. Let $\GG \in \Gamma^{mic}_{0,n}(\Bl_{\pt}\bbP^2,\beta)$ such that $\lvert E(\GG) \rvert> 0$ and $d_v \geq -1$ for all $v\in V(\GG)$. Then $d_{\GG} < d_{\GG_0}$. In particular, $\overline{\ccM}_{\GG}$ is not an irreducible component of $\overline{\ccM}_{0,n}(\Bl_{\pt}\bbP^2,\beta)$.  
\end{lemma}

\begin{proof}
    Since $\beta \cdot e\geq 0$, then $\beta \cdot[D_\rho]\geq 0$ for all $\rho \in \Sigma(1)$ because $s,\ell$ are nef divisors. 
    It suffices to prove that $d_{\GG} < d_{\GG_0}$, by \Cref{item:components_maps_smaller_0} in \Cref{thm:components_maps_collection}.
    
    By \Cref{subsec:geometry_Bl}, a point in $S_{\GG}$ is represented by a tuple $(C,L_s,L_e,s_0,s_1,s_2,s_3)$. We consider the triple $(C,L_e,s_3)$ and denote it by $(C,L,s)$. Note that $s\in H^0(C,L)$ and $\deg(L\mid_{C_v}) = \beta_v\cdot e = d_v$ for each $v \in V(\GG)$. 
    Thus, our assumptions imply that $\deg(L)=\beta \cdot e \geq 0$ and $H^1(C_v,L\mid_{C_v}) = 0$ for all $v \in V(\GG)$ by \Cref{eq:dim_h0}.
    
    Since $[D_0] = \ell$ and $[D_1] = [D_2] = s$ are nef divisor, 
    $h^0(\GG,L_{\rho}) = h^0(\GG_0,L_\rho)$ for $\rho \neq \rho_3$, as explained in \Cref{subsec:reduction_Le}.
    By Riemann-Roch and the definition of $d_{\GG}$ (\Cref{notation:def_iG_dG}), 
    \begin{equation}\label{eq:inequality_main}
        d_{\GG} - d_{\GG_0} = h^0(C,L) - h^0(\bbP^1,\ccO(d)) - \lvert E(\GG) \rvert = h^1(C,L) - \lvert E(\GG) \rvert.
    \end{equation}

    On the other hand, the normalization sequence on $C$ gives us an exact sequence
    \begin{equation}\label{eq:normalization_example}
        \begin{tikzcd}[column sep = small]
            0\arrow{r} & H^0(C,L) \arrow{r} & \bigoplus_{v\in V(G)} H^0(C_v,L\mid_{C_v}) \arrow{r}{ev} & \bigoplus_{e \in E(\GG)} \bbC \cdot p_e \arrow{r} & H^1(C,L) \arrow{r} & 0.
        \end{tikzcd}
    \end{equation}
    In particular, $h^1(C,L) \leq \lvert E(\GG) \rvert$.
    To describe the morphism $ev$, we choose an orientation for every edge. Then for $s_v \in H^0(C_v,L\mid_{C_v})$ and $e=(v_1,v_2) \in E(\GG)$,
    \[
        ev(s_v) = \begin{cases}
            s_v(p_e) & \text{ if } v=v_1,\\
            -s_v(p_e) & \text{ if } v=v_2,\\
            0 & \text{ if } v\neq v_1, v_2.
        \end{cases}
    \]

    It follows from \Cref{eq:inequality_main} and exact Sequence \eqref{eq:normalization_example} that    
    \[
        d_{\GG} - d_{\GG_0} \geq 0 \iff h^1(C,L) = \lvert E(\GG) \rvert \iff ev=0.
    \]
    If $ev = 0$ then $\deg(L\mid_{C_v}) < 0$ for every $v \in V(\GG)$, thus $\beta \cdot e = \deg(L) \leq -2$ since $\lvert V(\GG) \rvert \geq 2$ by the assumption that $\lvert E(\GG) \rvert \geq 1$. This contradicts our assumption, thus $d_{\GG} < d_{\GG_0}$, as we wanted to show. 
\end{proof}

\begin{corollary}\label{cor:tentative_negative-2}
    Let $\beta \in A_1^{mic}(\Bl_{\pt}\bbP^2)$ with $\beta \cdot e\geq 0$. Let $\GG \in \Gamma^{mic}_{0,n}(\Bl_{\pt}\bbP^2,\beta)$ such that $\lvert E(\GG) \rvert> 0$. If $\overline{\ccM}_{\GG}$ is an irreducible component of $\overline{\ccM}_{0,n}(\Bl_{\pt}\bbP^2,\beta)$ then there exists $v \in V(\GG)$ such that $d_v \leq -2$. 
\end{corollary}

\begin{proof}
    Follows immediately from \Cref{lem:not_component}.
\end{proof}

\begin{proposition} \label{prop:b_0_1_irreducible}
    Let $\beta = as + be \neq 0$ be non-negative and mic.
    If $b=0$ or $b=1$, then $\overline{\ccM}_{0,n}(\Bl_{\pt}\bbP^2,\beta)$ is irreducible.
\end{proposition}

\begin{proof}
    By assumption, $\beta \cdot [D_{\rho}]\geq 0$ for all $\rho \in \Sigma(1)$ and $\overline{\ccM}_{\GG_0}\neq \emptyset$ by \Cref{rmk:main_M_0n_nonempty_iff_mic}.
    By \Cref{thm:components_maps_collection} and \Cref{rmk:enough_smaller_subset_containing_components}, every irreducible component is of the form $\overline{\ccM}_{\GG}$ for some $\GG \in \Gamma_{0,n}^{mic}(\Bl_{\pt}\bbP^2,\beta)$. Conversely, if $\overline{\ccM}_{\GG}$ is an irreducible component for some $\GG \in \Gamma_{0,n}^{mic}(\Bl_{\pt}\bbP^2,\beta)$, then $\lvert E(\GG) \rvert = 0$ by  \Cref{lem:not_component}, thus $\GG = \GG_0$.
\end{proof}

To conclude this section, we prove that $\overline{\ccM}_{0,n}(\Bl_{\pt}\bbP^2,\beta)$ is reducible for any non-zero mic curve class $\beta$ of type $\times$ in \Cref{fig:Mori_cone}. These are precisely the non-zero mic curve classes not covered by \Cref{prop:b_0_irreducible} or \Cref{prop:b_0_1_irreducible}.

\begin{proposition}\label{prop:components_G_c_I}
    Let $\beta = as + be$ with $a\geq b\geq 2$. For each $2\leq c \leq b$ and each $I\subseteq \{1,\ldots, n\}$, the stable decorated graph $\GG_{c,I} \in \Gamma^{mic}_{0,n}(\Bl_{\pt}\bbP^2,\beta)$ pictured in \Cref{fig:GG_c_I} induces an irreducible component $\overline{\ccM}_{\GG_{c,I}}$ of $\overline{\ccM}_{0,n}(\Bl_{\pt}\bbP^2,\beta)$.
    In particular, $\overline{\ccM}_{0,n}(\Bl_{\pt}\bbP^2,\beta)$ has at least $(b-1) 2^n + 1$ irreducible components.
\end{proposition}

\begin{figure}
    \centering
    \begin{tikzpicture}
                \draw (0,0) node[] {\textbullet};
                \draw (0,0) node[anchor= south] {$ce$};
                \draw (0,-0.5) node[] {\ldots};
                \draw (0,-0.5) node[anchor= north] {$\{x_i\}_{i \in I}$};
                
                \draw (3,0) node[] {\textbullet};
                \draw (3,0) node[anchor= south] {$as+(b-c)e$};
                \draw (3.5,-0.5) node[] {\ldots};
                \draw (3.5,-0.5) node[anchor= north] {$\{x_i\}_{i \notin I}$};
                
                \draw (0,0) -- (3,0);
                \draw (0,0) -- (-1,-0.5);
                \draw (0,0) -- (-0.5,-0.5);
                \draw (0,0) -- (0.5,-0.5);
                \draw (3,0) -- (3,-0.5);
                \draw (3,0) -- (2.5,-0.5);
                \draw (3,0) -- (4,-0.5);
            \end{tikzpicture}
    \caption{A decorated tree $\GG_{c,I} \in \Gamma^{mic}_{0,n}(\Bl_{\pt}\bbP^2,as+be)$ for all $1\leq c \leq b$ (assuming that $a\geq b\geq 2$). The half-edges indicate the markings, distributed according to $I\subseteq \{1,\ldots, n\}$.}
    \label{fig:GG_c_I}
\end{figure}

\begin{proof}
    Note that the condition $a\geq b \geq 0$ is equivalent to $\beta$ being mic and non-negative.  Similarly, the condition $1\leq c \leq b$ ensures that $\GG_{c,I}$ is mic. Since the main component $\overline{\ccM}_{\GG_0}$ is an irreducible component (\Cref{rmk:main_M_0n_nonempty_iff_mic} and \Cref{item:components_maps_main} in \Cref{thm:components_maps_collection}), it suffices to show that $\overline{\ccM}_{\GG_{c,I}}$ is an irreducible component of $\overline{\ccM}_{0,n}(\Bl_{\pt}\bbP^2,\beta)$ for all $2\leq c\leq b$ and all $I\subseteq \{1,\ldots, n\}$.
    
    By \Cref{thm:components_maps_collection}, $\overline{\ccM}_{\GG_{c,I}}$ is an irreducible component if and only if $\overline{\ccM}_{\GG_{c,I}}\not\subseteq \overline{\ccM}_{\GG_0}$ if and only if $d_{\GG_{c,I}} - d_{\GG_0} \geq 0$.
    
    Using that $(be)\cdot e = -b \leq -2$ and \Cref{eq:dim_h0}, we see that 
    \[
        h^0(\GG_{c,I},L_e) = a+c-b+1-1 = a+c-b = h^0(\GG_0,L_e)+c-1.
    \]
    By \Cref{eq:inequality_main}, we have that
    \[
        d_{\GG_{c,I}} - d_{\GG_0} = h^0(\GG_{c,I},L_e) - h^0(\GG_0,L_e) - \lvert E(\GG_{c,I}) \rvert = c-2,
    \]
    which indeed is greater than or equal to 0 for all $c\geq 2$.
    
    Finally, for fixed $a,b$, we have the main components and $(b-1) 2^n$ possible choices of graphs $\GG_{c,I}$, all of which are non-isomorphic. 
\end{proof}

Collecting the previous results, we characterize the mic curve classes for which $\overline{\ccM}_{0,n}(\Bl_{\pt}\bbP^2,\beta)$ is irreducible. These are precisely the curve classes represented by $\bullet$ or $\blacklozenge$ in \Cref{fig:Mori_cone}, together with $\beta = 0$ (if $n\geq 3$).

\begin{theorem}\label{thm:irreducibility_BlP2}
    Let $\beta \in A_1^{mic}(\Bl_{\pt}\bbP^2)$ and $n\in \bbZ_{\geq 0}$ satisfying \Cref{assumption:stable_maps_non_empty}. 
    Then $\overline{\ccM}_{0,n}(\Bl_{\pt}\bbP^2,\beta)$ is irreducible if and only if one of the following three conditions holds
    \begin{enumerate}[(i)]
        \item $\beta = \lambda e$ for some $\lambda \geq 0$,
        \item $\beta = \lambda s$ for some $\lambda \geq 1$, or
        \item $\beta = \lambda s + e$ for some $\lambda \geq 1$.
    \end{enumerate}
\end{theorem}

\begin{proof}
    This is an immediate consequence of  \Cref{prop:b_0_irreducible,prop:b_0_1_irreducible,prop:components_G_c_I}.
\end{proof}

\subsection{Controlling $h^0(C,L)$} \label{subsec:h0}

Before studying the irreducible components of $\overline{\ccM}_{0,n}(\Bl_{\pt}\bbP^2,\beta)$ when it is reducible in \Cref{subsec:tentative,subsec:general_case}, we take a digression. By \Cref{thm:components_maps_collection}, we are interested in computing $d_{\GG}$ (see  Definitions \eqref{eq:definition_iG_prod} and \eqref{eq:definition_dG_prod}) for a decorated graph $\GG$. Since the exceptional divisor is the only non-nef toric boundary divisor of $\Bl_{\pt}\bbP^2$, computing $d_{\GG}$ reduces to computing $h^0(C,L_e)$ for any element in $\Pic_{\GG}$. For this reason, we devote this section to the computation of $h^0(C,L)$ in terms of the weighted dual graph of $(C,L)\in \Pic_{0,n,d}$.

Since we will phrase our results in terms of the dual graph, we introduce the following convention.

\begin{remark}\label{rmk:choice_elt}
    In the following results,
    when working with $\GG \in \Gamma_{0,n,d}$ we implicitly choose an element $(C,x_1,\ldots, x_n,L)\in \Pic_{\GG}$. We will not refer to the marked points, as they are not relevant in our study of $h^0(C,L)$. One can check that the results are independent of the chosen element. 
\end{remark}

We begin our study of $h^0(C,L)$ with several notations and useful constructions from graph theory and introducing certain restriction maps.

\begin{notation}\label{notation:restrictions}
    Let $\GG \in \Gamma_{0,n,d}$. With the convention in \Cref{rmk:choice_elt}, we shall denote
    \begin{enumerate}[(i)]
        \item  By $C_v$ the irreducible component $C$ indexed by $v \in V(\GG)$. 
        \item By $L_v$ the restriction $L\mid_{C_v}$.
        \item By $p_{(v,v')}$ the node in $C$ indexed by $(v,v')\in E(\GG)$. We simply write $p_{v'}$ if the choice of $v$ is clear.
        \item By $H^0(\GG)$ the vector space $H^0(C,L)$ and by $h^0(\GG)$ the number $h^0(C,L)$.
        \item By $V(\GG)_{v}$ the collection of vertices adjacent to $v$; that is,
        \[
            V(\GG)_{v}\coloneqq \{ v' \in V(\GG)\colon (v,v') \in E(\GG)\}.
        \]
    \end{enumerate}
    
    Furthermore, 
    \begin{enumerate}[(i)]
    \setcounter{enumi}{5}
        \item For $v_1,\ldots, v_n \in V(\GG)$, we denote by $\langle v_1,\ldots, v_n \rangle$ the smallest subtree $T$ of $\GG$ such that $v_1,\ldots, v_n \in V(T)$.
        \item For $v \in V(\GG)$, we denote by $\GG - v$ the forest obtained from $\GG$ by removing the vertex $v$ and all the edges in $E(\GG)$ that contain $v$.
        \item For a forest $F$ and $v\in V(F)$, we let $F_v$ denote the connected component of $v$ in $F$.
        \item For a subtree $T$ of a graph $\GG \in \Gamma_{0,n,d}$ (resp. $\GG \in \Gamma_{0,n}(\Bl_{\pt}\bbP^2,\beta)$), we denote by $\GG/T$ the tree obtained by replacing the subtree $T$ with a vertex $v_T$ of degree $d_{v_T} = \sum_{v \in V(T)} d_v$ (resp. of class $\beta_{v_T} = \sum_{v \in V(T)} \beta_v$). 
        
    \end{enumerate}
    In the previous definition, if $\GG \in \Gamma_{0,n}(\Bl_{\pt}\bbP^2,\beta)$ is stable, then so is $\GG/T \in \Gamma_{0,n}(\Bl_{\pt}\bbP^2,\beta)$; and if $\GG$ is mic then $\GG/T$ is mic if and only if $\beta_{v_T} = \sum_{v \in V(T)} \beta_v$ is mic (see \Cref{rmk:sum_mic_not_always_mic}).
    
    Finally, we consider the following restriction morphisms: 
     \begin{enumerate}[(i)]
     \setcounter{enumi}{9}
         \item $r_{C_v} = r^{\GG}_{C_v} \colon H^0(\GG)\to H^0(C_v,L_v) \colon s \mapsto s\mid_{C_v}$ for any $v \in V(\GG)$ and
        \item $r^{T}_{p_{(v,v')}}\colon H^0(T,L\mid_{T})\to \bbC \cdot p_{(v,v')} \colon s \mapsto s(p_{(v,v')})$ 
     \end{enumerate}
     for any $(v,v')\in E(\GG)$ and any subtree $T$ of $\GG$ with $v\in V(T)$.
\end{notation}

Our starting point to control $h^0(\GG)$ for $\GG \in \Gamma_{0,n,d}$ is the following naive formula:
\begin{equation}\label{eq:naive_formula_h0}
     h^0(\GG) = \sum_{v\in V(\GG)} h^0(C_v) - \lvert E(\GG) \rvert.
\end{equation}
This formula is based on the observation that a global section $s \in H^0(C,L)$ restricts to sections $s_v \in H^0(C_v,L_v)$ which glue at the edges, meaning $s_v(p_{(v,v')}) = s_{v'}(p_{(v,v')})$ for all $(v,v')\in E(\GG)$. If $H^1(\GG) = 0$, then the Naive Formula \eqref{eq:naive_formula_h0} holds as a consequence of the normalization sequence. This is the case, for example, if $d_v\geq 0$ for all $v \in V(\GG)$, see \Cref{lem:non-negative-degrees_implies_no_h1}. However, the Naive Formula \ref{eq:naive_formula_h0} does not hold for all $\GG \in \Gamma_{0,n,d}$, as we show in \Cref{ex:counterexample_naive_formula}.

\begin{example}\label{ex:counterexample_naive_formula}
    The Naive Formula \eqref{eq:naive_formula_h0} does not work for the weighted trees $\GG_1$ or $\GG_2$ in \Cref{fig:counterexamples_naive_formula}. For example, the right-hand side of \eqref{eq:naive_formula_h0} is negative if $d<-2$. Besides those pathological cases, one can check that for $i \in \{1,2\}$ and $d\geq -2$, the right-hand side of \eqref{eq:naive_formula_h0} equals $d+1$, while $h^0(\GG_i) = d+2$. 
    
    In the case $d\geq -2$ and $i=1$, it is easy to explain why the Naive Formula \eqref{eq:naive_formula_h0} does not hold. A global section must vanish entirely on the components $C_{v_2}$ and $C_{v_3}$, but then the gluing at the node $p_{(v_2,v_3)}$ is automatic. So, instead of subtracting $\lvert E(\GG_1) \rvert = 2$, we should only subtract $\lvert \{(v_1,v_2)\} \rvert = 1$ to get the correct number.
    
    A similar argument can be made for $\GG_2$ and $d\geq -2$. Instead, we observe the following similarity between the two graphs: the subtrees $T_1 = \langle v_2, v_3 \rangle \subset \GG_1$ and $T_2 = \langle v_2, v_3, v_4 \rangle \subset \GG_2$ have no global sections. To check that $h^0(\GG_i) = d+2$, it is enough to observe that $h^0(C_{v_1}) = d+3$
    and that $h^0(T_i) = 0$; therefore a global section in $H^0(\GG_i)$ is simply an element of $H^0(C_{v_1},L)$ vanishing at $p_{(v_1,v_2)}$ (that is, an element of  $H^0(C_{v_1},L(-p_{(v_1,v_2)})$) extended to $T_i$ by 0.
\end{example}

\begin{figure}
    \centering
    \begin{tikzpicture}
                \draw (-1,0) node[] {\textbullet};
                \draw (-1,0) node[anchor= south] {$d+2$};
                \draw (-1,0) node[anchor= north] {$v_1$};
                \draw (0,0) node[] {\textbullet};
                \draw (0,0) node[anchor= south] {$-1$};
                \draw (0,0) node[anchor= north] {$v_2$};
                \draw (1,0) node[] {\textbullet};
                \draw (1,0) node[anchor= south] {$-1$};
                \draw (1,0) node[anchor= north] {$v_3$};
                
                \draw (-1,0) -- (1,0);
                \draw (0,-0.5) node[anchor= north] {$\GG_1$};

                \draw (3,0) node[] {\textbullet};
                \draw (3,0) node[anchor= south] {$d+2$};
                \draw (3,0) node[anchor= north] {$v_1$};
                \draw (4,0) node[] {\textbullet};
                \draw (4,0) node[anchor= south] {$0$};
                \draw (4,0) node[anchor= north] {$v_2$};
                \draw (5,0.5) node[] {\textbullet};
                \draw (5,0.5) node[anchor= west] {$-1$};
                \draw (5,0.5) node[anchor= north] {$v_3$};
                \draw (5,-0.5) node[] {\textbullet};
                \draw (5,-0.5) node[anchor= west] {$-1$};
                \draw (5,-0.5) node[anchor= north] {$v_4$};
                
                \draw (3,0) -- (4,0);
                \draw (4,0) -- (5,0.5);
                \draw (4,0) -- (5,-0.5);
                \draw (3.5,-0.5) node[anchor= north] {$\GG_2$};

            \end{tikzpicture}
    \caption{Two graphs in $\Gamma_{0,0,d}$ where the Naive Formula  \eqref{eq:naive_formula_h0} does not hold.}
    \label{fig:counterexamples_naive_formula}
\end{figure}

We aim to generalize \Cref{ex:counterexample_naive_formula} to a formula for $h^0(\GG)$ which is valid for any $\GG \in \Gamma_{0,n,d}$ (\Cref{prop:formula_h0_at_v}). To do so, we first look at certain restriction maps.

\begin{definition}\label{def:supports_no_sections}
    Let $\GG \in \Gamma_{0,n,d}$. We say that a vertex $v \in V(\GG)$ \textit{supports no global sections} if the restriction map $r_{C_v} \colon H^0(\GG)\to H^0(C_v,L_v)$ is the 0 map. This means that the restriction of any global section of $L$ to $C_v$ is 0.
    If $r_{C_v} \neq 0$, we say that $v$ \textit{supports global sections}.
\end{definition}

The easiest example of a vertex which supports no sections is one with negative degree.
\begin{example}\label{ex:negative_degree_supports_no_sections}
    If $d_v < 0$, $v$ supports no sections because $\deg(L_v) = d_v < 0$, thus $H^0(C_v,L_v)=0$.
\end{example}

On the other hand, the graph $\GG_2 \in \Gamma_{0,0,d}$ in \Cref{fig:counterexamples_naive_formula} satisfies that $d_{v_2} \geq 0$ but $v_2$ supports no global sections because it is adjacent to two vertices of negative degree. The following definition formalizes the intuition that, for a vertex $v$ with $d_v\geq 0$, $v$ supports global sections or not depending on the vanishing conditions imposed by its surrounding subtrees.

\begin{definition}\label{def:base}
    Let $\GG \in \Gamma_{0,n,d}$ and $v \in V(\GG)$.
    The \textit{base} of $v$ is
    \[
        \Base(v) = \Base^{\GG}(v) \coloneqq \{v'\in V(\GG)_{v} \colon r^{T_{v'}}_{p_{(v,v')}} \colon H^0(T_{v'},L_{T_{v'}}) \to \bbC \cdot  p_{(v,v')} \text{ is the 0 map}\},
    \]
    where $T_{v'} \coloneqq (\GG - v)_{v'}$. We denote its cardinality by $\base(v) = \base^{\GG}(v) \coloneqq \lvert \Base(v)\rvert$.
\end{definition}

\begin{example}
    For the graphs in \Cref{fig:counterexamples_naive_formula}, it is clear that $v_2 \in \Base^{\GG_1}(v_1)$ because $d_{v_2} < 0$. We also have that $v_2 \in \Base^{\GG_2}(v_1)$ despite $d_{v_2} \geq 0$, because the subtree $T_{v_2} = \langle v_2,v_3,v_4\rangle$ appearing in \Cref{def:base} has no global sections.  
\end{example}

\begin{example}\label{ex:no_sections_implies_base_of_neighbours}
    It follows from \Cref{def:supports_no_sections,def:base} that if $v'$ supports no sections, then $v'\in \Base(v)$ for any $v \in V(\GG)_{v'}$.
\end{example}

The name ``base'' for $\Base(v)$ is justified by the following observation.
\begin{remark}\label{rmk:base_from_basepoint}
    If $v'\in \Base(v)$, then the restriction map $r_{C_v}$ factors as follows:
     \[
    \begin{tikzcd}
        H^0(\GG)\arrow{r}{r_{C_v}}\arrow{dr} &  H^0(C_v,L_v)\\
        & H^0 (C_v, L_v(-p_{(v,v')}))\arrow[hook]{u}
    \end{tikzcd}
    \]
    In other words, every global section of $L$ must vanish at $p_{(v,v')}$, so $p_{(v,v')}$ is a basepoint of $L$. 
\end{remark}

We now give a formula for $h^0(\GG)$ which depends on the choice of a vertex which supports global sections.

\begin{proposition}\label{prop:formula_h0_at_v}\label{prop:base_larger_degree_implies_no_sections}
    Let $\GG \in \Gamma_{0,n,d}$ and $v \in V(\GG)$. Then $v$ supports global sections if and only if $\base(v)\leq d_v$. In that case,
    \begin{equation}\label{eq:formula_h0_at_v}
        h^0(\GG) = h^0(C_v,L_v) - \base(v) + \sum_{v'\in \Base(v)} h^0(T_{v'},L_{T_{v'}}) + \sum_{v'\in V(\GG)_{v}\setminus \Base(v)} (h^0(T_{v'},L_{T_{v'}}) - 1).
    \end{equation}
\end{proposition}

\begin{proof}
    By \Cref{rmk:base_from_basepoint}, the map $r_{C_v}$ must factor through $H^0(C_v,L_v(-\sum_{p\in \Base(v)} p))$. Therefore, $v$ supports no global sections (\Cref{def:supports_no_sections}) if $\base(v) > d_v$. 
    
    Suppose that $\base(v)\leq d_v$ and
    consider a partial normalisation, leading to the exact sequence
    \begin{equation}\label{eq:partial_normalization}
        \begin{tikzcd}[column sep = small]
            0\arrow{r} & H^0(\GG) \arrow{r}{\operatorname{r}} & H^0(C_v) \oplus \displaystyle\bigoplus_{v'\in V(G)_v} H^0(T_{v'}) \arrow{r}{ev} & \displaystyle\bigoplus_{v' \in V(\GG)_{v}} \bbC \cdot p_{(v,v')} \arrow{r} & H^1(\GG)  \arrow{r} & \ldots
        \end{tikzcd}
    \end{equation}
    where $T_{v'} = (\GG - v)_{v'}$ for $v' \in V(\GG)_v$.
    By \Cref{rmk:base_from_basepoint}, the restriction map $r$ in \eqref{eq:partial_normalization} factors through 
    \begin{equation}\label{eq:aux_twisted_middle_term}
         V \coloneqq H^0(C_v,L_v(-\sum_{v' \in \Base(v)} p_{(v,v')})) \oplus \bigoplus_{v'\in \Base(v)} H^0(T_{v'},L_e\mid_{T_{v'}})  \oplus \bigoplus_{v''\in V(G)_v\setminus \Base(v)} H^0(T_{v''},L_e\mid_{T_{v''}})
    \end{equation}
    inducing the following sequence, where we still call $r$ and $ev$ the induced morphisms
    \begin{equation}\label{eq:ses_supporting_vertices}
        \begin{tikzcd}[column sep = small]
            0\arrow{r} & H^0(\GG) \arrow{r}{\operatorname{r}} & V \arrow{r}{ev} & \displaystyle \bigoplus_{v''\in V(\GG)_v \setminus\Base(v)} \bbC \cdot p_{(v,v'')} \arrow{r} & 0
        \end{tikzcd}
    \end{equation}
    
    We claim that \eqref{eq:ses_supporting_vertices} is exact. It is clear that $r$ is injective and that the sequence is exact at $V$ because the gluing is automatic at all edges $(v,v')$ with $v'\in \Base(v)$. Therefore, we only need to prove that $ev$ is surjective. For that, let $v_0 \in V(\GG)_v\setminus \Base(v)$ and let us find an element $s = (s_v,(s_{v'})_{v' \in \Base(v)},(s_{v''})_{v''\notin \Base(v)}) \in V$ such that $ev(s) = 1\cdot p_{(v,v_0)}$. 
    
    Since $\base(v)\leq d_v$ then $\deg(L_v(-\sum_{v' \in \Base(v)} p_{(v,v')}) = d_v-\base(v) \geq 0$. This implies that we can choose $s_v \in H^0(C_v,L_v(-\sum_{v' \in \Base(v)} p_{(v,v')}))$ such that $s_v(p_{v,v''}) \neq 0$ for every $v'' \in V(\GG)_v\setminus \Base(v)$.
    For any $v'\in \Base(v)$ we take $s_{v'} = 0$. For any $v'' \in V(\GG)_v \setminus \Base(v)$, by definition of $\Base(v)$ (\Cref{def:base}) there exists a section $s_{v''} \in H^0(T_{v''})$ not vanishing at $p_{(v,v'')}$. Finally, we rescale the sections $s_{v''}$ as follows: 
    \begin{enumerate}[(i)]
        \item  if $v'' \neq v_0$, we rescale $s_{v''}$ to ensure that $s_{v''}(p_{(v,v'')}) = s_{v}(p_{(v,v'')})$ and
        \item we rescale $s_{v_0}$ so that $s_{v_0}(p_{(v,v_0)}) = s_{v}(p_{(v,v_0)}) + 1$.
    \end{enumerate}
    By construction, the collection $s = (s_v,(s_{v'})_{v' \in \Base(v)},(s_{v''})_{v''\notin \Base(v)}) \in V$ satisfies that $ev(s) = 1\cdot p_{(v,v_0)}$. This proves that \eqref{eq:ses_supporting_vertices} is a short exact sequence. Then Formula \eqref{eq:formula_h0_at_v} follows immediately.
    
    Finally, if in the previous construction of the element $s = (s_v,(s_{v'})_{v' \in \Base(v)},(s_{v''})_{v''\notin \Base(v)}) \in V$ we rescale $s_{v_0}$ to ensure that $s_{v_0}(p_{(v,v_0)}) = s_{v}(p_{(v,v_0)})$, then $ev(s) = 0$. Therefore, $s$ corresponds to a global section of $H^0(\GG)$ which does not vanish at $p_{(v,v_0)} \in C_v$ by construction. This proves that $v$ supports global sections if $\base(v) \leq d_v$ and concludes the proof.
\end{proof}

\begin{example}
    As a sanity check, one can take $\GG_1$ and $\GG_2$ in \Cref{fig:counterexamples_naive_formula} with $v=v_1$. In both cases, the assumption $\base(v_1)\leq d_{v_1}$ is met if and only if $d\geq -1$. In that case, Formula \eqref{eq:formula_h0_at_v} shows that $h^0(\GG_i) = d+2$. For $d\leq -2$, one can see there are no global sections because there is no vertex supporting global sections by \Cref{prop:base_larger_degree_implies_no_sections}. 
\end{example}

\begin{remark}\label{rmk:strategy_recursive_formula_h0}
    \Cref{prop:formula_h0_at_v} can be used to compute $h^0(\GG)$ for any weighted tree $\GG \in \Gamma_{0,n,d}$. Indeed, either no vertex supports global sections, and then $h^0(\GG) = 0$; or 
    we can apply Formula \eqref{eq:formula_h0_at_v} to some vertex $v \in V(\GG)$. Then all the trees appearing in the formula have strictly fewer vertices than $\GG$.
\end{remark}

We conclude this section with a simple lemma and a result that compares $h^0(\GG)$ and $h^0(\GG/T)$ under several assumptions. This result will be useful in \Cref{subsec:structural}.

\begin{lemma}\label{lem:equivalence_being_base}
    Let $\GG \in \Gamma_{0,n,d}$ and $(v,v') \in E(\GG)$. Set $T_{v'} = (\GG - v)_{v'}$. The following are equivalent
    \begin{enumerate}[(i)]
        \item\label{item:is_in_base} $v' \in \Base^{\GG}(v)$,
        \item\label{item:supports_no_sections_after_deleting} $v'$ supports no global sections in $T_{v'}$,
        \item\label{item:degree_base_inequality_after_deleting} $d_{v'} < \base^{T_{v'}}(v')$.
    \end{enumerate}
\end{lemma}

\begin{proof}
    \Cref{item:supports_no_sections_after_deleting,item:degree_base_inequality_after_deleting} are equivalent by \Cref{prop:base_larger_degree_implies_no_sections} applied to $T_{v'}$. \Cref{item:supports_no_sections_after_deleting} implies \Cref{item:is_in_base} by \Cref{ex:no_sections_implies_base_of_neighbours}. Finally, 
    if $v'$ supports global sections in $T_{v'}$, there exists some section in $H^0(T_{v'})$ which is not identically 0 on $C_{v'}$. In particular, we can find $s_{v'} \in H^0(T_{v'})$ not vanishing at $p_{(v,v')}$, thus $v' \notin \Base^{\GG}(v)$. This proves that \Cref{item:is_in_base} implies \Cref{item:supports_no_sections_after_deleting}.
\end{proof}

\begin{corollary}\label{cor:no_sections_subtree_2_vertices}
    Let $\GG \in \Gamma_{0,n,d}$, let $(v,v') \in E(\GG)$ such that $v$ and $v'$ support no sections and set $T = \langle v, v' \rangle \subseteq \GG$. Then $v_T$ supports no sections in $\GG/T$ and $h^0(\GG) = h^0(\GG/T)$.
\end{corollary}

\begin{proof}
    By \Cref{prop:formula_h0_at_v}, to prove that $v_T$ supports no global sections, we need to show that $d_{v} + d_{v'} = d_{v_T} < \base^{\GG/T}(v_T)$.
    We have that
    \[
        (\Base^{\GG}(v) \sqcup \Base^{\GG}(v')) \setminus \{v,v'\}\subseteq \Base^{\GG/T}(v_T) \subseteq \Base^{\GG}(v)\sqcup \Base^{\GG}(v'),
    \]
    thus
    \[
        \base^{\GG}(v) + \base^{\GG}(v') - 2 \leq \base^{\GG/T}(v_T) \leq \base^{\GG}(v) + \base^{\GG}(v').
    \]

    By \Cref{prop:formula_h0_at_v}, the assumption can be rewritten as
    \begin{equation}\label{eq:aux_ineq_vprime}
        d_{v'} < \base^{\GG}(v')
    \end{equation}
    and 
    \begin{equation}\label{eq:aux_ineq_v}
        d_{v} < \base^{\GG}(v).
    \end{equation}
    Furthermore, by \Cref{lem:equivalence_being_base} we have that 
    \begin{equation}\label{eq:aux_stronger_ineq_vprime}
        v' \in \Base(v) \iff d_{v'} < \base^{T_{v'}}(v')
    \end{equation}
    with $T_{v'} = (\GG - v)_{v'}$.
    
    If $v' \notin \Base^{\GG}(v)$ then it is clear that $\base^{\GG}(v) = \base^{T_v}(v)$, so $d_v < \base^{T_v}(v)$ by Inequality \eqref{eq:aux_ineq_v}, and it follows that $v \in \Base^{\GG}(v')$ by \Cref{lem:equivalence_being_base}. Similarly, $v \notin \Base(v')$ implies that $v \in \Base^{\GG}(v')$.
    
    This means we are left with the following three cases:
    \begin{enumerate}[(i)]
        \item\label{item:vprime_notin_v_in} $v' \notin \Base(v)$ and $v \in \Base(v')$, 
        \item\label{item:vprime_in_v_notin} $v' \in \Base(v)$ and $v \notin \Base(v')$, or
        \item\label{item:vprime_and_v_in} $v' \in \Base(v)$ and $v \in \Base(v')$.
    \end{enumerate}
    
    In Case \ref{item:vprime_notin_v_in}, we have that $\base(v_T) = \base(v) + \base(v') - 1$ and $\base^{T_{v'}}(v') = \base^{\GG}(v')-1$. Combining these with Inequality \eqref{eq:aux_ineq_vprime} and \Cref{lem:equivalence_being_base}, we see that
    \[
        \base^{\GG}(v')  > d_{v'} \geq \base^{T_{v'}}(v') = \base^{\GG}(v')-1 \implies d_{v'} = \base^{\GG}(v')-1.
    \]
    Combining this equality with Inequality \eqref{eq:aux_ineq_v}, we conclude that 
    \[
        d_{v_T} = d_v + d_{v'} < \base(v) + \base(v') - 1 = \base(v_T).
    \]
    The same argument applies in Case \ref{item:vprime_in_v_notin}.
    
    In Case \ref{item:vprime_and_v_in}, we have that $\base(v_T) = \base(v) + \base(v') -2$, $\base^{T_v'}(v') = \base^{\GG}(v')-1$ and $\base^{T_v}(v) = \base^{\GG}(v)-1$. These equalities combined with Inequalities \eqref{eq:aux_ineq_vprime} and \eqref{eq:aux_ineq_v} imply that
    \[
        d_{v_T} = d_v + d_{v'} < \base(v) + \base(v') - 2 = \base(v_T).
    \]
    This concludes the proof that $v_T$ supports no sections in $\GG/T$.
    
    Finally, one can check that since $v$ and $v'$ support no global sections in $\GG$ and $v_T$ supports no global sections in $\GG/T$, then $H^0(\GG)$ and $H^0(\GG/T)$ can both be identified with tuples $(s_{v''})_{v'' \in V(\GG/T)_{v_T}}$ with $s_{v''} \in H^0(T_{v''})$ for $T_{v''} = (\GG/T - v_T)_{v''}$ such that $s_{v''}(p_{(v_T,v'')}) = 0$. Thus, $h^0(\GG) = h^0(\GG/T)$.
\end{proof}

\subsection{The shape of an irreducible component} \label{subsec:structural}

By \Cref{thm:components_maps_collection}, every irreducible component of $\overline{\ccM}_{0,n}(\Bl_{\pt}\bbP^2,\beta)$ is of the form $\overline{\ccM}_{\GG}$ for some $\GG \in \Gamma^{st}_{0,n}(\Bl_{\pt}\bbP^2,\beta)$. Moreover, we have seen that $\GG$ must be mic in \Cref{rmk:condition_MG_nonempty}, and there must be $v \in V(\GG)$ with $d_v \leq -2$ unless $\beta$ is negative or $\GG = \GG_0$. In this section, we collect several results that restrict the shape of $\GG$ if $\overline{\ccM}_{\GG}$ is an irreducible component.

We first prove that irreducible components cannot have any contracted component unless the total degree $\beta$ is 0.

\begin{proposition}\label{prop:no_zero_class}
    Let $\GG \in \Gamma^{mic}_{0,n}(\Bl_{\pt}\bbP^2,\beta)$ with $\beta\neq 0$. If there exists $v \in V(\GG)$ with $\beta_v = 0$ then $\overline{\ccM}_{\GG}$ is not an irreducible component.
\end{proposition}

\begin{proof}
    By \Cref{item:components_maps_negative_same_i} in \Cref{thm:components_maps_collection}, 
    it suffices to find a subtree $T$ of $\GG$ such that $h^0(\GG) = h^0(\GG/T)$. In the following, we show how to find such $T$ in different cases.
    
    We have $V(\GG)_v \neq \emptyset$ because $\beta \neq 0$. Let $v' \in V(\GG)_v$ and set $T = \langle v, v' \rangle$. If $d_{v'}\leq 0$ then every global section on $\GG$ is constant on $C_T$, and the same holds for $C_{v_T}$ in $\GG/T$.  Thus, $H^0(\GG)$ and $H^0(\GG/T)$ can both be identified with elements $(s_{v''}) \in \oplus_{v'' \in V(\GG/T)_{v_T}} H^0(T_{v''})$ with $s_{v''_1}(p_{v''_1}) = s_{v''_2}(p_{v''_2})$ for all $v''_1,v''_2 \in \Base(v_T)$, proving that $h^0(\GG) = h^0(\GG/T)$. 
    
    The same argument works if there exists $v'\in V(\GG)_v$ which  supports no global sections in $\GG$. Indeed, we have $d_{v'} < \base(v')$ by \Cref{prop:formula_h0_at_v}. Since $d_{v_T} = d_{v'}$ and $\base(v')\leq \base(v_T)$, we conclude that $v_T$ supports no sections in $\GG/T$ by \Cref{prop:formula_h0_at_v}. This means that every global section of $\GG$ is constant on $C_T$ and every section of $\GG/T$ is constant on  $C_{v_T}$, so the previous case applies.
    
    Therefore, we may assume without loss of generality that for every $v' \in V(\GG)_v$, it holds that $d_{v'}>0$ and $v'$ supports global sections. In particular, $\Base(v) = \emptyset$, because if $v'$ supports global sections in $\GG$, it also supports global sections in $T_{v'}$, and we can choose a section $s \in H^0(T_{v'})$ with non-zero value at $p_{(v,v')}$. Now, if we choose any $v'\in V(\GG)_v$ and let $T = \langle v, v' \rangle$, we deduce that $\Base(v') = \Base(v_T)$, thus $v_T$ supports global sections by \Cref{prop:base_larger_degree_implies_no_sections}. Applying Formula \eqref{eq:formula_h0_at_v} to $\GG$ at $v'$ and to $\GG/T$ at $v_T$ and using that $h^0(C_{v_T}) = h^0(C_{v'})$, we conclude that $h^0(\GG) = h^0(\GG/T)$. This finishes the proof.
\end{proof}

Next, we rule out decorated graphs with a non-negative vertex which supports no global sections.

\begin{proposition}\label{cor:base_larger_degree_implies_empty}
    Let $\GG \in \Gamma^{mic}_{0,n}(\Bl_{\pt}\bbP^2,\beta)$. If there exists $v \in V(\GG)$ with $\beta_v\neq 0$ and  $\base(v)>d_v\geq 0$, then $\ccM_{\GG} = \emptyset$.
\end{proposition}

\begin{proof}
    By contradiction, suppose there is a stable map $(C,f)$ in $\ccM_{\GG}$, with data $(C,L_s,L_e,s_0,s_1,s_2,s_3)$ (as in \Cref{subsec:geometry_Bl}). Recall that given $p\in C$, $s_3(p) = 0$ if and only if $f(p)\in E$. Therefore, $f(C_v)\subseteq E$ because the assumption $d_v < \base(v)$ implies that $s_3\mid_{C_v}\equiv 0$ by \Cref{prop:base_larger_degree_implies_no_sections}. Since $\beta_v\neq 0$, $f(C_v)$ is not a point, so $f\mid_{C_v}$ is a $k$-to-1 cover of the exceptional divisor $E$ for some $k\geq 1$. Then $\beta_v = k[f(C_v)] = ke$ and we get a contradiction
    \[
        0\leq d_v = \beta_v\cdot e =k e \cdot e = -k < 0, 
    \]
    proving no such stable map can exist.
\end{proof}

Having proven that certain types of vertices cannot appear in the decorated graph associated to a component, we now look at certain types of edges which cannot appear  either. 
The first type of edges that we rule out are those between negative vertices.

\begin{proposition}\label{prop:no_adjacent_negative_vertices}
    Let $\GG \in \Gamma^{mic}_{0,n}(\Bl_{\pt}\bbP^2,\beta)$. If there is $(v,v') \in E(\GG)$ with $d_v,d_{v'}< 0$ then $\overline{\ccM}_{\GG}$ is not a component.
\end{proposition}

\begin{proof}
    Let $T = \langle v, v'\rangle$. The result follows from \Cref{cor:no_sections_subtree_2_vertices} and \Cref{item:components_maps_negative_same_i} in \Cref{thm:components_maps_collection}.
\end{proof}

Next, we rule out edges between non-negative vertices whose curve classes are not 0.

\begin{proposition}\label{prop:no_adjacent_positive_nonconstant}
    Let $\GG \in \Gamma^{mic}_{0,n}(\Bl_{\pt}\bbP^2,\beta)$. If there is $(v,v') \in E(\GG)$ such that $d_v,d_{v'}\geq 0$ and $\beta_v,\beta_{v'}\neq 0$ then $\overline{\ccM}_{\GG}$ is not a component.
\end{proposition}

\begin{proof}
    Let $T = \langle v, v'\rangle$.
    Then $\beta_{v_T} = \beta_{v} + \beta_{v'}$ is mic (\Cref{rmk:sum_mic_not_always_mic}) and $d_{v_T} = d_v + d_{v'} \geq 0$. 
    By \Cref{cor:base_larger_degree_implies_empty}, we may assume that $d_v\geq \base(v)$ and $d_{v'}\geq \base(v')$. By \Cref{lem:equivalence_being_base}, this implies that $v \notin \Base(v')$ and $v' \notin \Base(v)$, thus $\Base(v_T) = \Base(v) \sqcup \Base(v')$. Therefore $d_{v_T} \geq \base(v_T)$, so we can compute $h^0(\GG/T)$ using Formula \eqref{eq:formula_h0_at_v} at the vertex $v_T$. 
    
    On the other hand, the argument used in the proof of \Cref{prop:formula_h0_at_v} shows that in this case
    \begin{equation}\label{eq:aux_no_adjacent_positive_G}
        h^0(\GG) = h^0(T,L_e\mid_{T}(-\sum_{v'' \in \Base(v_T)} p_{v''})) + \sum_{v'' \in \Base(v_T)} h^0(T_{v''})+ \sum_{v'' \notin \Base(v_T)}  (h^0(T_{v''}) - 1),
    \end{equation}
    where $p_{v''}$ denotes $p_{(v,v'')}$ if $v''\in V(\GG)_v$ or $p_{(v',v'')}$ if $v'' \in V(\GG)_{v'}$.
    
   We also have that
    \begin{align}
        h^0(T,L_e\mid_{T}(-\sum_{v'' \in \Base(v_T)} p_{v''})) &= h^0(C_v,L_e(-\sum_{v'' \in \Base(v)} p_{(v'',v)})) + \nonumber\\ & + h^0(C_{v'},L_e(-\sum_{v'' \in \Base(v')} p_{(v'',v')})) - 1 = \nonumber\\ & = h^0(C_v) + h^0(C_{v'}) -1 - \base(v) - \base(v') = h^0(C_{v_T}) - \base(v_T).\label{eq:aux_no_adjacent_positive_T_vT}
    \end{align}

     Comparing the expression for $h^0(\GG/T)$ obtained from Formula \eqref{eq:formula_h0_at_v} with Formula \eqref{eq:aux_no_adjacent_positive_G} and using \Cref{eq:aux_no_adjacent_positive_T_vT}, we deduce that $h^0(\GG) = h^0(\GG/T)$. Therefore, $\ccM_{\GG} \subseteq \overline{\ccM}_{\GG/T}$ by \Cref{item:components_maps_negative_same_i} in \Cref{thm:components_maps_collection},
    so the conclusion follows.
\end{proof}

To conclude this section, we prove that if $\beta$ is non-negative, then irreducible components cannot have any component of degree $-1$.

\begin{proposition}\label{prop:no_minus_one}
    Let $\GG \in \Gamma^{mic}_{0,n}(\Bl_{\pt}\bbP^2,\beta)$ with $\beta \cdot e \geq 0$. 
    If there is $v \in V(\GG)$ with $d_v = -1$ then $\overline{\ccM}_{\GG}$ is not an irreducible component.
\end{proposition}

\begin{proof}
    If $\beta = 0$, the result is immediate (see \Cref{thm:irreducibility_BlP2}). Therefore, we assume that $\beta \neq 0$. Fix $v \in V(\GG)$ as in the statement. By \Cref{item:components_maps_negative_same_i} in \Cref{thm:components_maps_collection},
    it suffices to show that 
    we can find a subtree $T$ of $\GG$ such that $h^0(\GG) = h^0(\GG/T)$. 
    
    Since $\beta \cdot e \geq 0$ and $d_v < 0$, we must have that $\lvert V(\GG) \rvert \geq 2$.
    By \Cref{prop:no_zero_class,prop:no_adjacent_negative_vertices,prop:no_adjacent_positive_nonconstant}, we may assume there are no adjacent vertices in $\GG$ of the same sign. This implies that $d_{v''}\geq 0$ for all $v'' \in V(\GG)_{v}$. In fact, we may further assume that $d_{v''}>0$ and $v''$ supports global sections in $\GG$ by \Cref{cor:base_larger_degree_implies_empty} (see also \Cref{ex:negative_degree_supports_no_sections} and \Cref{prop:base_larger_degree_implies_no_sections}).

    Thus, we may choose $v' \in V(\GG)_{v}$ with $d_{v'}>0$ which supports global sections and $d_{v''}<0$ for all $v'' \in V(\GG)_{v'}$. Set $T = \langle v, v'\rangle$.
    Then $\Base(v') = \Base(v_T) \sqcup \{v\}$, therefore 
    \[
        \base(v')\leq d_{v'} \implies \base(v_T) = \base(v')-1 \leq d_{v'}-1 = d_{v_T}.
    \]
    This implies that \Cref{prop:formula_h0_at_v} applies to both $v'$ in $\GG$ and $v_T$ in $\GG/T$. We conclude that $h^0(\GG) = h^0(\GG/T)$ by comparing Formula \eqref{eq:formula_h0_at_v} for each of them, and using that
    \[
        h^0(C_{v'}) = d_{v'} + 1 = d_{v_T} + 2 = h^0(C_{v_T}) + 1
    \]
    and
    \[
        h^0(T_{v}) = \sum_{v'' \notin \Base(v_T)} (h^0(T_{v''}) - 1),
    \]
    where $T_v \coloneqq (\GG - v')_v$ and $T_{v''}\coloneqq (G/T - v_T)_{v''}$. This finishes the proof.
\end{proof}

\subsection{Tentative components}\label{subsec:tentative}

We introduce the notion of tentative component to collect the restrictions from \Cref{subsec:structural} that $\GG \in \Gamma^{mic}_{0,n}(\Bl_{\pt}\bbP^2,\beta)$ must satisfy if $\overline{\ccM}_{\GG}$ is an irreducible component. For those, we give a simple formula for $h^0(\GG)$ and $d_{\GG}$. We also upgrade \Cref{thm:components_maps_collection} for tentative components.

\begin{notation}
    For $\GG \in \Gamma_{0,n}(X,\beta)$, we write $V(\GG)_{\geq 0}$ (resp. $V(\GG)_{<0}$) for the collection of non-negative (resp. negative) vertices in $V(\GG)$.
\end{notation}

\begin{definition}\label{def:tentative_component}
    Let $\GG \in \Gamma^{mic}_{0,n}(\Bl_{\pt}\bbP^2,\beta)$. We say that $\GG$ is
    \begin{itemize}
        \item \textit{sign-alternating} if for every edge in $(v,v') \in E(\GG)$ it holds that $d_v \geq 0$ and $d_{v'}\leq -2$ up to order, 
        \item \textit{maximally supporting} if $d_v\geq 0$ implies that $\valency(v)\leq d_v$ for all $v \in V(\GG)$,
        \item \textit{a tentative component} if it is sign-alternating,  
        maximally supporting and $d_{\GG}\geq 0$.
    \end{itemize}
\end{definition}

\begin{example}
    For any $n\in\bbZ_{\geq 0}$ and $\beta \in A_1^{mic}(\Bl_{\pt}\bbP^2)$ satisfying \Cref{assumption:stable_maps_non_empty}, $\GG_0 \in \Gamma_{0,n}^{mic}(\Bl_{\pt}\bbP^2,\beta)$ is a tentative component.
\end{example}

\begin{remark}
    If $\beta \neq 0$ and $\GG \in \Gamma_{0,n}^{mic}(\Bl_{\pt}\bbP^2,\beta)$ is maximally supporting then $\beta_v \neq 0$ for all $v \in V(\GG)$. Indeed, if $\valency(v) > 0$ then maximally supporting implies that either $d_v<0$ or $d_v \geq \valency(v)>0$, thus $\beta_v\neq 0$. If $\valency(v) = 0$ then $\GG=\GG_0$ and $\beta_v = \beta \neq 0$ by assumption.
\end{remark}

\begin{remark}\label{rmk:reformulation_supporting_if_sign_alternating}
    If $\GG \in \Gamma^{mic}_{0,n}(\Bl_{\pt}\bbP^2,\beta)$ is sign-alternating, we have that $\base(v) = \valency(v)$ for all $v\in V(\GG)_{\geq 0}$. Thus, a sign-alternating graph $\GG$ is maximally supporting if and only if every non-negative vertex $v \in V(\GG)$ supports global sections, see \Cref{prop:formula_h0_at_v}.
\end{remark}

For a sign-alternating, maximally supporting mic tree $\GG$ whose total class is non-negative, the following result gives a simple formula for $h^0(\GG)$, correcting the Naive Formula \eqref{eq:naive_formula_h0}, as well as a simple formula for $d_{\GG}$.

\begin{proposition}\label{prop:simplification_dG_tentative_components}
    Let $\beta \in A_1^{mic}(\Bl_{\pt}\bbP^2)$ with $\beta \cdot e \geq 0$. 
    Let $\GG \in \Gamma^{mic}_{0,n}(\Bl_{\pt}\bbP^2,\beta)$ be sign-alternating and maximally supporting. Then
    \[
        h^0(\GG) = \sum_{v\in V(\GG)_{\geq 0}} d_v + \lvert V(\GG)_{\geq 0}\rvert - \lvert E(\GG) \rvert
    \]
    and
    \[
        d_{\GG} = -\sum_{v \in V(\GG)_{<0}} d_v + \lvert V(\GG)_{\geq 0} \rvert - 2\lvert E(\GG) \rvert - 1
    \]
\end{proposition}

\begin{proof}
    Since $\GG$ is sign-alternating, we have an isomorphism
    \[
        H^0(\GG) \simeq \bigoplus_{v \in V(\GG)_{\geq 0}} H^0(C_v,L_e\mid_{C_v}(-\sum_{v'\in V(\GG)_{v}} p_{(v,v')})).
    \]
    Since $\GG$ is maximally supporting, for any $v \in V(\GG)_{\geq 0}$ we have that
    \[
        h^0(C_v,L_e\mid_{C_v}(-\sum_{v'\in V(\GG)_{v}} p_{(v,v')})) = d_v - \lvert V(\GG)_v \rvert + 1.
    \]
    The formula for $h^0(\GG)$ in the statement is obtained by summing over all $v \in V(\GG)_{\geq 0}$. 
    
    On the other hand, since $\beta \cdot e \geq 0$, then $h^0(\bbP^1,\ccO_{\bbP^1}(\beta \cdot e)) = 1 + \beta \cdot e = 1 + \sum_{v \in V(\GG)} d_v$. The formula for $d_{\GG}$ in the statement is then an immediate consequence of the definitions of $i_{\GG}$ and $d_{\GG}$ in \eqref{eq:definition_iG_prod} and \eqref{eq:definition_dG_prod}.
\end{proof}

The next result upgrades \Cref{thm:components_maps_collection} for tentative components.

\begin{theorem}\label{thm:component_implies_tentative_component}\label{thm:criteria_tentative_component}
    Let $n \in \bbZ_{\geq 0}$ and $\beta \in A_1^{mic}(\Bl_{\pt}\bbP^2)$ with $\beta \cdot e \geq 0$ 
    and $\beta \neq 0$. Then
    \begin{enumerate}[(i)]
        \item\label{item:component_implies_tentative_component} Every irreducible component of $\overline{\ccM}_{0,n}(\Bl_{\pt}\bbP^2,\beta)$ is of the form $\overline{\ccM}_{\GG}$ for some tentative component $\GG \in \Gamma^{mic}_{0,n}(\Bl_{\pt}\bbP^2,\beta)$.
        \item\label{item:main_Bl}   $\overline{\ccM}_{\GG_0}$ is an irreducible component of $\overline{\ccM}_{0,n}(X,\beta)$.
    \end{enumerate}
    Furthermore, given a tentative component $\GG \in \Gamma_{0,n}^{mic}(\Bl_{\pt}\bbP^2,\beta)$ with $\ccM_{\GG}\neq \emptyset$,
    \begin{enumerate}[(i)]
    \setcounter{enumi}{2}
        \item\label{item:sufficient_criterion} If $d_{\GG}\geq d_{\GG'}$ for every tentative component $\GG' \in \Gamma_{0,n}^{mic}(\Bl_{\pt}\bbP^2,\beta)$ which is a contraction of $\GG$ (\Cref{def:contraction_decorated_trees}), then $\overline{\ccM}_{\GG}$ is an irreducible component of $\overline{\ccM}_{0,n}(\Bl_{\pt}\bbP^2,\beta)$.
        \item\label{item:necessary_criterion} If there exists a tentative component  $\GG' \in \Gamma_{0,n}^{mic}(\Bl_{\pt}\bbP^2,\beta)$ which is a contraction of $\GG$ with $\GG'\neq \GG$ and $h^0(\GG) = h^0(\GG')$ then $\overline{\ccM}_{\GG}$ is not an irreducible component of $\overline{\ccM}_{0,n}(\Bl_{\pt}\bbP^2,\beta)$.
    \end{enumerate}
\end{theorem}

\begin{proof}
    Every irreducible component of $\overline{\ccM}_{0,n}(\Bl_{\pt}\bbP^2,\beta)$ is of the form $\overline{\ccM}_{\GG}$ for some decorated graph $\GG \in \Gamma^{mic}_{0,n}(\Bl_{\pt}\bbP^2,\beta)$ by \Cref{item:components_maps_description} in \Cref{thm:components_maps_collection} and \Cref{item:MG_empty_not_mic} in \Cref{rmk:condition_MG_nonempty}. Furthermore, $\GG$ must be 
    \begin{enumerate}[(i)]
        \item sign-alternating by \Cref{prop:no_zero_class,prop:no_adjacent_negative_vertices,prop:no_adjacent_positive_nonconstant,prop:no_minus_one},
        \item maximally supporting by \Cref{cor:base_larger_degree_implies_empty} and \Cref{rmk:reformulation_supporting_if_sign_alternating} and
        \item have $d_{\GG} \geq 0$ by \Cref{item:components_maps_smaller_0} in \Cref{thm:components_maps_collection}.
    \end{enumerate}
    This means $\GG$ must be a tentative component, proving \Cref{item:component_implies_tentative_component}. \Cref{item:main_Bl} is an immediate consequence of \Cref{item:components_maps_main} in \Cref{thm:components_maps_collection} and \Cref{rmk:main_M_0n_nonempty_iff_mic}.
    
    On the other hand, the proof of \Cref{item:sufficient_criterion} is analogous to that of \Cref{item:components_maps_sufficient_numerical_condition} in \Cref{thm:components_maps_collection}, using the idea in \Cref{rmk:enough_smaller_subset_containing_components}. Finally, \Cref{item:necessary_criterion} is a particular case of \Cref{item:components_maps_negative_same_i} in \Cref{thm:components_maps_collection}. 
\end{proof}

\begin{remark}
    There exist tentative components which do not satisfy \Cref{item:sufficient_criterion} nor \Cref{item:necessary_criterion} in \Cref{thm:criteria_tentative_component} and for which we cannot determine whether they give rise to irreducible components or not. See \Cref{ex:component_or_not}.
\end{remark}

\subsection{Irreducible components of $\overline{\ccM}_{0,n}(\Bl_{\pt}\bbP^2,\beta)$ in low degree}\label{subsec:general_case}

Having determined the curve classes $\beta\in A_1^{mic}(\Bl_{\pt}\bbP^2)$ for which $\overline{\ccM}_{0,n}(\Bl_{\pt}\bbP^2,\beta)$ is irreducible in \Cref{thm:irreducibility_BlP2}, we focus now on the reducible case. We consider $\beta = as+be \in A_1(\Bl_{\pt}\bbP^2)$ non-zero non-negative and mic with $b\geq 2$. These are represented by $\times$ in \Cref{fig:Mori_cone}. In the following, we describe all the irreducible components of $\overline{\ccM}_{0,n}(\Bl_{\pt}\bbP^2,as+be)$ for $b= 2,3,4$. We also find an obstacle for $b\geq 5$ in \Cref{ex:component_or_not}.

The following remark reduces the general case to the case $n=0$. In it, we consider the forgetful morphism $F\colon \Gamma_{0,n}(X,\beta) \to \Gamma_{0,0}(X,\beta)$. Moreover, given $\HH \in \Gamma_{0,0}(X,\beta)$ and a map $m\colon V(\HH)\to \{1,\ldots, n\}$, we denote by $(\HH,m) \in \Gamma_{0,n}(X,\beta)$ the decorated tree obtained by endowing $\HH$ with the marking $m$.

\begin{remark}\label{rmk:reduction_n} 
    Let $n\in \bbZ_{\geq 0}$ and $\beta \in A_1^{mic}(\Bl_{\pt}\bbP^2)$ with $\beta \neq 0$ and $\beta \cdot e \geq 0$. Due to \Cref{item:components_maps_description} in \Cref{thm:components_maps_collection} and \Cref{prop:no_zero_class}, we expect the following properties to hold:
    \begin{enumerate}[(i)]
        \item\label{item:reduction_n_forget} If $\overline{\ccM}_{\GG}$ is an irreducible component of $\overline{\ccM}_{0,n}(\Bl_{\pt}\bbP^2,\beta)$ for some $\GG \in \Gamma_{0,n}^{mic}(X,\beta)$, then $\overline{\ccM}_{F(\GG)}$ is an irreducible component of $\overline{\ccM}_{0,0}(\Bl_{\pt}\bbP^2,\beta)$.
        \item\label{item:reduction_n_remember} Conversely, if $\overline{\ccM}_{\HH}$ is an irreducible component of $\overline{\ccM}_{0,0}(\Bl_{\pt}\bbP^2,\beta)$ for some $\HH \in \Gamma_{0,0}^{mic}(X,\beta)$, then $\overline{\ccM}_{(\HH,m)}$ is an irreducible component of $\overline{\ccM}_{0,n}(\Bl_{\pt}\bbP^2,\beta)$ for any map $m\colon V(\HH) \to \{1,\ldots, n\}$.
    \end{enumerate}
    This expectation is certainly true for those $\beta$ for which all the irreducible components of $\overline{\ccM}_{0,0}(X,\beta)$ can be determined using the numerical criteria in
    \Cref{item:necessary_criterion,item:sufficient_criterion} of \Cref{thm:criteria_tentative_component},
    as they are independent of $n$. This condition is satisfied for $\beta = as+be$ non-negative with $b=2,3,4$, as we shall see through this section.
\end{remark}

Before looking at concrete examples in low degree, we prove two useful results. The first one concerns tentative components with a single negative vertex. The second one characterizes the irreducible components which have a single negative vertex of class $be$ (which is as negative as allowed by $\beta$). These are characterized by the valency of the negative vertex.

We recall that a graph $G$ is called a star if there exists $v \in V(\GG)$ which is adjacent to any other vertex. In that case, we say $v$ is an internal vertex of $\GG$.

\begin{proposition}\label{prop:unique_negative_vertex_condition}
    Let $\beta = as+be\in A_1^{mic}(\Bl_{\pt}\bbP^2)$ be non-negative with $b\geq 2$  and
    $\GG \in \Gamma^{mic}_{0,n}(\Bl_{\pt}\bbP^2,\beta)$ with a unique negative vertex $v\in V(\GG)$.
    If $\GG$ is a tentative component,
    then $\GG$ is a star with internal vertex $v$ and $\beta_v = b'e$ with $2\leq b'\leq b$ and $\valency(v)< b'$.
\end{proposition}

\begin{proof}
    Since $\GG$ is sign-alternating, $\GG$ must be a star graph with internal vertex $v$ and $d_v = b'e$ for some $2\leq b'\leq b$. Since $\GG$ is a tentative component, it satisfies $d_{\GG}\geq 0$. By \Cref{prop:simplification_dG_tentative_components}, we conclude that
    \[
        d_{\GG} = b' - \valency(v) - 1 \geq 0 \iff b' > \valency(v). \qedhere
    \]
\end{proof}

\begin{proposition}\label{prop:maximal_negative_vertex}
    Let $\beta = as+be\in A_1^{mic}(\Bl_{\pt}\bbP^2)$ be non-negative with $b\geq 2$ and
    $\GG \in \Gamma^{mic}_{0,n}(\Bl_{\pt}\bbP^2,\beta)$. If there exists $v\in V(\GG)$ with $\beta_v = be$, the following are equivalent:
    \begin{enumerate}[(i)]
        \item $\overline{\ccM}_{\GG}$ is an irreducible component of $\overline{\ccM}_{0,n}(\Bl_{\pt}\bbP^2,\beta)$.
        \item\label{item:G_star_plus_conditions} $\GG$ is a star with internal vertex $v$, $\valency(v) < b$, the leaves $v_1,\ldots, v_{\valency(v)}$ have degree $a_is$ with $a_i>0$ and $\sum_{i=1}^{\valency(v)} a_i = a$.
    \end{enumerate}
\end{proposition}

\begin{figure}
    \centering
    \begin{tikzpicture}
                \draw (0,0) node[] {\textbullet};
                \draw (0,0) node[anchor= south] {$be$};
                \draw (-2,-1) node[] {\textbullet};
                \draw (-2,-1) node[anchor= north] {$a_1 s$};
                \draw (2,-1) node[] {\textbullet};
                \draw (2,-1) node[anchor= north] {$a_{k} s$};
                \draw (0,-1) node {$\ldots$};
                
                \draw (0,0) -- (-2,-1);
                \draw (0,0) -- (2,-1);
            \end{tikzpicture}
    \caption{For $a\geq b \geq 2$, the irreducible components of  $\overline{\ccM}_{0,0}(\Bl_0\bbP^2,as+be)$ with a vertex of class $be$ are exactly those of this form, satisfying $a_1,\ldots, a_k > 0$, $\sum_{i=1}^{k} a_i = a$ and $k < b$.
    }
    \label{fig:components_maximal_negative_vertex}
\end{figure}

\begin{proof}
    Consider $\GG$ as in the statement. 
    If $\overline{\ccM}_{\GG}$ is an irreducible component, then $\GG$ is a tentative component by \Cref{thm:component_implies_tentative_component}. Then, the assumption $\beta_v=be$ implies that $v$ is the unique negative vertex in $\GG$, so \Cref{item:G_star_plus_conditions} follows from \Cref{prop:unique_negative_vertex_condition}.

    Conversely, suppose that $\GG$ satisfies the conditions in \Cref{item:G_star_plus_conditions}.
    By \Cref{item:sufficient_criterion} in \Cref{thm:criteria_tentative_component}, it suffices to prove that $d_{\GG} \geq d_{\GG'}$ for every contraction $\GG' \in \Gamma^{mic}_{0,0}(\Bl_{\pt}\bbP^2,\beta)$ of $\GG$ which is a tentative component.
    If we contract the edges corresponding to $a_1 s,\ldots, a_\ell s$ for some $1\leq \ell <k$ into $v$, then the new class $\beta'_v = \sum_{i=1}^k a_i s + be$ of $v$ is mic if and only if $\sum_{i=1}^k a_i \geq b$, which is equivalent to $\GG'$ not being sign-alternating. This implies that the only tentative component $\GG'$ obtained by $\GG$ from a sequence of edge contractions is $\GG_0$. Finally, by the proof of \Cref{prop:unique_negative_vertex_condition}, the inequality $b>\valency(v)$ ensures that $d_{\GG}\geq d_{\GG_0} = 0$. Thus,
    all the tentative components in \Cref{item:G_star_plus_conditions} give rise to irreducible components.
\end{proof}

\begin{theorem}\label{thm:classification_2l}
    Let $\alpha \in \bbZ_{\geq 0}$. The irreducible components of $\overline{\ccM}_{0,n}(\Bl_{\pt}\bbP^2,2\ell+\alpha s)$ are of the form $\overline{\ccM}_{\GG}$ for a decorated graph $\GG$ in the following list
    \begin{itemize}
        \item $\GG_0$,  
        \item $\GG_{2,I}$ with $I\subseteq \{1,\ldots, n\}$, pictured in \Cref{fig:GG_c_I}.
    \end{itemize}
    In particular, there are $2^n + 1$ irreducible components.
\end{theorem}

\begin{proof}
    By \Cref{thm:criteria_tentative_component}, $\overline{\ccM}_{\GG_0}$ is an irreducible component and every other component is of the form $\overline{\ccM}_{\GG}$ for some tentative component $\GG \neq \GG_0$. If $\GG\neq \GG_0$ is a tentative component, there exists $v \in V(\GG)$ with $d_v\leq -2$ by definition. Our choice of total class $\beta$ implies $\GG$ has a single negative vertex $v$, of class $\beta_v = 2e$. Then \Cref{prop:maximal_negative_vertex} implies that $\overline{\ccM}_{\GG}$ is an irreducible component if and only if $\GG = \GG_{2,I}$ for some $I\subseteq \{1,\ldots, n\}$. The enumeration is immediate.
\end{proof}

In the following results, we restrict to the case $n=0$ for simplicity. For $n>0$, see \Cref{rmk:reduction_n}.

\begin{theorem}\label{thm:classification_3l}
    Let $\alpha \in \bbZ_{\geq 0}$. The irreducible components of $\overline{\ccM}_{0,0}(\Bl_{\pt}\bbP^2,3\ell+\alpha s)$ are of the form $\overline{\ccM}_{\GG}$ for a decorated graph $\GG$ in the following list, pictured in \Cref{fig:classification_3l}:
    \begin{itemize}
        \item $\GG_0$,  
        \item $\GG_{3,\alpha+3}$,  
        \item $\GG_{3,\alpha_1,\alpha_2}$ for some $1\leq \alpha_1,\alpha_2$ with $\alpha_1+\alpha_2 = \alpha + 3$,  
        \item $\GG_{2,\alpha+3}$.
    \end{itemize}
\end{theorem}

\newcommand{\scalefactorthree}{.75}
\newcommand{\Tthreezero}{
    \begin{tikzpicture}[scale=\scalefactorthree]
                \draw (0,0) node[] {\textbullet};
                \draw (0,0) node[anchor= south] {$3\ell + \alpha s$};
    \end{tikzpicture}
}

\newcommand{\Tthreethreealpha}{
    \begin{tikzpicture}[scale=\scalefactorthree]
                \draw (0,0) node[] {\textbullet};
                \draw (0,0) node[anchor= south] {$3e$};
                \draw (0,-1) node[] {\textbullet};
                \draw (0,-1) node[anchor= north] {$(\alpha + 3)s$};
                
                \draw (0,0) -- (0,-1);
    \end{tikzpicture}
}

\newcommand{\Tthreethreealphas}{
    \begin{tikzpicture}[scale=\scalefactorthree]
                \draw (0,0) node[] {\textbullet};
                \draw (0,0) node[anchor= south] {$3e$};
                \draw (-1,-1) node[] {\textbullet};
                \draw (-1,-1) node[anchor= north] {$\alpha_1 s$};
                \draw (1,-1) node[] {\textbullet};
                \draw (1,-1) node[anchor= north] {$\alpha_2s$};
                
                \draw (0,0) -- (-1,-1);
                \draw (0,0) -- (1,-1);
    \end{tikzpicture}
}
\newcommand{\Tthreetwo}{
    \begin{tikzpicture}[scale=\scalefactorthree]
                \draw (0,0) node[] {\textbullet};
                \draw (0,0) node[anchor= south] {$2e$};
                \draw (0,-1) node[] {\textbullet};
                \draw (0,-1) node[anchor= north] {$(\alpha + 3)s + e$};
                
                \draw (0,0) -- (0,-1);
    \end{tikzpicture}
}

\begin{figure}
    \centering
    \begin{tabular}{cccc}
    \Tthreezero & \Tthreethreealpha & \Tthreethreealphas & \Tthreetwo \\
    $\GG_0$  & $\GG_{3,\alpha+3}$  & $\GG_{3,\alpha_1,\alpha_2}$  & $\GG_{2,\alpha+3}$
    \end{tabular}
    \caption{All tentative components in $\Gamma^{mic}_{0,0}(\Bl_{\pt}\bbP^2, 3\ell+\alpha s)$}
    \label{fig:classification_3l}
\end{figure}

\begin{proof}
    First, we show that the decorated graphs listed in the statement are precisely the tentative components $\GG \in \Gamma_{0,0}^{mic}(\Bl_{\pt}\bbP^2,3\ell+\alpha s)$. If $\GG \neq \GG_0$, then $\GG$ must have at least one negative vertex $v$, with $d_v\leq -2$ because $\GG$ is sign-alternating. This implies that $\beta_v\in \{2e, 3e\}$. If $\beta_v = 3e$, the only possibilities are the graph $\GG_{3,\alpha+3}$ or the graphs $\{\GG_{3,\alpha_1,\alpha_2}\}$ with $1\leq \alpha_1,\alpha_2$ and $\alpha_1+\alpha_2 = \alpha +3$ by \Cref{prop:maximal_negative_vertex}. If $\beta_v = 2e$, then $\valency(v) = 1$ by \Cref{prop:unique_negative_vertex_condition}, so the only possibility is $\GG_{2,\alpha+3}$. Using \Cref{prop:simplification_dG_tentative_components}, it is easy to check that $d_{\GG_{3,\alpha+3}} = 1$ and all the other graphs have $d_{\GG} = 0$.
    
    To conclude, we need to prove that all the tentative components previously listed actually give rise to components. We already know this is the case for $\GG_0$ by \Cref{item:main_Bl} in \Cref{thm:criteria_tentative_component}, and for the graphs with $\beta_v=3e$ by \Cref{prop:maximal_negative_vertex}. For $\GG_{2,\alpha+3}$, it follows from
    \Cref{item:sufficient_criterion} in \Cref{thm:criteria_tentative_component},
    because $\GG_{2,\alpha+3}$ only contracts into $\GG_0$.
\end{proof}

\begin{corollary}\label{cor:classification_3l_count}
    Let $\alpha \in \bbZ_{\geq 0}$. The number of irreducible components of $\overline{\ccM}_{0,n}(\Bl_{\pt}\bbP^2,3\ell+\alpha s)$ is
    \[
        1 + 2^{n+1} + \left\lceil \frac{3^n(\alpha + 2)}{2} \right\rceil 
    \]
\end{corollary}

\begin{proof}
    Using \Cref{rmk:reduction_n}, the irreducible components of $\overline{\ccM}_{0,n}(\Bl_{\pt}\bbP^2,3\ell+\alpha s)$ are labelled by pairs $(\GG,m)$ where $\GG \in \Gamma_{0,0}^{mic}(\Bl_{\pt}\bbP^2,\beta)$ is one of the graphs in \Cref{thm:classification_3l} and $m\colon \{1,\ldots, n\} \to V(\GG)$ is any marking. We want to enumerate those up to isomorphism (of the underlying $\GG$ which commutes with the map $m$). 
    
    It is clear that $\GG_0$ gives rise to a single pair and that $\GG_{3,\alpha+3}$ and $\GG_{2,\alpha+3}$ each gives rise to $2^n$ of them due to markings. Besides, none of these are equivalent because the underlying graphs have no automorphisms. 
    
    On the other hand, each of the $\alpha + 2$ different graphs of type $\GG_{3,\alpha_1,\alpha_2}$ admits $3^n$ different markings. The only pair $(\GG_{3,\alpha_1,\alpha_2},m)$ admitting a non-trivial automorphism has $\alpha_1 = \alpha_2$ (which implies $\alpha$ being even) and has all the marks in the vertex of class $3e$. Using Burnside's Lemma, one can check there are $\lceil \frac{3^n(\alpha+2)}{2} \rceil$ such isomorphism classes. Adding this with the previous counts, we obtain the result.
\end{proof}

\begin{theorem}\label{thm:classification_4l}
    Let $\alpha \in \bbZ_{\geq 0}$. The irreducible components of $\overline{\ccM}_{0,0}(\Bl_{\pt}\bbP^2,4\ell+\alpha s)$ are of the form $\overline{\ccM}_{\GG}$ for a decorated graph $\GG$ in the following list, pictured in \Cref{fig:classification_4l_1,fig:classification_4l_2}:
    \begin{itemize}
        \item $\GG_0$
        \item $\GG_{4,\alpha+4}$
        \item $\GG_{4,\alpha_1,\alpha_2}$ for some $1\leq \alpha_1,\alpha_2$ with $\alpha_1+\alpha_2 = \alpha + 4$
        \item $\GG_{4,\alpha_1,\alpha_2,\alpha_3}$ for some $1\leq \alpha_i$ with $\sum_{i=1}^3 \alpha_i = \alpha + 4$
        \item $\GG_{3,\alpha+4}$  
        \item $\GG_{3,\alpha_1,\alpha_2}$ for some $1\leq \alpha_1$, $2\leq \alpha_2$ with $\alpha_1+\alpha_2 = \alpha + 4$
        \item $\GG_{2,\alpha+4}$  
        \item $\GG_{2,\alpha+4,2}$.
    \end{itemize}
\end{theorem}

\newcommand{\scalefactorfour}{.75}
\newcommand{\Tfourzero}{
    \begin{tikzpicture}[scale=\scalefactorfour]
                \draw (0,0) node[] {\textbullet};
                \draw (0,0) node[anchor= south] {$4\ell + \alpha s$};
    \end{tikzpicture}
}

\newcommand{\Tfouralpha}{
    \begin{tikzpicture}[scale=\scalefactorfour]
                \draw (0,0) node[] {\textbullet};
                \draw (0,0) node[anchor= south] {$4e$};
                \draw (0,-1) node[] {\textbullet};
                \draw (0,-1) node[anchor= north] {$(\alpha + 4)s$};
                
                \draw (0,0) -- (0,-1);
    \end{tikzpicture}
}

\newcommand{\Tfouralphaonetwo}{
    \begin{tikzpicture}[scale=\scalefactorfour]
                \draw (0,0) node[] {\textbullet};
                \draw (0,0) node[anchor= south] {$4e$};
                \draw (-1,-1) node[] {\textbullet};
                \draw (-1,-1) node[anchor= north] {$\alpha_1s$};
                \draw (1,-1) node[] {\textbullet};
                \draw (1,-1) node[anchor= north] {$\alpha_2s$};
                
                \draw (0,0) -- (-1,-1);
                \draw (0,0) -- (1,-1);
    \end{tikzpicture}
}
\newcommand{\Tfouralphaonetwothree}{
    \begin{tikzpicture}[scale=\scalefactorfour]
                \draw (0,0) node[] {\textbullet};
                \draw (0,0) node[anchor= south] {$4e$};
                \draw (-1,-1) node[] {\textbullet};
                \draw (-1,-1) node[anchor= north] {$\alpha_1s$};
                \draw (0,-1) node[] {\textbullet};
                \draw (0,-1) node[anchor= north] {$\alpha_2s$};
                \draw (1,-1) node[] {\textbullet};
                \draw (1,-1) node[anchor= north] {$\alpha_3s$};
                
                \draw (0,0) -- (-1,-1);
                \draw (0,0) -- (0,-1);
                \draw (0,0) -- (1,-1);
    \end{tikzpicture}
}

\newcommand{\Tfourthreealpha}{
    \begin{tikzpicture}[scale=\scalefactorfour]
                \draw (0,0) node[] {\textbullet};
                \draw (0,0) node[anchor= south] {$3e$};
                \draw (0,-1) node[] {\textbullet};
                \draw (0,-1) node[anchor= north] {$(\alpha + 4)s+e$};
                
                \draw (0,0) -- (0,-1);
    \end{tikzpicture}
}

\newcommand{\Tfourthreealphaonetwo}{
    \begin{tikzpicture}[scale=\scalefactorfour]
                \draw (0,0) node[] {\textbullet};
                \draw (0,0) node[anchor= south] {$3e$};
                \draw (-1,-1) node[] {\textbullet};
                \draw (-1,-1) node[anchor= north] {$\alpha_1 s$};
                \draw (1,-1) node[] {\textbullet};
                \draw (1,-1) node[anchor= north] {$\alpha_2s+e$};
                
                \draw (0,0) -- (-1,-1);
                \draw (0,0) -- (1,-1);
    \end{tikzpicture}
}

\newcommand{\Tfourtwoalpha}{
    \begin{tikzpicture}[scale=\scalefactorfour]
                \draw (0,0) node[] {\textbullet};
                \draw (0,0) node[anchor= south] {$2e$};
                \draw (0,-1) node[] {\textbullet};
                \draw (0,-1) node[anchor= north] {$(\alpha + 4)s+2e$};
                
                \draw (0,0) -- (0,-1);
    \end{tikzpicture}
}
\newcommand{\Tfourtwoalphatwo}{
    \begin{tikzpicture}[scale=\scalefactorfour]
                \draw (0,0) node[] {\textbullet};
                \draw (0,0) node[anchor= west] {$2e$};
                \draw (0,-1) node[] {\textbullet};
                \draw (0,-1) node[anchor= west] {$(\alpha + 4)s$};
                \draw (0,-2) node[] {\textbullet};
                \draw (0,-2) node[anchor= west] {$2e$};
                
                \draw (0,0) -- (0,-2);
    \end{tikzpicture}
}

\begin{figure}
    \centering
    \begin{tabular}{cccc}
    \Tfourzero & \Tfouralpha & \Tfouralphaonetwo & \Tfouralphaonetwothree \\
    $\GG_0$  & $\GG_{4,\alpha+4}$  & $\GG_{4,\alpha_1,\alpha_2}$  & $\GG_{4,\alpha_1,\alpha_2,\alpha_3}$
    \end{tabular}
    \caption{A collection of tentative components in $\Gamma^{mic}_{0,0}(\Bl_{\pt}\bbP^2, 4\ell+\alpha s)$}
    \label{fig:classification_4l_1}
\end{figure}

\begin{figure}
    \centering
    \begin{tabular}{cccc}
    \Tfourthreealpha & \Tfourthreealphaonetwo & \Tfourtwoalpha & \Tfourtwoalphatwo \\
    $\GG_{3,\alpha+4}$  & $\GG_{3,\alpha_1,\alpha_2}$  & $\GG_{2,\alpha+4}$  & $\GG_{2,\alpha+4,2}$
    \end{tabular}
    \caption{A collection of tentative components in $\Gamma^{mic}_{0,0}(\Bl_{\pt}\bbP^2, 4\ell+\alpha s)$}
    \label{fig:classification_4l_2}
\end{figure}

\begin{proof}
    We follow the strategy of the proof of \Cref{thm:classification_3l}.
    First, we show that the decorated graphs listed in the statement are precisely the tentative components $\GG \in \Gamma_{0,0}^{mic}(\Bl_{\pt}\bbP^2,4\ell+\alpha s)$. 
    If $\GG \neq \GG_0$, then $\GG$ must have at least one negative vertex $v$, with $d_v\leq -2$ because $\GG$ is sign-alternating. This implies that $\beta_v\in \{2e, 3e, 4e\}$. 
    
    If $\beta_v = 4e$ the only possibilities are the graph $\GG_{4,\alpha+4}$, $\GG_{4,\alpha_1,\alpha_2}$ or $\GG_{4,\alpha_1,\alpha_2,\alpha_3}$, with the conditions in the statement, by \Cref{prop:maximal_negative_vertex}.
    
    If $\beta_v = 3e$, then there are no other negative vertices because $\GG$ is sign-alternating, thus the only possibilities are the graph $\GG_{3,\alpha+4}$ or the graphs $\GG_{3,\alpha_1,\alpha_2}$, with the conditions in the statement, by \Cref{prop:unique_negative_vertex_condition}.
    
    Finally, suppose that $\beta_v = 2e$. If $v$ is the unique negative vertex, the only option is $\GG_{2,\alpha+4}$ by \Cref{prop:unique_negative_vertex_condition}. On the contrary, if $v$ is not the unique negative vertex there is another negative vertex, of class $2e$. In this case, the two negative vertices cannot be adjacent since $\GG$ is sign-alternating. If $\lvert V(\GG)\rvert = 3$, then $\GG_{2,\alpha+4,2}$ is the only possibility. Otherwise, one can check that $\lvert V(\GG)\rvert \in \{4,5\}$ and all such possibilities have $d_{\GG}<0$.
    
    Using \Cref{prop:simplification_dG_tentative_components}, it is easy to check that $d_{\GG_{4,\alpha+4}} = 2$, that $d_{\GG_{4,\alpha_1,\alpha_2}} = d_{\GG_{3,\alpha+4}} = 1$ and all other graphs have $d_{\GG} = 0$.
    
    To conclude, we need to prove that all the tentative components previously listed actually give rise to components. We already know this is the case for $\GG_0$ by \Cref{item:main_Bl} in \Cref{thm:criteria_tentative_component}, and for the graphs with $\beta_v=4e$ by \Cref{prop:maximal_negative_vertex}. For the remaining graphs, the result follows from \Cref{item:sufficient_criterion} in \Cref{thm:criteria_tentative_component} and our previous computation of $d_{\GG}$.
\end{proof}

For the curve non-negative classes $\beta = as+be \in A_1^{mic}(\Bl_{\pt}\bbP^2)$ (grey cone in \Cref{fig:Mori_cone}) with $b\leq 4$, we have proven that all tentative components induce irreducible components. The smallest degree for which we do not know if this is the case is $\beta = 5\ell$, see the following example.

\begin{example}\label{ex:component_or_not}
    The criteria in \Cref{thm:criteria_tentative_component} are not enough to determine if the tentative components $\GG_1, \GG_2, \GG_3 \in \Gamma^{mic}_{0,0}(\Bl_{\pt}\bbP^2,5\ell)$ in \Cref{fig:three_unknown_components_5l} are actually irreducible components or not.
    
    Indeed, we cannot apply \Cref{item:sufficient_criterion} in \Cref{thm:criteria_tentative_component} because $\GG_i$ contracts to the tentative component $\GG$ in \Cref{fig:contraction_graph_5l} (which is an irreducible component by \Cref{item:sufficient_criterion} in \Cref{thm:criteria_tentative_component}), but $0 = d_{\GG_i} < d_{\GG} = 1$.
    
    On the other hand, we cannot apply \Cref{item:necessary_criterion} in \Cref{thm:criteria_tentative_component} because the only contractions of $\GG_i$ which are tentative components are $\GG$, with $4 = h^0(\GG_i) > h^0(\GG) = 3$, and $\GG_0$ with $4 = h^0(\GG_i) > h^0(\GG_0) = 1$. 
\end{example}

\newcommand{\Tfiveone}{
    \begin{tikzpicture}[scale=\scalefactorfour]
                \draw (0,0) node[] {\textbullet};
                \draw (0,0) node[anchor= west] {$3e$};
                \draw (0,-1) node[] {\textbullet};
                \draw (0,-1) node[anchor= west] {$4s$};
                \draw (0,-2) node[] {\textbullet};
                \draw (0,-2) node[anchor= west] {$2e$};
                \draw (0,-3) node[] {\textbullet};
                \draw (0,-3) node[anchor= west] {$s$};
                
                \draw (0,0) -- (0,-3);
    \end{tikzpicture}
}

\newcommand{\Tfivetwo}{
    \begin{tikzpicture}[scale=\scalefactorfour]
                \draw (0,0) node[] {\textbullet};
                \draw (0,0) node[anchor= west] {$3e$};
                \draw (0,-1) node[] {\textbullet};
                \draw (0,-1) node[anchor= west] {$3s$};
                \draw (0,-2) node[] {\textbullet};
                \draw (0,-2) node[anchor= west] {$2e$};
                \draw (0,-3) node[] {\textbullet};
                \draw (0,-3) node[anchor= west] {$2s$};
                
                \draw (0,0) -- (0,-3);
    \end{tikzpicture}
}

\newcommand{\Tfivethree}{
    \begin{tikzpicture}[scale=\scalefactorfour]
                \draw (0,0) node[] {\textbullet};
                \draw (0,0) node[anchor= west] {$3e$};
                \draw (0,-1) node[] {\textbullet};
                \draw (0,-1) node[anchor= west] {$2s$};
                \draw (0,-2) node[] {\textbullet};
                \draw (0,-2) node[anchor= west] {$2e$};
                \draw (0,-3) node[] {\textbullet};
                \draw (0,-3) node[anchor= west] {$3s$};
                
                \draw (0,0) -- (0,-3);
    \end{tikzpicture}
}
\newcommand{\Tfivecontraction}{
    \begin{tikzpicture}[scale=\scalefactorfour]
                \draw (0,0) node[] {\textbullet};
                \draw (0,0) node[anchor= south] {$3e$};
                \draw (0,-1) node[] {\textbullet};
                \draw (0,-1) node[anchor= north] {$5s+2e$};
                
                \draw (0,0) -- (0,-1);
    \end{tikzpicture}
}

\begin{figure}
    \centering
    \begin{tabular}{ccc}
    \Tfiveone & \Tfivetwo & \Tfivethree \\
    $\GG_1$  & $\GG_2$  & $\GG_3$  
    \end{tabular}
    \caption{Three tentative components in $\Gamma^{mic}_{0,0}(\Bl_{\pt}\bbP^2, 5\ell)$ for which we do not know if $\overline{\ccM}_{\GG_i}$ are irreducible components.}
    \label{fig:three_unknown_components_5l}
\end{figure}

\begin{figure}
    \centering
    \begin{tabular}{c}
    \Tfivecontraction \\
    $\GG$
    \end{tabular}
    \caption{A tentative components in $\Gamma^{mic}_{0,0}(\Bl_{\pt}\bbP^2, 5\ell)$ for which $\overline{\ccM}_{\GG}$ is an irreducible component.}
    \label{fig:contraction_graph_5l}
\end{figure}

Although the analogue of \Cref{thm:components_maps_collection} for quasimaps holds, as explained in \Cref{sec:components_qmaps}, the results in \Cref{sec:Bl_P2} do not necessarily hold for quasimaps. Concretely, \Cref{ex:qmaps_non_mic_component} is a counterexample to \Cref{thm:criteria_tentative_component}. 

\begin{example}\label{ex:qmaps_non_mic_component}
    We exhibit a non-zero non-negative mic curve class $\beta \in A_1(\Bl_{\pt}\bbP^2)$ and a decorated graph $\GG\in \Gamma^{st,\ccQ}_{0,n}(\Bl_{\pt}\bbP^2,\beta)$ for $n\geq 4$ such that $\overline{\ccQ}_{\GG}$ is an irreducible component of $\ccQ_{0,n}(\Bl_{\pt}\bbP^2,\beta)$ but $\GG$ is not a tentative component.
    
    Let $\beta = 3\ell$. We construct a quasimap $q\colon C \dashrightarrow \Bl_{\pt}\bbP^2$ with dual graph $\GG$ in \Cref{fig:qmaps_non_mic_component}. Note that $(\beta_{v_1}\cdot [D_i])_i = (1,3,3,-2)$ and $(\beta_{v_2}\cdot [D_i])_i = (2,0,0,2)$.
    Let $C_{v_1}\simeq C_{v_2}\simeq \bbP^1$, both with homogeneous coordinates $[x\colon y]$. Let $C$ be the gluing of $C_{v_1}$ and $C_{v_2}$ along $[1\colon 0]$. Take
    \[
        q_1\colon C_{v_1}\dashrightarrow \Bl_{\pt}\bbP^2 \colon [x\colon y] \mapsto [x\colon x^3\colon x^3\colon 0]
    \]
    and
    \[
        q_2\colon C_{v_2}\dashrightarrow \Bl_{\pt}\bbP^2 \colon [x\colon y] \mapsto [x^2\colon 1\colon 1\colon y^2]
    \]
    Then $q_1([1\colon 0]) = [1\colon 1 \colon 1 \colon 0] = q_2([1\colon 0])$, so $q_1$ and $q_2$ glue to a quasimap $q$, with the desired properties once we add $n\geq 4$ markings, at least two of them on each component. 
    This shows that $\ccQ_{\GG}\neq \emptyset$.
    
    To prove that $\overline{\ccQ}_{\GG}$ is an irreducible component, it suffices to show that $d_{\GG}\geq d_{\GG_0} = 0$ by the analogue of \Cref{thm:components_maps_collection} for quasimaps (which is justified in \Cref{sec:components_qmaps}). Since $d_{v_1} = -2$ and $d_{v_2} = 2$, we have that
    \[
        d_{\GG} = h^0(\GG) - h^0(\bbP^1,\ccO_{\bbP^1}) - \lvert E(\GG) \rvert = 2-1-1 = 0.
    \]
    Thus, $\overline{\ccQ}_{\GG}$ is an irreducible component even though $\GG$ is not mic (\Cref{def:mic_class}). In particular, $\overline{\ccM}_{\GG}$ is not an irreducible component of $\overline{\ccM}_{0,n}(\Bl_{\pt}\bbP^2,\beta)$.
\end{example}

\begin{figure}
    \centering
    \begin{tikzpicture}[scale=\scalefactorthree]
                \draw (0,0) node[] {\textbullet};
                \draw (0,0) node[anchor= south] {$\ell+2e$};
                \draw (0,0) node[anchor= north] {$v_1$};
                \draw (1,0) node[] {\textbullet};
                \draw (1,0) node[anchor= south] {$2s$};
                \draw (1,0) node[anchor= north] {$v_2$};
                
                \draw (0,0) -- (1,0);
    \end{tikzpicture}
    \caption{A decorated graph $\GG\in \Gamma_{0,0}(\Bl_{\pt}\bbP^2,3\ell)$ which lies in $\Gamma_{0,n}^{st,\ccQ}(\Bl_{\pt}\bbP^2,3\ell)$ for any $n\geq 4$ if at least 2 marks are added to each vertex.}
    \label{fig:qmaps_non_mic_component}
\end{figure}

\textbf{Disclosure on the use of Artificial Intelligence Systems:} No AI Systems were used during the development of this paper.

\textbf{Conflicts of interest:} none. 

\textbf{Financial support:}
Both authors gratefully acknowledge support from the Simons Center for Geometry and Physics, Stony Brook University, where this collaboration started.

A.C.R. was supported by Fonds Wetenschappelijk Onderzoek (FWO), through reference G0B3123N, and through a Junior Postdoctoral Fellowship with reference 1263127N.

E.M. benefits from the support of the French government “Investissements d’Avenir” program integrated
to France 2030, bearing the following reference PRC ANR-17-CE40-0014.

Both authors thank the university of Angers for supporting the research visit of A.C.R in May and June 2024. 

We thank also the IRP ``Derieng'' from the INSMI-CNRS.

\bibliographystyle{alpha}
\bibliography{references}

@article {Viscardi,
    AUTHOR = {Viscardi, Michael},
     TITLE = {Alternate compactifications of the moduli space of genus one
              maps},
   JOURNAL = {Manuscripta Math.},
  FJOURNAL = {Manuscripta Mathematica},
    VOLUME = {139},
      YEAR = {2012},
    NUMBER = {1-2},
     PAGES = {201--236},
      ISSN = {0025-2611,1432-1785},
   MRCLASS = {14D23 (14H10)},
  MRNUMBER = {2959678},
MRREVIEWER = {Jonathan\ Wise},
       DOI = {10.1007/s00229-011-0513-2},
       URL = {https://doi.org/10.1007/s00229-011-0513-2},
}

@Article{Cox,
 Author = {Cox, David A.},
 Title = {The functor of a smooth toric variety},
 FJournal = {T{\^o}hoku Mathematical Journal. Second Series},
 Journal = {T{\^o}hoku Math. J. (2)},
 ISSN = {0040-8735},
 Volume = {47},
 Number = {2},
 Pages = {251--262},
 Year = {1995},
 Language = {English},
 DOI = {10.2748/tmj/1178225594},
 zbMATH = {794316},
 Zbl = {0828.14035}
}

@article{Chang-Li-maps-with-fields,
issn = {1073-7928},
journal = {International mathematics research notices},
pages = {4163--4217},
volume = {2012},
number = {18},
year = {2012},
title = {{Gromov-Witten} invariants of stable maps with fields},
copyright = {Copyright 2012 Elsevier B.V., All rights reserved.},
language = {eng},
author = {Chang, Huai-Liang and Li, Jun},
}

@misc{Starr,
      title={The {K}odaira dimension of spaces of rational curves on low degree hypersurfaces}, 
      author={Jason Starr},
      year={2003},
    eprint={math/0305432},
      archivePrefix={arXiv},
      primaryClass={math.AG},
        note="arXiv preprint \href{https://arxiv.org/abs/math/0305432}{arXiv:math/0305432}"
}

@misc{stacks-project,
  author       = {The {Stacks project authors}},
  title        = {The Stacks project},
  howpublished = {\url{https://stacks.math.columbia.edu}},
  year         = {2022},
}

@book {Hartshorne_AG,
    AUTHOR = {Hartshorne, Robin},
     TITLE = {Algebraic geometry},
    SERIES = {Graduate Texts in Mathematics, No. 52},
 PUBLISHER = {Springer-Verlag, New York-Heidelberg},
      YEAR = {1977},
     PAGES = {xvi+496},
      ISBN = {0-387-90244-9},
   MRCLASS = {14-01},
  MRNUMBER = {0463157},
MRREVIEWER = {Robert Speiser},
}

@misc{Cobos,
      title={The contraction morphism between maps and quasimaps to toric varieties}, 
      author={Cobos Rabano, Alberto},
      year={2024},
      eprint={2412.16295},
      archivePrefix={arXiv},
      primaryClass={math.AG},
      url={https://arxiv.org/abs/2412.16295}, 
     URL = {https://arxiv.org/abs/2412.16295},
     arXiv = {arXiv:2412.16295},
    note="arXiv preprint \href{https://arxiv.org/abs/2412.16295}{arXiv:2412.16295}"
}

@book {CLS,
	AUTHOR = {Cox, David A. and Little, John B. and Schenck, Henry K.},
	TITLE = {Toric varieties},
	SERIES = {Graduate Studies in Mathematics},
	VOLUME = {124},
	PUBLISHER = {American Mathematical Society, Providence, RI},
	YEAR = {2011},
	PAGES = {xxiv+841},
	ISBN = {978-0-8218-4819-7},
	MRCLASS = {14M25 (05A15 05E45 52B12)},
	MRNUMBER = {2810322},
	MRREVIEWER = {Ivan Arzhantsev},
	DOI = {10.1090/gsm/124},
	URL = {https://doi.org/10.1090/gsm/124},
}

@article{Battistella_Carocci,
 author = {Battistella, Luca and Carocci, Francesca},
 title = {A geographical study of {{\(\overline{\mathcal{M}}_2\left(\mathbb{P}^2, 4\right)^{\text{main }} \)}}},
 fjournal = {Advances in Geometry},
 journal = {Adv. Geom.},
 issn = {1615-715X},
 volume = {22},
 number = {4},
 pages = {463--480},
 year = {2022},
 language = {English},
 doi = {10.1515/advgeom-2022-0017},
 zbMATH = {7604585},
 Zbl = {1497.14045}
}

@article{Battistella_Carocci_g2,
 author = {Battistella, Luca and Carocci, Francesca},
 title = {A smooth compactification of the space of genus two curves in projective space: via logarithmic geometry and {Gorenstein} curves},
 fjournal = {Geometry \& Topology},
 journal = {Geom. Topol.},
 issn = {1465-3060},
 volume = {27},
 number = {3},
 pages = {1203--1272},
 year = {2023},
 language = {English},
 doi = {10.2140/gt.2023.27.1203},
 zbMATH = {7699411},
 Zbl = {1521.14052}
}

@article {Hu_Li_g1,
    AUTHOR = {Hu, Yi and Li, Jun},
     TITLE = {Genus-one stable maps, local equations, and {V}akil-{Z}inger's
              desingularization},
   JOURNAL = {Math. Ann.},
  FJOURNAL = {Mathematische Annalen},
    VOLUME = {348},
      YEAR = {2010},
    NUMBER = {4},
     PAGES = {929--963},
      ISSN = {0025-5831},
   MRCLASS = {14D23 (14D20 14E15)},
  MRNUMBER = {2721647},
MRREVIEWER = {Cristina Mart\'{\i}nez},
       DOI = {10.1007/s00208-010-0504-8},
       URL = {https://doi.org/10.1007/s00208-010-0504-8},
}

@Article{Hu_Li_derived_resolution,
 Author = {Hu, Yi and Li, Jun},
 Title = {Derived resolution property for stacks, {Euler} classes and applications},
 FJournal = {Mathematical Research Letters},
 Journal = {Math. Res. Lett.},
 ISSN = {1073-2780},
 Volume = {18},
 Number = {4},
 Pages = {677--690},
 Year = {2011},
 Language = {English},
 DOI = {10.4310/MRL.2011.v18.n4.a7},
 zbMATH = {6194600},
 Zbl = {1342.14112}
}

@article {Hu_Li_Niu_g2,
    AUTHOR = {Hu, Yi and Li, Jun and Niu, Jingchen},
     TITLE = {Genus two stable maps, local equations and modular
              resolutions},
   JOURNAL = {Moduli},
  FJOURNAL = {Moduli},
    VOLUME = {3},
      YEAR = {2026},
     PAGES = {Paper No. e12, 128},
      ISSN = {2949-7647,2977-1382},
   MRCLASS = {14H10 (14D23 14E15)},
  MRNUMBER = {5119733},
       DOI = {10.1017/s2949764726100265},
       URL = {https://doi.org/10.1017/s2949764726100265},
}

@article{niu1,
 author = {Hu, Yi and Niu, Jingchen},
 title = {Moduli of curves of genus one with twisted fields},
 fjournal = {The Asian Journal of Mathematics},
 journal = {Asian J. Math.},
 issn = {1093-6106},
 volume = {25},
 number = {5},
 pages = {683--714},
 year = {2021},
 language = {English},
 doi = {10.4310/AJM.2021.v25.n5.a4},
 zbMATH = {7554802},
 Zbl = {1514.14013}
}

@article {niu2,
    AUTHOR = {Hu, Yi and Niu, Jingchen},
     TITLE = {A theory of stacks with twisted fields and resolution of
              moduli of genus two stable maps},
   JOURNAL = {Asian J. Math.},
  FJOURNAL = {Asian Journal of Mathematics},
    VOLUME = {29},
      YEAR = {2025},
    NUMBER = {3},
     PAGES = {325--422},
      ISSN = {1093-6106,1945-0036},
   MRCLASS = {14D23 (14E15)},
  MRNUMBER = {4955072},
MRREVIEWER = {Federico\ Scavia},
       DOI = {10.4310/ajm.250905113140},
       URL = {https://doi.org/10.4310/ajm.250905113140},
}

@article{Lee_Oh_part2,
 author = {Lee, Sanghyeon and Oh, Jeongseok},
 title = {Algebraic reduced genus one {Gromov}-{Witten} invariants for complete intersections in projective spaces. {II}},
 fjournal = {Journal of the London Mathematical Society. Second Series},
 journal = {J. Lond. Math. Soc., II. Ser.},
 issn = {0024-6107},
 volume = {106},
 number = {2},
 pages = {1319--1356},
 year = {2022},
 language = {English},
 doi = {10.1112/jlms.12600},
 zbMATH = {7730839},
 Zbl = {1521.14098}
}

@article{Lee_Oh_part1,
 author = {Lee, Sanghyeon and Oh, Jeongseok},
 title = {Algebraic reduced genus one {Gromov}-{Witten} invariants for complete intersections in projective spaces},
 fjournal = {IMRN. International Mathematics Research Notices},
 journal = {Int. Math. Res. Not.},
 issn = {1073-7928},
 volume = {2021},
 number = {23},
 pages = {18149--18180},
 year = {2021},
 language = {English},
 doi = {10.1093/imrn/rnz336},
 zbMATH = {7458168},
 Zbl = {1485.14105}
}

@misc{Lee_Li_Oh,
      title={Quantum Lefschetz property for genus two stable quasimap invariants}, 
      author={Sanghyeon Lee and Mu-Lin Li and Jeongseok Oh},
      year={2022},
      eprint={2204.08757},
      archivePrefix={arXiv},
      primaryClass={math.AG},
      url={https://arxiv.org/abs/2204.08757}, 
}

@book{Fulton,
	title={Intersection theory},
	author={Fulton, William},
	volume={2},
	year={2013},
	publisher={Springer Science \& Business Media}
}

@book{Eisenbud,
 author = {Eisenbud, David},
 title = {Commutative algebra. {With} a view toward algebraic geometry},
 fseries = {Graduate Texts in Mathematics},
 series = {Grad. Texts Math.},
 issn = {0072-5285},
 volume = {150},
 isbn = {3-540-94269-6; 3-540-94268-8},
 year = {1995},
 publisher = {Berlin: Springer-Verlag},
 language = {English},
 zbMATH = {704831},
 Zbl = {0819.13001}
}

@article {CRMMP,
    AUTHOR = {Cobos Rabano, Alberto and Mann, Etienne and Manolache,
              Cristina and Picciotto, Renata},
     TITLE = {Higher-genus reduced {G}romov-{W}itten invariants via
              desingularizations of sheaves},
   JOURNAL = {Geom. Topol.},
  FJOURNAL = {Geometry \& Topology},
    VOLUME = {30},
      YEAR = {2026},
    NUMBER = {7},
     PAGES = {2599--2684},
      ISSN = {1465-3060,1364-0380},
   MRCLASS = {14N35},
  MRNUMBER = {5118601},
       DOI = {10.2140/gt.2026.30.2599},
       URL = {https://doi.org/10.2140/gt.2026.30.2599},
}

@article{Zinger-genus-one-pseudo-holomorphic,
issn = {1364-0380},
journal = {Geometry \& topology},
pages = {2427--2522},
volume = {13},
publisher = {Cornell University Library, arXiv.org},
number = {5},
year = {2009},
title = {A sharp compactness theorem for genus-one pseudo-holomorphic maps},
copyright = {Copyright 2010 Elsevier B.V., All rights reserved.},
language = {eng},
address = {Ithaca},
author = {Zinger, Aleksey},
}

@article {Vakil-Zinger-desingu-main-compo-,
    AUTHOR = {Vakil, Ravi and Zinger, Aleksey},
     TITLE = {A desingularization of the main component of the moduli space
              of genus-one stable maps into {$\Bbb P^n$}},
   JOURNAL = {Geom. Topol.},
  FJOURNAL = {Geometry \& Topology},
    VOLUME = {12},
      YEAR = {2008},
    NUMBER = {1},
     PAGES = {1--95},
      ISSN = {1465-3060},
   MRCLASS = {14D20 (14N35)},
  MRNUMBER = {2377245},
MRREVIEWER = {Charles D. Cadman},
       DOI = {10.2140/gt.2008.12.1},
       URL = {https://doi.org/10.2140/gt.2008.12.1},
}

@article{ranganathan2019moduli,
  title={Moduli of stable maps in genus one and logarithmic geometry, I},
  author={Ranganathan, Dhruv and Santos-Parker, Keli and Wise, Jonathan},
  journal={Geometry \& Topology},
  volume={23},
  number={7},
  pages={3315--3366},
  year={2019},
  publisher={Mathematical Sciences Publishers}
}

@article{Zinger-standard-vs-reduced,
issn = {1465-3060},
journal = {Geometry \& topology},
pages = {1203--1241},
volume = {12},
number = {2},
year = {2008},
title = {Standard versus reduced genus-one {G}romov–{W}itten invariants},
language = {eng},
author = {Zinger, Aleksey},
}

@article{Zinger-reduced-g1-CY,
issn = {0894-0347},
journal = {Journal of the American Mathematical Society},
pages = {691--737},
volume = {22},
publisher = {American Mathematical Society},
number = {3},
year = {2009},
title = {The reduced genus 1 {G}romov-{W}itten invariants of {C}alabi-{Y}au hypersurfaces},
copyright = {2009 American Mathematical Society},
language = {eng},
author = {Zinger, Aleksey},
}

@article{li-zinger,
  title={Gromov--{W}itten invariants of a quintic threefold and a rigidity conjecture},
  author={Li, Jun and Zinger, Aleksey},
  journal={Pacific Journal of Mathematics},
  volume={233},
  number={2},
  pages={417--480},
  year={2007},
  publisher={Mathematical Sciences Publishers}
}

@Article{Zhou,
 Author = {Zhou, Yang},
 Title = {Quasimap wall-crossing for {GIT} quotients},
 FJournal = {Inventiones Mathematicae},
 Journal = {Invent. Math.},
 ISSN = {0020-9910},
 Volume = {227},
 Number = {2},
 Pages = {581--660},
 Year = {2022},
 Language = {English},
 DOI = {10.1007/s00222-021-01071-z},
 zbMATH = {7470588},
 Zbl = {1518.14067}
}

@article{stable_quotients,
	issn = {1465-3060},
	journal = {Geom. Topol.},
	pages = {1651--1706},
	volume = {15},
	publisher = {GEOMETRY and TOPOLOGY PUBLICATIONS},
	number = {3},
	year = {2011},
	title = {The moduli space of stable quotients},
	copyright = {Copyright 2011 Elsevier B.V., All rights reserved.},
	language = {eng},
	address = {COVENTRY},
	author = {Marian, Alina and Oprea, Dragos and Pandharipande, Rahul},
}

@Article{clader2017higher,
 Author = {Clader, Emily and Janda, Felix and Ruan, Yongbin},
 Title = {Higher-genus wall-crossing in the gauged linear sigma model},
 FJournal = {Duke Mathematical Journal},
 Journal = {Duke Math. J.},
 ISSN = {0012-7094},
 Volume = {170},
 Number = {4},
 Pages = {697--773},
 Year = {2021},
 Language = {English},
 DOI = {10.1215/00127094-2020-0053},
 zbMATH = {7369840},
 Zbl = {1481.14085}
}

@article{ciocan2017higher,
  title={Higher genus quasimap wall-crossing for semipositive targets},
  author={Ciocan-Fontanine, Ionu{\c{t}} and Kim, Bumsig},
  journal={J. Eur. Math. Soc. (JEMS)},
  volume={19},
  number={7},
  pages={2051--2102},
  year={2017}
}

@article{ciocan2020quasimap,
  title={Quasimap wall-crossings and mirror symmetry},
  author={Ciocan-Fontanine, Ionu{\c{t}} and Kim, Bumsig},
  journal={Publ. Math. Inst. Hautes {\'E}tudes Sci.},
  volume={131},
  number={1},
  pages={201--260},
  year={2020},
  publisher={Springer}
}

@article {KMMP,
    AUTHOR = {Kern, David and Mann, \'Etienne and Manolache, Cristina and
              Picciotto, Renata},
     TITLE = {Derived moduli of sections and push-forwards},
   JOURNAL = {Selecta Math. (N.S.)},
  FJOURNAL = {Selecta Mathematica. New Series},
    VOLUME = {31},
      YEAR = {2025},
    NUMBER = {2},
     PAGES = {Paper No. 40, 46},
      ISSN = {1022-1824,1420-9020},
   MRCLASS = {14N35 (18N99)},
  MRNUMBER = {4887845},
       DOI = {10.1007/s00029-025-01033-w},
       URL = {https://doi.org/10.1007/s00029-025-01033-w},
}

@article {Cristina-stable-maps,
    AUTHOR = {Manolache, Cristina},
     TITLE = {Stable maps and stable quotients},
   JOURNAL = {Compos. Math.},
  FJOURNAL = {Compositio Mathematica},
    VOLUME = {150},
      YEAR = {2014},
    NUMBER = {9},
     PAGES = {1457--1481},
      ISSN = {0010-437X,1570-5846},
   MRCLASS = {14N35 (14D23)},
  MRNUMBER = {3260138},
MRREVIEWER = {Qile\ Chen},
       DOI = {10.1112/S0010437X14007258},
       URL = {https://doi.org/10.1112/S0010437X14007258},
}

@article {Cristina-virt-push,
    AUTHOR = {Manolache, Cristina},
     TITLE = {Virtual push-forwards},
   JOURNAL = {Geom. Topol.},
  FJOURNAL = {Geometry \& Topology},
    VOLUME = {16},
      YEAR = {2012},
    NUMBER = {4},
     PAGES = {2003--2036},
      ISSN = {1465-3060,1364-0380},
   MRCLASS = {14C15 (14A20 14N35)},
  MRNUMBER = {2975298},
MRREVIEWER = {Ruifang\ Song},
       DOI = {10.2140/gt.2012.16.2003},
       URL = {https://doi.org/10.2140/gt.2012.16.2003},
}

@article{C-FK,
	issn = {0001-8708},
	journal = {Adv. Math.},
	pages = {3022--3051},
	volume = {225},
	publisher = {Elsevier Inc},
	number = {6},
	year = {2010},
	title = {Moduli stacks of stable toric quasimaps},
	copyright = {2010 Elsevier Inc.},
	language = {eng},
	address = {SAN DIEGO},
	author = {Ciocan-Fontanine, Ionuţ and Kim, Bumsig},
}

@book{Mumford,
series = {Tata Institute of Fundamental Research. Studies in mathematics, 5},
publisher = {Published for the Tata Institute of Fundamental Research, Bombay [by] Oxford University Press},
year = {1970},
title = {Abelian varieties.},
language = {eng},
address = {London]},
author = {Mumford, David},
}

@article{Axelsson-Magnusson,
     author = {Axelsson, Reynir and Magn\'usson, J\'on},
     title = {A closure operation on complex analytic cones and torsion},
     journal = {Annales de la Facult\'e des sciences de Toulouse : Math\'ematiques},
     pages = {5--33},
     publisher = {Universit\'e Paul Sabatier. Facult\'e des sciences},
     address = {Toulouse},
     volume = {Ser. 6, 7},
     number = {1},
     year = {1998},
     zbl = {0931.32005},
     mrnumber = {1658436},
     language = {en},
     url = {http://www.numdam.org/item/AFST_1998_6_7_1_5_0/}
}

@article{Vakil_thesis,
 author = {Vakil, Ravi},
 title = {The enumerative geometry of rational and elliptic curves in projective space},
 fjournal = {Journal f{\"u}r die Reine und Angewandte Mathematik},
 journal = {J. Reine Angew. Math.},
 issn = {0075-4102},
 volume = {529},
 pages = {101--153},
 year = {2000},
 language = {English},
 doi = {10.1515/crll.2000.094},
 zbMATH = {1544818},
 Zbl = {0970.14029}
}

@article{VZ,
 author = {Vakil, Ravi and Zinger, Aleksey},
 title = {A natural smooth compactification of the space of elliptic curves in projective space},
 fjournal = {Electronic Research Announcements of the American Mathematical Society},
 journal = {Electron. Res. Announc. Am. Math. Soc.},
 issn = {1079-6762},
 volume = {13},
 pages = {53--59},
 year = {2007},
 language = {English},
 doi = {10.1090/S1079-6762-07-00174-6},
 zbMATH = {5246226},
 Zbl = {1131.14028}
}

\end{document}